\documentclass[12pts]{article}
\usepackage{amsmath,amssymb, amsthm, amsopn,
  amsfonts}

\usepackage{graphicx}

\usepackage[all]{xy}
\usepackage{color}

\newcommand{\field}[1]{\mathbb{#1}} \newcommand{\rz}{\field{R}}
\newcommand{\R}{\field{R}}
\newcommand{\cz}{\field{C}} \newcommand{\nz}{\field{N}}

\newcommand{\E}{\mathbb{E}}

\newtheorem{theoreme}{Theorem}[section]
\newtheorem{proposition}[theoreme] {Proposition}
\newtheorem{lemme}[theoreme]{Lemma}
\newtheorem{corollaire}[theoreme]{Corollary}
\newtheorem{definition}[theoreme]{Definition}
\newtheorem{remarque}[theoreme]{Remark}

\newtheorem{hypothese}{Hypothesis}

\def\XXint#1#2#3{{\setbox0=\hbox{$#1{#2#3}{\int}$}
     \vcenter{\hbox{$#2#3$}}\kern-.5\wd0}}

\numberwithin{equation}{section}

\DeclareMathOperator\supp{supp}

\DeclareMathOperator\Ran{Ran} 
 
 \DeclareMathOperator\Id{Id}
\DeclareMathOperator\Ker{Ker}

  \def\12{\frac{1}{2}}

\def\Hess{{\rm Hess}\,}

\title{Low temperature asymptotics for Quasi-Stationary
  Distributions in a bounded domain}
\author{T.~Leli\`evre\thanks{Universit\'e Paris-Est, CERMICS,
    Project-team  INRIA MicMac, 6 \& 8 avenue Blaise Pascal, 77455
    Marne-la-Vall\'ee, France.  \textit{email: lelievre@cermics.enpc.fr} }\; ,
  F.~Nier\thanks{IRMAR, UMR-CNRS 6625,  Universit{\'e} de Rennes 1, Campus
    de Beaulieu, 35042 Rennes, France.
 \textit{email: francis.nier@univ-rennes1.fr}}\; , }
\begin{document}
\bibliographystyle{plain}
\maketitle 
\begin{abstract}
We analyze the low temperature asymptotics of the quasi-stationary distribution associated with the
overdamped Langevin dynamics (a.k.a.  the Einstein-Smoluchowski
diffusion equation)~in a bounded domain. This analysis is
useful to  rigorously prove the consistency of an algorithm used in
molecular dynamics (the hyperdynamics), in the small temperature
regime. More precisely, we show that the algorithm is exact in terms
of state-to-state dynamics up to exponentially small factor in the
limit of small temperature. The proof is based on the asymptotic spectral analysis of associated Dirichlet and
Neumann realizations of Witten Laplacians. In order to cover a reasonably large range
of applications, the usual assumptions that the energy
landscape is a Morse function has been relaxed as much as possible.

\end{abstract}


\section{Introduction}
\label{se.intro}
The motivation of this work comes from the mathematical analysis of an algorithm used in
molecular dynamics, called the hyperdynamics~\cite{voter-97}. The aim
of this algorithm is to generate very efficiently the discrete state-to-state
dynamics associated with a continuous state space metastable Markovian dynamics,
by modifying the potential function. In Section~\ref{sec.MD}, we
explain the principle of the algorithm, and state the mathematical
problem. In Section~\ref{sec.res}, the main result of this article is
given in a simple setting.

\subsection{Molecular dynamics, hyperdynamics and the quasi stationary
distribution}\label{sec.MD}
Molecular dynamics calculations consist in simulating very long
trajectories of a particle model of matter, in order to infer
macroscopic properties from an atomic description. Examples include
the study of the change of conformation of large molecules (such as
proteins) with applications in biology, or the description of the
motion of defects in materials.

In a constant temperature environment, the dynamics used in practice
contains stochastic terms which model thermostating. The
 prototypical example, which is the focus of this work, is the overdamped Langevin dynamics: 
\begin{equation}\label{eq:OL}
dX_t = -\nabla f(X_t) \, dt + \sqrt{2 \beta^{-1}} dB_t,
\end{equation}
where $X_t \in \R^{3N}$ is the position vector of $N$ particles, $f:
\R^{3N} \to \R$ is the potential function (assumed to be smooth in the
following), and $\beta^{-1}=k_B T$ with
$k_B$ the Boltzmann constant and $T$ the temperature. The stochastic
process $B_t$ is a standard $3N$-dimensional Brownian motion. The
dynamics~\eqref{eq:OL} admits the canonical ensemble $\mu(dx)=Z^{-1}
\exp(-\beta f(x)) \, dx$ as an invariant probability measure.

To relate the macroscopic properties of matter to the microscopic
phenomenon, one simulates the process $(X_t)_{t \ge 0}$ (or processes
following related dynamics, like the Langevin dynamics) over very long
times. 
The difficulty associated with such simulations
is {\em metastability}, namely the fact that the stochastic process
remains trapped for very long times in some regions of the
configurational space called the metastable states. The timestep used
to obtain stable discretization is typically $10^{-15} s$, while the
macroscopic timescales of interest is of the order a few microseconds,
up to seconds. At the macroscopic level, the details of the dynamics
$(X_t)_{t \ge 0}$ do not matter. The important information is the
history of the visited metastable states, the so-called {\em
  state-to-state dynamics}.

The principle of the hyperdynamics
algorithm~\cite{voter-97} is to modify the potential $f$ in order to
accelerate the exit from metastable states, while keeping a correct
state-to-state dynamics. In the following, we focus on one elementary
brick of this dynamics, namely the exit event from a given metastable state.

In mathematical terms, the problem is the following (we refer
to~\cite{le-bris-lelievre-luskin-perez-12} for the mathematical proofs
of the statements below).  
Assuming that the
process remained trapped for a very long time in a domain $\Omega_+ \subset
\R^{3N}$ ($\Omega_+$ is a metastable state\footnote{We use the
  notation $\Omega_+$ since in the following, we will need a subdomain
  $\Omega_{-}$ such that $\overline{\Omega_-} \subset \Omega_+$.}, as mentioned above), 
it is known that the process reaches a local equilibrium
called the quasi stationary distribution (QSD) $\nu$ attached to the domain
$\Omega_+$, before leaving it. We assume in the following that $\Omega_+$ is a smooth bounded
domain in $\R^{3N}$. The probability distribution~$\nu$ has support $\Omega_+$, and is such
that, for all smooth test function $\varphi: \R^{3N} \to \R$,
\begin{equation}\label{eq:QSDlimit}
\lim_{t \to \infty} \E(\varphi (X_t) | \tau > t) = \int_{\Omega_+} \varphi \,
d \nu
\end{equation}
where $$\tau=\inf \{t > 0, \, X_t \not\in \Omega_+\},$$
is the first exit time from $\Omega_+$ for $X_t$. The metastability of the
well $\Omega_+$ can be quantified through the rate of convergence of the
limit in~\eqref{eq:QSDlimit}: in the following, it is assumed that
this convergence is infinitely fast.
From a PDE viewpoint, $\nu$
has a density $v$ with respect to the Boltzmann Gibbs measure
$\mu(dx)=e^{-\beta f(x)}\,dx$, $v$ being the first eigenvector of the infinitesimal
generator of the dynamics~\eqref{eq:OL}, with Dirichlet boundary
conditions on $\partial \Omega_+$:
\begin{equation}\label{eq:lambda_u}
\left\{
\begin{aligned}
-\nabla f \cdot \nabla v + \beta^{-1} \Delta v &= - \lambda v \text{ in $\Omega_+$},\\
v&=0 \text{ on $\partial \Omega_+$},
\end{aligned}
\right.
\end{equation}
where $-\lambda <0$ is the first eigenvalue. In other words:
$$d\nu = \frac{\displaystyle 1_{\Omega_+} (x) v(x) \exp(-\beta f(x)) \, dx}{\displaystyle\int_{\Omega_+} v(x)  \exp(-\beta f(x)) \, dx}.$$
Starting from the QSD $\nu$ (namely if $X_0 \sim \nu$), the way the
stochastic process $X_t$ solution to~\eqref{eq:OL} leaves the
well $\Omega_+$ is known: the law of the couple of random variables $(\tau,X_{\tau})$
(exit time, exit point) is characterized by the following three properties: (i)
$\tau$ and $X_\tau$ are independent and (ii) $\tau$ is exponentially
distributed with parameter~$\lambda$:
\begin{equation}\label{eq:tau}
\tau \sim {\mathcal E}(\lambda),
\end{equation}
where the notation $\sim$ is used to indicate the law of a random variable. These two properties are the
building blocks of a Markovian transition starting from
$\Omega_+$. Moreover, (iii) the exit point
distribution has an analytic expression in terms of $v$: for all
smooth test functions $\varphi : \partial \Omega_+ \to \R$,
\begin{equation}\label{eq:Xtau}
\E^\nu (\varphi(X_\tau)) = - {\frac{\displaystyle\int_{\partial \Omega_+}
  \varphi \, \partial_n \left(v \exp(-\beta f) \right)
  d\sigma}{\displaystyle\beta \lambda \int_{\Omega_+} v(x)  \exp(-\beta f(x))
  \, dx} } 
\end{equation}
where, for any smooth function $w: \Omega_+ \to \R$, $\partial_ n w=
\nabla w \cdot n$ denotes the outward normal derivative, $\sigma$ is the Lebesgue measure on $\partial \Omega_+$ and $\E^\nu$ indicates the expectation for the stochastic process
$X_t$ following~\eqref{eq:OL} and starting under the QSD: $X_0 \sim \nu$.

In practical cases of interest, the typical exit time is very
large ($\E(\tau)=1/\lambda$ is very large). The principle of the hyperdynamics is to modify the potential
$f$ in the state $\Omega_+$ to lead to smaller exit times, {\em while keeping a correct statistics
on the exit points}. Let us make this more precise, and let us consider
the process $X^{\delta f}_t$ which evolves on a new potential $f + \delta f$:
\begin{equation}\label{eq:OL2}
dX^{\delta f}_t = -\nabla (f+\delta f) (X^{\delta f}_t) \, dt + \sqrt{2 \beta^{-1}} dB_t.
\end{equation}
Instead of simulating $(X_t)_{t \ge0}$ following the
dynamics~\eqref{eq:OL} and considering the associated random variables
$(\tau,X_\tau)$, the hyperdynamics algorithm consists in
simulating $(X^{\delta f}_t)_{t \ge 0}$ and considering the associated
random variables $(\tau^{\delta f},X^{\delta f}_{\tau^{\delta
  f}})$, where $\tau^{\delta f}$ is the first exit time from $\Omega_+$ for
$X^{\delta f}_t$.

The assertion underlying the hyperdynamics algorithm is the following:
under appropriate assumption on the perturbation $\delta f$, then (i) the exit point distribution of $X^{\delta f}_t$ from $\Omega_+$ is
(almost) the same as the exit point distribution of $X_t$ from $\Omega_+$
and (ii) the exit time distribution for $X_t$ can be inferred from the
exit time distribution for $X^{\delta f}_t$ by a simple multiplicative
factor (see Equations~\eqref{eq:law_exit}--\eqref{eq:boost_factor} below).

More precisely, the assumptions on $\delta f$ in the original
paper~\cite{voter-97} can be stated as follows:
(i) $\delta f$ is sufficiently small so that $\Omega_+$ is still a metastable
state for $X^{\delta f}_t$ and (ii) $\delta f$ is zero on the boundary of~$\Omega_+$. The first hypothesis implies that we can assume that $X^{\delta f}_0$ is
distributed according to the QSD $\nu^{\delta f}$ associated
with~\eqref{eq:OL2} and
$\Omega_+$. The aim of this paper is to prove that, in the small
temperature regime (namely $\beta \to \infty$) and under appropriate
assumptions on $\delta f$, we indeed have the following
equality in law:
\begin{equation}\label{eq:law_exit}
(\tau,X_\tau) \stackrel{\mathcal L}{\simeq} ( B \tau^{\delta f}, X^{\delta f}_{\tau^{\delta
  f}}),
\end{equation}
where, in the left-hand side, $X_0 \sim \nu$ and, in the
right-hand side, $X^{\delta f}_0 \sim \nu^{\delta f}$. 
The so-called
boost factor $B$ has the following expression:
\begin{equation}\label{eq:boost_factor}
B
=
\frac{\int_{\Omega_+} \exp(-\beta f)}{\int_{\Omega_+} \exp(-\beta(f+\delta
  f))}=\int_{\Omega_+} \exp(\beta \delta f) \frac{\exp(-\beta(f+\delta
  f)}{\int_{\Omega_+} \exp(-\beta(f+\delta f))}.
\end{equation}
The second formula is interesting because it shows that $B$ can be
approximated through ergodic averages on the process $(X^{\delta f}_t)_{t
  \ge 0}$ (and this is actually exactly what is done in practice).

In view of the formulas~\eqref{eq:tau}--\eqref{eq:Xtau} stated above for the laws of the 
distributions of the two random variables exit time and
exit point, a crucial point for the mathematical analysis of the hyperdynamics algorithm
is to study how the first eigenvalue $\lambda$ and the
normal derivative $\partial_n v$ ($v$ being the first eigenvector, see~\eqref{eq:lambda_u}) are modified when changing the
potential $f$ to $f + \delta f$. More precisely, we would like to
check that, in the limit $\beta \to \infty$, $\lambda^{\delta f}=B
\lambda$ and, up to multiplicative constant, $\partial_n
  v^{\delta f} \propto \partial_n
  v$, where, with obvious notation, $(-\lambda^{\delta f},v^{\delta
    f})$ denotes the first eigenvalue eigenvector pair solution
  to~\eqref{eq:lambda_u} when $f$ is replaced by $f + \delta f$.

\subsection{The main results in a simple setting}
\label{sec.res}

Let us state the main results obtained in this paper in a simple and restricted
setting. We assume the following on the potential $f$. There exists a
subdomain $\Omega_-$ such that $\overline{\Omega_-} \subset \Omega_+$ and:
\begin{description}
\item[(i)]  $f$ and $f|_{\partial
  \Omega_+}$ are Morse
functions, namely $\mathcal{C}^{\infty}$ functions with non-degenerate critical points;
\item[(ii)] $|\nabla f| \neq 0$ in $\overline{\Omega_+} \setminus \Omega_-$,
$\partial_n f >0$ on $\partial \Omega_-$ and $\min_{\partial \Omega_+}
f \ge \min_{\partial \Omega_-} f$\,;
\item[(iii)]   the critical values of $f$ in $\Omega_{-}$
are all distinct and the differences $f(U^{(1)})-f(U^{(0)})$, where
$U^{(0)}$ ranges over the local minima of $f\big|_{\Omega_{-}}$ and
$U^{(1)}$ ranges over the critical points of $f\big|_{\Omega_{-}}$
with index $1$\,, are all distinct;
\item[(iv)] the maximal value of $f$ at critical points, denoted by $\mathrm{cvmax}=\max \{ f(x), \, x \in \Omega_+ \, \text{ s.t. } |\nabla
f(x)|=0 \}=\max \{ f(x), \, x \in \Omega_{-} \, \text{ s.t. } |\nabla
f(x)|=0 \}$, satisfies
\begin{equation}
\label{eq.condcvmax}
 \min_{\partial \Omega_-} f - \mathrm{cvmax} > \mathrm{cvmax} - \min_{\Omega_-} f.
\end{equation}
\end{description}
Concerning the perturbation $\delta f$, let us assume that $f + \delta
f$ satisfies the same four above hypotheses as~$f$, and that, in addition:
$$\delta f= 0 \text{ on } \Omega_+ \setminus \Omega_-.$$

Under these assumptions on $f$ and $\delta f$, it can be shown that the first
eigenvalue/eigenfunction pairs $(-\lambda,v)$ and $(-\lambda^{\delta
  f},v^{\delta f})$ respectively solutions
  to~\eqref{eq:lambda_u} with the potential $f$ and $f + \delta f$
  satisfy the following estimate: for some positive constant $c$, in
  the limit $\beta \to \infty$,
$$\frac{\lambda^{\delta f}}{\lambda} = B ( 1 + \mathcal{O}(e^{-\beta c}))$$
where, we recall, $B$ is defined by~\eqref{eq:boost_factor} and
$$ \frac{\partial_{n} v \big|_{\partial \Omega_+}}{ 
\|\partial_{n} v \|_{L^{1}(\partial \Omega_+)}}
= 
\frac{\partial_{n} v^{\delta f} \big|_{\partial \Omega_+}}{ 
\|\partial_{n} v^{\delta f}\|_{L^{1}(\partial
 \Omega_+)}} +
\mathcal{O}(e^{-\beta c})\quad \text{in}~L^{1}(\partial \Omega_+).$$
These results are simple consequences of the general Theorem~\ref{th.main}
below (see Corollary~\ref{co.variaf}) together with Proposition~\ref{pr.assMorse} and
Remark~\ref{re.assMorse}.

\subsection{Outline of the article}

The main result of this article, namely Theorem~\ref{th.main}, gives general
 asymptotic formulas for the first eigenvalue $\lambda$ and the
normal derivative $\partial_n v$ in the limit of small
temperature. This Theorem will be proven under assumptions involving
the low lying spectra 
of Witten Laplacians on
$\Omega_{-}$ and on $\Omega_+ \setminus \overline{\Omega_-}$.
These
assumptions hold for potentials satisfying the four
conditions (i)--(iv) stated above, but they are also valid in much more general
cases. In particular, we have in mind (i) assumptions stated only in
terms of $\Omega_{+}$ (see Remark~\ref{se.assmor+}), or (ii)
potentials not fulfilling the Morse assumption (see Section~\ref{se.nonMorse}).



The outline of the article is the following. In
Section~\ref{se.assresult}, we specify our general assumptions  and
state the two main theorems: Theorem~\ref{th.main} and Theorem~\ref{th.main2}. In Section~\ref{se.expdecay}, exponential decay
  estimates for the eigenvectors in terms of Agmon distances are reviewed. In Section~\ref{se.quasimode}, approximate eigenvectors for the
  Dirichlet Witten Laplacians on $\Omega_{+}$ are constructed in terms
  of eigenvectors for the
  Neumann Witten Laplacians on
  $\Omega_{-}$ and eigenvectors for the
  Dirichlet Witten Laplacians on the shell
  $\Omega_{+}\setminus\overline{\Omega_{-}}$. Following the strategy of
  \cite{HKN,HeNi,Lep1,Lep3,Lep4,LNV}, 
  accurate approximation of singular values of the Witten
  differential $d_{f,h}$ 
  are computed using matrix arguments in
  Section~\ref{se.matrix}. Theorem~\ref{th.main} and Theorem~\ref{th.main2} are finally
  proved in Section~\ref{se.proof}. The general assumptions used to
  prove the theorems are then thoroughly discussed and illustrated
  with various examples in Section~\ref{se.ass}. Our approach relies on the introduction of boundary Witten
Laplacians (namely Witten Laplacians with Dirichlet or Neumann boundary conditions) and requires notions and notation of Riemannian differential
geometry. A short presentation of these notions is given in the
Appendix~\ref{se.formulas}. 
 
\section{Assumptions and precise statements of the main results}
\label{se.assresult}

In order to prove the result, we first need to restate the
eigenvalue problem~\eqref{eq:lambda_u} with the standard notation used
in the framework of Witten Laplacians, which will be the central tool
in the following. It is easy to check that $(\lambda,v)$
satisfies~\eqref{eq:lambda_u} if and only if $(\lambda_1,u_1)$
satisfies
$$\Delta_{f,h}^{D,(0)}(\Omega_+) u_1 = \lambda_1 u_1$$
with
$$h=\frac{2}{\beta}, \qquad  \lambda_1=\frac{4}{\beta}\lambda=2h\lambda, 
\qquad u_1=\exp\left(-\frac{\beta f}{2} \right) v=\exp \left(-\frac{f}{h}\right)v,$$
and where $\Delta_{f,h}^{D, (0)}(\Omega_+)$  is
the  Witten Laplacian on zero-forms on $\Omega_+\subset\rz^{d}$\,, $d=3N$\,,
with homogeneous Dirichlet boundary conditions on $\partial \Omega_+$
(see Equation~\eqref{eq:Wlapl} below for more general formulas on $p$-forms):
\begin{equation}\label{eq:lapl_0_form}
\Delta_{f,h}^{D,(0)}(\Omega_+)) u_1 =  (-h\nabla +\nabla f) \cdot \left( (h\nabla +\nabla
f)u_1 \right)= -h^{2}\Delta u_1 + \left( |\nabla
  f|^{2}-h\Delta f \right) u_1.
\end{equation}
Notice that the operator $\Delta_{f,h}^{D,(0)}(\Omega_+)$ is a {\em positive}
symmetric operator.
We recall that $\Omega_+$ is the
metastable domain of interest, and $\Omega_-$ is a subdomain of
$\Omega_+$, where the potential $f$ is
modified in the hyperdynamics algorithm. In the following, we will
thus study how the first
eigenvalue $\lambda_1$ and eigenfunction $u_1$ of the Witten Laplacian
$\Delta_{f,h}^{D,(0)}(\Omega_+)$ depend on $f|_{\Omega_-}$. We will state the results in
a very general setting, namely for open regular bounded connected subsets $\Omega_{-}$ and $\Omega_{+}$ of a
$d$-dimensional Riemannian
manifold $(M,g)$ such that $\overline{\Omega_{-}}\subset
\Omega_{+}$.

The first assumption we make on $f$ is the following:
\begin{hypothese}
\label{hyp.0}  
The function $f:M \to \R$ is a ${\mathcal C}^\infty$ function satisfying
\begin{equation}
\label{eq.Hyp0}
|\nabla f|>0
\text{ on $\overline{\Omega_{+}}\setminus \Omega_{-}$}\,, \qquad
\partial_{n}f  >0 \text{ on   $\partial\Omega_{-}$}\qquad \text{
  and   }\qquad \min_{\partial\Omega_{+}}f\geq \min_{\partial \Omega_{-}}f\,.
\end{equation}
\end{hypothese}
\noindent
In~\eqref{eq.Hyp0}, $n$ denotes the unit normal vector on $\partial
\Omega_-$ which points outward from
$\Omega_-$. This first assumption has simple consequences that will be
used many times in the following.
\begin{lemme}
Under Hypothesis~\ref{hyp.0}, for all
$x\in \overline{\Omega_{+}}\setminus\Omega_{-}$,
$$f(x)\geq \min_{\partial
  \Omega_{-}} f >\min_{\Omega_{-}}f=\min_{\Omega_{+}}f.$$
\end{lemme}
\begin{proof}
The last equality is a simple consequence of the fact that the critical
points are in $\Omega_-$ and the inequality
$\min_{\partial\Omega_{+}}f\geq \min_{\partial \Omega_{-}}f$. Let us
now consider the first inequality. Let us
denote $\gamma_x(t)$ the gradient
trajectory $\dot \gamma_{x}=-\nabla f(\gamma_{x})$ starting from $x\in
\Omega_{+}$  ($\gamma_{x}(0)=x$). Let us consider $x \in
\overline{\Omega_{+}}\setminus\Omega_{-}$ such that $f(x) <
\min_{\partial \Omega_+} f$. Since $t \mapsto f(\gamma_x(t))$ is non
increasing, $(\gamma_x(t))_{t \ge 0}$ remains in the bounded domain $\Omega_+$ and is
thus well defined for all positive times. Moreover, necessarily, the
distance of $\gamma_x(t)$ to the set of critical points of $f$ tends
to $0$ as $t \to \infty$. This implies that there exists $t_0>0$ such
that $\gamma_x(t_0) \in \Omega_-$, and thus, $f(x)=f(\gamma_x(0)) \ge
f(\gamma_{x}(t_0)) \ge \min_{\partial \Omega_-} f$. This concludes the
proof of the first inequality. The second inequality is a
consequence of the assumption $\partial_{n}f  >0$ on
$\partial\Omega_{-}$, and is proven by considering the trajectory
$(\gamma_x(t))_{t \ge 0}$ with $x \in \arg\min_{\partial \Omega_-} f$.
\end{proof}
\begin{remarque}
One can easily check, using the same arguments, that $\partial_{n}f>0$ on $\partial
  \Omega_{+}$ together with the two first conditions of Hypothesis~\ref{hyp.0}
  implies $\min_{\partial \Omega_{+}}f > \min_{\partial
    \Omega_{-}}f$\,.
\end{remarque}
The second assumption on $f$ is:
\begin{hypothese}
\label{hyp.3}
The set of critical points of $f$ in $\Omega_{+}$ is included in $\left\{f\leq \min_{\partial\Omega_{+}}f-c_{0}\right\}$:
$$
\left\{x\in \Omega_{+}\,, \nabla f(x)=0\right\}\subset \left\{x\in
  \Omega_{+},\, f(x)< \min_{\partial\Omega_{+}}f-c_{0}\right\}\,.
$$
\end{hypothese}
In addition to Hypotheses~\ref{hyp.0} and~\ref{hyp.3}, our main
results are stated under assumptions on the spectrum of the Witten
Laplacians associated with $f$, on $\Omega_{-}$ and
$\Omega_{+}\setminus\overline{\Omega_{-}}$ (see Hypotheses~\ref{hyp.1}
and~\ref{hyp.2} below). We will discuss more explicit assumptions on $f$ for which those additional hypotheses
are satisfied in Section~\ref{se.ass}. Let us first
define the Witten Laplacians.  We refer
the reader to~\cite{Wit,HeSj4,CFKS,Bur,Zha}
for introductory texts to the semiclassical analysis of Witten Laplacians and its famous application
to Morse inequalities and related results. 

The Witten Laplacians are defined on $\bigwedge {\cal
  C}^{\infty}(M)=\oplus_{p=0}^{d}\bigwedge^{p} {\cal
  C}^{\infty}(M)$  as
\begin{equation}\label{eq:Wlapl}
\left\lbrace
\begin{aligned}
&\Delta_{f,h}=(d_{f,h}^{*}+d_{f,h})^{2}=d_{f,h}^{*}d_{f,h}+d_{f,h}d_{f,h}^{*}\\
&\text{where }  d_{f,h}=e^{-\frac{f}{h}}(hd)e^{\frac{f}{h}} \text{ and }
d_{f,h}^{*}=e^{\frac{f}{h}}(hd^{*})e^{-\frac{f}{h}}\,.
\end{aligned}
\right.
\end{equation}
On a domain $\Omega\subset M$ and for $m\in \nz$, the Sobolev space
$\bigwedge W^{m,2}(\Omega)$ is defined as the set of  $u\in
\Lambda L^{2}(\Omega)$ such that locally
$\partial_{x}^{\alpha}u\in \bigwedge L^{2}(\Omega)$ for all
$\alpha\in\nz^{d}$\,, $|\alpha|\leq m$  (this property does not depend
on the local coordinate system $(x^{1},\ldots, x^{d})$). 
When $\Omega$ is a regular bounded domain, $\bigwedge W^{m,2}(\Omega)$
coincides with the set of $u\in
\bigwedge L^{2}$ such that there exists $\tilde{u}\in \bigwedge
W^{m,2}(M)$ such that $\tilde{u}\big|_{\Omega}=u$. The spaces $\bigwedge
W^{s,2}(\Omega)$ for $s \in \R$ are then defined by duality and
interpolation. For $m=1$, the quantity
$\sqrt{\|u\|_{L^2(\Omega)}^2+\|du\|_{L^2(\Omega)}^2+\|d^*u\|_{L^2(\Omega)}^2}$
  is equivalent to the $W^{1,2}(\Omega)$-norm. This is a well-known
  result when $\Omega=\R^d$. The extension to a regular bounded domain
  is proved by using local charts and the reflexion principle, see~\cite{Tay,ChPi}.

In a regular bounded domain $\Omega$ of $M$\,, various self-adjoint
realizations of $\Delta_{f,h}$ can be considered:
\begin{itemize}
\item The Dirichlet realization $\Delta_{f,h}^{D}(\Omega)$ with domain
$$
D(\Delta_{f,h}^{D}(\Omega))=\left\{\omega\in \bigwedge W^{2,2}(\Omega)\,,
\quad
\mathbf{t}\omega\big|_{\partial\Omega}=0\,, \quad
\mathbf{t}d_{f,h}^{*}\omega\big|_{\partial\Omega}=0
\right\}\,.
$$
It is the Friedrichs extension of the closed quadratic form
\begin{equation}
  \label{eq.form}
  {\cal D}(\omega,\omega')=\langle d_{f,h}\omega\,,\,
d_{f,h}\omega'\rangle_{L^{2}}+\langle d_{f,h}^{*}\omega\,,\, d_{f,h}^{*}\omega'\rangle_{L^{2}}
\end{equation}
defined on the domain
$$
\bigwedge W^{1,2}_{D}(\Omega)=\left\{\omega\in \bigwedge W^{1,2}(\Omega)\,,\quad 
\mathbf{t}\omega\big|_{\partial\Omega}=0\right\}\,.
$$
Its restriction to zero-forms (functions) is simply the
operator~\eqref{eq:lapl_0_form} on $\Omega$ with homogeneous Dirichlet boundary
conditions. It is associated with the stochastic process~\eqref{eq:OL} killed at the boundary.
\item The Neumann realization $\Delta_{f,h}^{N}(\Omega)$ with domain
$$
D(\Delta_{f,h}^{N})(\Omega)=
\left\{\omega\in \bigwedge W^{2,2}(\Omega)\,,
\quad
\mathbf{n}\omega\big|_{\partial\Omega}=0\,, \quad
\mathbf{n}d_{f,h}\omega\big|_{\partial\Omega}=0
\right\}\,.
$$
It is the Friedrichs extension of the closed quadratic
form~\eqref{eq.form} defined on the domain
$$
\bigwedge W^{1,2}_{N}(\Omega)=\left\{\omega\in \bigwedge W^{1,2}(\Omega)\,,\quad 
\mathbf{n}\omega\big|_{\partial\Omega}=0\right\}\,.
$$
Its restriction to zero-forms (functions) is simply the
operator~\eqref{eq:lapl_0_form} on $\Omega$ with homogeneous Neumann boundary
conditions. It is associated with the stochastic process~\eqref{eq:OL} reflected at the boundary.
\end{itemize}

We will handle exponentially small quantities and we shall use the
following notation.
\begin{definition}
 \label{de.O} Let $(E,\|~\|)$ be a normed space. For two functions
 $a:(0,h_{0})\to E$  and $b:(0,h_{0})\to \rz_{+}$ we write 
 \begin{itemize}
 \item $a(h)=\mathcal{O}(b(h))$ 
if there exists a constant $C>0$ such that 
$
\forall h\in (0,h_{0}),\, \|a(h)\|\leq Cb(h);
$
\item $a(h)=\tilde{\mathcal{O}}(b(h))$ if for every $\varepsilon>0$,
  $a(h)=\mathcal{O}(b(h)e^{\frac{\varepsilon}{h}})$, or equivalently:
$$
\forall \varepsilon >0,\; \exists C_{\varepsilon}\geq 1,\; \forall
h\in (0,h_{0}),\, \|a(h)\|\leq C_{\varepsilon}b(h)e^{\frac{\varepsilon}{h}}\,.
$$
 \end{itemize}
\end{definition}
Notice that $a(h)=\tilde{\mathcal{O}}(b(h))$ is equivalent to
$\limsup_{h \to 0} h \log(\|a(h)\|/b(h)) \le 0$: up to an exponentially
small multiplicative term, $\|a(h)\|$ is smaller than $b(h)$.
Note in particular the identity
$\mathcal{O}(e^{-\frac{c_{1}}{h}})\tilde{\mathcal{O}}(e^{-\frac{c_{2}}{h}})
=
\tilde{\mathcal{O}}(e^{-\frac{c_{1}+c_{2}}{h}}) =
\mathcal{O}(e^{-\frac{c'}{h}})$
for any fixed $c'<c_{1}+c_{2}$ independently of $h\in (0,h_{0})$.

We are now in position to state the two additional Hypotheses on $f$,
which are stated as assumptions on the eigenvalues of Witten
Laplacians on $\Omega_{-}$ and $\Omega_{+}\setminus\overline{\Omega_{-}}$.
We assume that there exist a constant
$c_{0}>0$ and a function 
$\nu:(0,h_{0})\to (0,+\infty)$ with
\begin{equation}
  \label{eq.hypnu}
\begin{aligned}
&\forall \varepsilon>0\,,\exists C_{\varepsilon}>1,\quad
\frac{1}{C_{\varepsilon}}e^{-\frac{\varepsilon}{h}}\leq \nu(h)\leq
h,\\
\text{or equivalently }&\log\left(\frac{\nu(h)}{h}\right)\leq 0 \text{
  and }
\lim_{h\to 0}h\log\left(\nu(h)
\right)=0,
\end{aligned}
\end{equation}
and  such that the following hypotheses are fulfilled:
\begin{hypothese}
\label{hyp.1}  
The Neumann Witten Laplacian 
defined on $\Omega_{-}$ and restricted to forms of degree~$0$ and~$1$, 
$\Delta_{f,h}^{N,(p)}(\Omega_{-})$, $p=0,1$,
satisfies
\begin{align}
  \label{eq.nombN}
& \#\left[\sigma(\Delta_{f,h}^{N,(p)}(\Omega_{-}))
\cap [0,\nu(h)]\right] =: m_{p}^{N}(\Omega_{-})\,,\\
&
\label{eq.tailleN}
\sigma(\Delta_{f,h}^{N,(p)}(\Omega_{-}))
\cap [0,\nu(h)] \subset 
\left[0,e^{-\frac{c_{0}}{h}}\right]\,,
\end{align}
with $m_{p}^{N}(\Omega_{-})$ independent of $h \in (0,h_0)$. Here and
in the following, eigenvalues are counted with multiplicity, and the
symbol $\#$ denotes the cardinal of a finite ensemble.

In addition, there exists in $\overline{\Omega_{-}}$ 
 an open neighborhood $\mathcal{V}_{-}$ of $\partial\Omega_{-}$ 
such that any eigenfunction $\psi(h)$  of $\Delta_{f,h}^{N,(0)}(\Omega_{-})$
associated with a small nonzero eigenvalue $\mu(h)$ (namely $0<\mu(h)\leq \nu(h)$)  satisfies
\begin{equation}
  \label{eq.decayhyp}
  \|\psi(h)\|_{L^{2}(\mathcal{V}_{-})}=\tilde{\mathcal{O}}\left(\sqrt{\mu(h)}
    \right).
\end{equation}
\end{hypothese}
\begin{hypothese}
  \label{hyp.2}
The Dirichlet Witten Laplacian on $\Omega_{+}\setminus
  \overline{\Omega_{-}}$ restricted to one-forms satisfies
  \begin{align}
\label{eq.nombD}
    &\sharp\left[\sigma(\Delta_{f,h}^{D,(1)}(\Omega_{+}\setminus\overline{\Omega_{-}}))
\cap
[0,\nu(h)]\right]=:m_{1}^{D}(\Omega_{+}\setminus\overline{\Omega_{-}})\,,\\
\label{eq.tailleD}
& 
\sigma(\Delta_{f,h}^{D,(1)}(\Omega_{+}\setminus\overline{\Omega_{-}}))
\cap
[0,\nu(h)]\subset \left[0,e^{-\frac{c_{0}}{h}}\right]\,,
  \end{align}
with $m_{1}^{D}(\Omega_{+}\setminus\overline{\Omega_{-}})$ independent
of $h\in (0,h_{0})$.
\end{hypothese}
Our main results concern the smallest eigenvalue as well as
properties of the associated eigenfunction of $\Delta_{f,h}^{D,(0)}(\Omega_{+})$. 
\begin{theoreme}
\label{th.main}
Assume Hypotheses~\ref{hyp.0}, \ref{hyp.3},~\ref{hyp.1}, \ref{hyp.2} and $h\in (0,h_{0})$ with 
$h_{0}>0$ small enough. The eigenvalues contained in $[0,\nu(h)]$ of the
Dirichlet Witten Laplacians $\Delta_{f,h}^{D,(p)}(\Omega_{+})$, for
$p=0,1$, satisfy:
\begin{align*}
& m_{0}^{D}(\Omega_{+}):=\sharp\left[\sigma(\Delta_{f,h}^{D,(0)}(\Omega_{+}))\cap[0,\nu(h)]\right]
=m_{0}^{N}(\Omega_{-})\,,\\
& m_{1}^{D}(\Omega_{+}):=\sharp\left[\sigma(\Delta_{f,h}^{D,(1)}(\Omega_{+}))\cap[0,\nu(h)]\right]
=m_{1}^{N}(\Omega_{-})+m_{1}^{D}(\Omega_{+}\setminus\overline{\Omega_{-}})\,,\\
&\sigma(\Delta_{f,h}^{D,(p)}(\Omega_{+}))\cap[0,\nu(h)]\subset [0,e^{-\frac{c}{h}}]\,.
\end{align*}
Let $(u_{k}^{(1)})_{1\leq k\leq
  m_{1}^{D}(\Omega_{+}\setminus\overline{\Omega_{-}})}$ be an
orthonormal basis of the spectral subspace
$\Ran
1_{[0,\nu(h)]}(\Delta_{f,h}^{D,(1)}(\Omega_{+}\setminus\overline{\Omega_{-}}))$
and set 
$$
\kappa_{f}=\min_{\partial \Omega_{+}}f -\min_{\Omega_{+}}f\,.
$$
The smallest eigenvalue of $\Delta_{f,h}^{D,(0)}(\Omega_{+})$
satisfies, in the limit $h \to 0$:
\begin{align}
&\lim_{h \to 0} h \log \lambda_{1}^{(0)}(\Omega_{+}) = -2
\kappa_f \label{eq.applambda10}
\\
  &\lambda_{1}^{(0)}(\Omega_{+})=
\frac{h^{2}\sum_{k=1}^{m_{1}^{D}(\Omega_{+}\setminus\overline{\Omega_{-}})}
\left|\int_{\partial\Omega_{+}} e^{-\frac{f}{h}}u_{k}^{(1)}(n)({\sigma})~d\sigma\right|^{2}}{
\int_{\Omega_{+}}
e^{-\frac{2f(x)}{h}}~dx} (1+\mathcal{O}(e^{-\frac{c}{h}}))   \label{eq:expansion_lambda10}
\end{align}
for some constant $c>0$ and
$u_{k}^{(1)}(n)({\sigma})=\mathbf{i}_{n}u_{k}^{(1)} (\sigma)$ with the
interior product notation~\eqref{eq:interior}.
Moreover, the non negative $L^2(\Omega_+)$-normalized eigenfunction $u_1^{(0)}$
satisfies
\begin{align}
\label{eq.appu10}
  \left\|u_{1}^{(0)}-\frac{e^{\frac{-f}{h}}}{\left(\int_{\Omega_{+}}e^{-\frac{2f(x)}{h}}~dx\right)^{1/2}}\right\|_{W^{2.2}(\Omega_{+})}
&=\mathcal{O}(e^{-\frac{c}{h}})\,,
\\
\label{eq.appdu10}
\forall n\in\nz, \,
\left\|d_{f,h}u_{1}^{(0)}+\sum_{k=1}^{m_{1}^{D}(\Omega_{+}\setminus\overline{\Omega_{-}})}
\frac{h\int_{\partial
    \Omega_{+}}e^{-\frac{f(\sigma)}{h}}u_{k}^{(1)}(n)({\sigma})~d\sigma}{\left(\int_{\Omega_{+}}e^{-\frac{2f(x)}{h}}~dx\right)^{1/2}
}u_{k}^{(1)}\right\|_{W^{n,2}(\mathcal{V})}&=\mathcal{O}(e^{-\frac{\kappa_{f}+c_{\mathcal{V}}}{h}}),
\end{align}
where $\mathcal{V}$ is any neighborhood of $\partial\Omega_{+}$ lying
in $\Omega_{+}\setminus\overline{\Omega_{-}}$ and $c_{\mathcal{V}}>0$
is a constant independent on~$n$ and~$h$. The symbols $d\sigma$ and $n(\sigma)$
respectively denote the infinitesimal volume on $\partial \Omega_{+}$
and the outward normal vector at $\sigma\in \partial
\Omega_{+}$\,.
\end{theoreme} 
\begin{remarque}
It would be interesting for practical applications to relax the
assumption $|\nabla f|>0$
on $\overline{\Omega_{+}}\setminus \Omega_{-}$ in
Hypothesis~\ref{hyp.0} in order to be
able to consider saddle points on $\partial \Omega_+$.
\end{remarque}
\begin{remarque}
While proving these results, we will actually show that necessarily
$m_{1}^{D}(\Omega_{+}\setminus    \overline{\Omega_{-}}) \neq
0$, see Remark~\ref{re.m1D} below.
\end{remarque}
\begin{remarque}
In spectral theory, it is natural to work with complex-valued
functions or complex-valued forms. In view of the probabilistic
interpretation of our results, the above result is stated, and actually most of the analysis of
this text is carried out, with real-valued functions or forms.
One exception is Section~\ref{se.restric} which requires
 functional calculus and resolvents for complex spectral parameters.
Notice that it is straightforward to write a complex-valued version of
the previous results, by replacing the real scalar
product by the hermitian scalar product. For example, in~\eqref{eq.appdu10}, it simply
consists in changing $\int_{\partial
    \Omega_{+}}e^{-\frac{f(\sigma)}{h}}u_{k}^{(1)}(n)(\sigma)~d\sigma$
to
 $\int_{\partial\Omega_{+}}e^{-\frac{f(\sigma)}{h}}\overline{u_{k}^{(1)}}(n)(\sigma)~d\sigma$.
\end{remarque}
Corollaries and variations of Theorem~\ref{th.main} are given in
Section~\ref{se.proof}. Among the consequences, one can prove the
following result when additionally $f\big|_{\partial \Omega_{+}}$ is a
Morse function and $\partial_n f > 0$ on $\partial \Omega_{+}$.
\begin{theoreme}
\label{th.main2}
Assume Hypotheses~\ref{hyp.0}, \ref{hyp.3}, \ref{hyp.1}, \ref{hyp.2}, $h\in (0,h_{0})$ with
 $h_{0}>0$ small enough. Assume moreover that
$f\big|_{\partial\Omega_{+}}$ is a Morse function  and $\partial_n f > 0$ on $\partial \Omega_{+}$. Then the first
eigenvalue $\lambda_{1}^{(0)}(\Omega_{+})$ of $\Delta_{f,h}^{D,(0)}(\Omega_{+})$
and the corresponding $L^2(\Omega_+)$-normalized non negative eigenfunction
$u_{1}^{(0)}$ satisfy
\begin{align}
\label{eq.l10th2}
\lambda_{1}^{(0)}(\Omega_{+})&=
\frac{\int_{\partial \Omega_{+}} 2\partial_{n}f(\sigma)
  e^{-2\frac{f(\sigma)}{h}}~d\sigma}{
\int_{\Omega_{+}}e^{-2\frac{f(x)}{h}}~dx}
(1+\mathcal{O}(h))\,,\\
\label{eq.denssortth2}
- \frac{\partial_{n}\left[e^{-\frac{f}{h}}u_{1}^{(0)}\right]\big|_{\partial \Omega_{+}}}{ 
\left\|\partial_{n}\left[e^{-\frac{f}{h}}u_{1}^{(0)}\right]\right\|_{L^{1}(\partial
  \Omega_{+})}}
&=
\frac{(2\partial_{n}f) e^{-\frac{2f}{h}}\big|_{\partial \Omega_{+}}
}{
\left\|
(2\partial_{n}f)
  e^{-\frac{2f}{h}}\right\|_{L^{1}(\partial \Omega_{+})}
}
 + \mathcal{O}(h)\quad\text{in}~L^{1}(\partial\Omega_{+})\,.
\end{align}
\end{theoreme}
The proof of Theorem~\ref{th.main} is given in
Proposition~\ref{pr.number}, Lemma~\ref{le.infsigD0},
Proposition~\ref{prop:u10} and
Proposition~\ref{pr.expralmostorth}. The proof of
Theorem~\ref{th.main2} is given in Section~\ref{se.resultmorse}.

\section{ A priori exponential decay and first consequences}
\label{se.expdecay}
By applying Agmon's type estimate (see  for example \cite{Hel,DiSj} for a
general introduction) for boundary Witten Laplacians, we give
here exponential decay estimates for the eigenvectors of
$\Delta_{f,h}^{N}(\Omega_{-})$,
$\Delta_{f,h}^{D}(\Omega_{+}\setminus\overline{\Omega_{-}})$ and
$\Delta_{f,h}^{D}(\Omega_{+})$\,.

\subsection{Agmon identity}
\label{se.agmon}
We shall use an identity for boundary Witten Laplacians, proved in
\cite{HeNi} in the Dirichlet case and in \cite{Lep3} in the Neumann
case.
\begin{lemme}
\label{le.Agmon}
Let $\Omega$ be a regular bounded domain of $(M,g)$ and let
$\Delta_{f,h}^{D}(\Omega)$ (resp. $\Delta_{f,h}^{N}(\Omega)$) be the Dirichlet
(resp. Neumann)
realization of $\Delta_{f,h}(\Omega)$. Let $\varphi$ be a real-valued Lipschitz function
on $\overline{\Omega}$. Then, for any real-valued $\omega\in D(\Delta_{f,h}^{D}(\Omega))$
(resp. $\omega\in D(\Delta_{f,h}^{N}(\Omega))$),
\begin{align*}
 \langle\omega\,,\,
  e^{\frac{2\varphi}{h}}\Delta_{f,h}^{D}(\Omega) \omega\rangle_{L^{2}(\Omega)}
&=
h^{2}\|de^{\frac{\varphi}{h}}\omega\|_{L^{2}(\Omega)}^{2}+
h^{2}\|d^{*}e^{\frac{\varphi}{h}}\omega\|_{L^{2}(\Omega)}^{2}
\\
&\quad 
+\langle (|\nabla f|^{2}-|\nabla \varphi|^{2}+h{\mathcal L}_{\nabla
  f}+h{\mathcal L}_{\nabla
  f}^{*})e^{\frac{\varphi}{h}}\omega\,,\,
e^{\frac{\varphi}{h}}\omega\rangle_{L^{2}(\Omega)}
\\
&\quad
-
h\int_{\partial \Omega}\langle \omega\,,\, \omega\rangle_{
  T^{*}_{\sigma}\Omega}\,
e^{\frac{2\varphi(\sigma)}{h}}\frac{\partial f}{\partial
  n}(\sigma)~d\sigma,\\
(\text{resp.})\quad
\langle \omega\,,\,
  e^{\frac{2\varphi}{h}}\Delta_{f,h}^{N}(\Omega)\omega\rangle_{L^{2}(\Omega)}
&=
h^{2}\|de^{\frac{\varphi}{h}}\omega\|_{L^{2}(\Omega)}^{2}+
h^{2}\|d^{*}e^{\frac{\varphi}{h}}\omega\|_{L^{2}(\Omega)}^{2}
\\
&\quad 
+\langle (|\nabla f|^{2}-|\nabla \varphi|^{2}+h{\mathcal L}_{\nabla
 f}+h{\mathcal L}_{\nabla
  f}^{*})e^{\frac{\varphi}{h}}\omega\,,\,
e^{\frac{\varphi}{h}}\omega\rangle_{L^{2}(\Omega)}
\\
&\quad
+
h\int_{\partial \Omega}\langle \omega\,,\, \omega\rangle_{
  T^{*}_{\sigma}\Omega} \,
e^{\frac{2\varphi(\sigma)}{h}}\frac{\partial f}{\partial
  n}(\sigma)~d\sigma.
\end{align*}
\end{lemme}
In the previous formulae, the notation ${\mathcal L}_X$ refers to the
Lie derivative, see~\eqref{eq:Lie}.
We shall use this lemma with specific functions $\varphi$ associated
with the metric $|\nabla f|^{2}g$\,. 
\begin{lemme}
  \label{le.dist}
Let $\Omega$ be an open subset of $M$\,, $f\in
\mathcal{C}^{\infty}(\overline{\Omega})$,  and let $d_{Ag}$ be the geodesic
pseudo-distance on $\overline{\Omega}$
 associated with the possibly degenerate metric $|\nabla f|^{2}g$. The
 function $(x,y)\mapsto d_{Ag}(x,y)$ is Lipschitz (and thus almost
 everywhere differentiable) and satisfies 
 \begin{align}
   \text{for all } y_{0}\in \overline{\Omega}\,,\text{for}~ a.e.~x\in
\Omega \,, \quad
|\nabla_{x}d_{Ag}(x,y_{0})|&\leq |\nabla f(x)|, \nonumber
\\
\text{for all } x,y\in\overline{\Omega},\quad
|f(x)-f(y)|&\leq d_{Ag}(x,y)\,. \label{eq:fx-fy}
\end{align}
The equality $d_{Ag}(x,y)=|f(x)-f(y)|$ occurs if there is an
integral curve of $\nabla f$ joining $x$ to $y$. Moreover, for any
$A\subset \overline{\Omega}$, the function $x \mapsto d_{Ag}(x,A)$
(where $d_{Ag}(x,A)=\inf_{a\in
  A}d_{Ag}(x,a)$) is Lipschitz and
satisfies
$$
\text{for}~a.e.~x\in \Omega, \, |\nabla_x d_{Ag}(x,A)|\leq |\nabla f(x)|.
$$
\end{lemme}
\begin{proof}
The Lipschitz property comes from the triangular inequality for
$d_{Ag}(x,y)$. It carries over to $d_{Ag}(x,A)$.
The comparison between $|f(x)-f(y)|$ and $d_{Ag}(x,y)$ comes from
$$
|f(x)-f(y)|=\left|\int_{0}^{1}\nabla f(\gamma(t)) \cdot \dot{\gamma}(t)~dt\right|
\leq 
\int_{0}^{1}|\nabla f(\gamma(t))||\dot{\gamma}(t)|~dt=|\gamma|_{Ag}
$$
for any ${\cal C}^{1}$-path $\gamma$ joining $x$ to $y$ and by denoting
$|\gamma|_{Ag}$ its length according to $d_{Ag}$\,. 
\end{proof}
\begin{remarque}
When $f$ is Morse function, a detailed discussion about the equality $d_{Ag}(x,y)=|f(x)-f(y)|$,
  which involves the notion of generalized integral curves of $\nabla
  f$, can be found in \cite{HeSj4}.
\end{remarque}
\subsection{Exponential decay for the eigenvectors of
  $\Delta_{f,h}^{N,(p)}(\Omega_{-})$ ($p=0,1$)}
\label{se.decayO-}
Notice that from Hypothesis~\ref{hyp.0}, there exists an open set
$U$ such that
\begin{equation}\label{eq:U}
\overline{U}\subset\Omega_{-} \text{ and }|\nabla f|\neq 0 \text{ in }\overline{\Omega_{-}}\setminus {U}.
\end{equation}
\begin{proposition}
\label{pr.decayO-}
Let $U$ be an open set satisfying~\eqref{eq:U} and let $d_{Ag}(x,U)$ be
the Agmon distance to~$U$ defined for $x\in \Omega_{-}$\,. There exists a
constant $C>0$ independent of $h \in [0,h_0]$ such that every normalized eigenvector
$\omega_{\lambda_{h}}$ of $\Delta_{f,h}^{N}(\Omega_{-})$ 
associated with  an eigenvalue $\lambda_{h}\in [0,\nu(h)]$ satisfies
\begin{align*}
&\|e^{\frac{d_{Ag}(\cdot,U)}{h}}\omega_{\lambda_{h}}\|_{L^{2}(\Omega_{-}\setminus
  U)}\leq \|e^{\frac{d_{Ag}(\cdot,U)}{h}}\omega_{\lambda_{h}}\|_{L^{2}(\Omega_{-})} \leq C,\\
&\|e^{\frac{d_{Ag}(\cdot,U)}{h}}\omega_{\lambda_{h}}\|_{W^{1,2}(\Omega_{-}\setminus
  U)} \leq
\|e^{\frac{d_{Ag}(\cdot,U)}{h}}\omega_{\lambda_{h}}\|_{W^{1,2}(\Omega_{-})} 
\leq \frac{C}{h^{\frac{1}{2}}}.
\end{align*}
\end{proposition}
\begin{proof}
The function $d_{Ag}(\cdot,U)$ vanishes in $\overline{U}$ and satisfies the
properties of Lemma~\ref{le.dist} with $(\Omega,A)=(\Omega_{-},\overline{U})$.
Let us now apply Lemma~\ref{le.Agmon} on $\Delta_{f,h}^{N}(\Omega_{-})$, with the function $\varphi=(1-\alpha
h)d_{Ag}(\cdot,U)$ (where $\alpha$ is a positive constant to be fixed
later on) and a normalized eigenvector $\omega$:
$\Delta_{f,h}^{N}(\Omega_{-})\omega=\lambda\,\omega$ where 
$\lambda\in [0,\nu(h)]$. With $\frac{\partial f}{\partial n}>0$
on $\partial \Omega_{-}$, $\nu(h)\leq h$ and $|\nabla \varphi|^2 \le
(1 - \alpha h) |\nabla f|^2$ (for $h< 1/\alpha$), we obtain
\begin{equation}\label{eq:ineq1}
0\geq
h^{2}\|de^{\frac{\varphi}{h}}\omega\|_{L^{2}(\Omega_{-})}^{2}+h^{2}\|d^{*}e^{\frac{\varphi}{h}}\omega\|_{L^{2}(\Omega_{-})}^{2}
+  h \left[\alpha\langle e^{\frac{\varphi}{h}}\omega\,,\, |\nabla
  f|^{2} e^{\frac{\varphi}{h}} \omega\rangle_{L^{2}(\Omega_{-})}
  -C_{f}\|e^{\frac{\varphi}{h}}\omega\|^2_{L^{2}(\Omega_{-})}\right]\,.
\end{equation}
Here, we have used the fact that for any vector field $X$,
$\mathcal{L}_{X}+\mathcal{L}_{X}^{*}$ is a differential operator of
order $0$ involving derivatives of $X$ and $g$, which are uniformly
bounded in $\overline{\Omega_{-}}$.

Using~\eqref{eq:U}, choose $\alpha$ such that
$
\alpha\min_{x\in
  \overline{\Omega_{-}}\setminus U}|\nabla f(x)|^{2}\geq 2C_{f}
$
and add $2C_{f}h \|e^{\frac{\varphi}{h}}\omega\|_{L^{2}(U)}^{2}$ on both sides of the
inequality~\eqref{eq:ineq1}. Using the fact that
$$
2C_{f}h\geq 2C_{f}h\|\omega\|_{L^{2}(U)}^{2}= 2C_{f}h
\|e^{\frac{\varphi}{h}}\omega\|_{L^{2}(U)}^{2}
$$
one obtains:
$$
2C_{f}h\geq 
h^{2}\|de^{\frac{\varphi}{h}}\omega\|_{L^{2}(\Omega_{-})}^{2}+h^{2}\|d^{*}e^{\frac{\varphi}{h}}\omega\|_{L^{2}(\Omega_{-})}^{2}
+C_{f}h\|e^{\frac{\varphi}{h}}\omega\|_{L^{2}(\Omega_{-})}^{2}.
$$
This implies
$\|e^{\frac{(1-\alpha
    h)d_{Ag}(\cdot,U)}{h}}\omega\|_{L^{2}(\Omega_{-})}^{2}\leq 2$ and
$\|(hd) e^{\frac{(1-\alpha h)d_{Ag}(\cdot,U)}{h}}\omega\|_{L^{2}(\Omega_{-})}^{2}+
\|(hd)^{*}e^{\frac{(1-\alpha
    h)d_{Ag}(\cdot,U)}{h}}\omega\|_{L^{2}(\Omega_{-})}^{2}\leq 2C_{f}h$. 
Since $d_{Ag}(.,U)$ is a Lipschitz (and thus also bounded) function on
$\overline{\Omega}_{-}$, this ends the proof.
\end{proof}
Here is a useful consequence.
\begin{proposition}
\label{pr.almostorhtO-}
Let $(\psi_{j}^{(0)})_{1\leq j\leq m_{0}^{N}(\Omega_{-})}$
(resp. $(\psi_{k}^{(1)})_{1\leq k\leq m_{1}^{N}(\Omega_{-})}$) be an
orthonormal basis of eigenvectors of
$\Delta_{f,h}^{N,(0)}(\Omega_{-})$
(resp. $\Delta_{f,h}^{N,(1)}(\Omega_{-})$) associated with the
eigenvalues lying in $[0,\nu(h)]$ (or owing to Hypothesis~\ref{hyp.1}
 in $[0,e^{-\frac{c_{0}}{h}}]$)\,. Let $\chi_{-}\in
\mathcal{C}^{\infty}_{0}(\Omega_{-})$ be a cut-off function such that
$0\leq \chi_{-}\leq 1$ and
$\chi_{-}\equiv 1$ on a neighborhood of $U$ (where, as above, $U \subset\Omega_{-}$ satisfies~\eqref{eq:U}). The 
functions $v_{j}^{(0)}=\chi_{-} \psi_{j}^{(0)}$\,, $1\leq j\leq
  m_{0}^{N}(\Omega_{-})$\,
 (resp. one-forms $v_{k}^{(1)}=\chi_{-}\psi_{k}^{(1)}$\,, $1\leq k\leq
   m_{1}^{N}(\Omega_{-})$) belong to  the domain 
$D(\Delta_{f,h}^{D,(0)}(\Omega_{+}))$ (resp. $ D(\Delta_{f,h}^{D,(1)}(\Omega_{+}))$)
 of the Dirichlet realization of $\Delta_{f,h}$ in $\Omega_{+}$ and
 they satisfy: for $h \in [0, h_0]$,
\begin{align*}
  &\sum_{j=1}^{m_{0}^{N}(\Omega_{-})}\|\psi_{j}^{(0)}-v_{j}^{(0)}\|_{W^{1,2}(\Omega_{-})}
+\sum_{k=1}^{m_{1}^{N}(\Omega_{-})}\|\psi_{k}^{(1)}-v_{k}^{(1)}\|_{W^{1,2}(\Omega_{-})}
= {\mathcal O}(e^{-\frac{c_{\chi_{-}}}{h}}),
\\
&
\left(\langle v_{j}^{(0)}\,,\,
  v_{j'}^{(0)}\rangle_{L^{2}(\Omega_{+})}\right)_{j,j'}=\Id_{m_{0}^{N}(\Omega_{-})}+\mathcal{O}(e^{-\frac{c_{\chi_{-}}}{h}})\,,\quad
\left(\langle v_{k}^{(1)}\,,\,
  v_{k'}^{(1)}\rangle_{L^{2}(\Omega_{+})}\right)_{k,k'}=\Id_{m_{1}^{N}(\Omega_{-})}+\mathcal{O}(e^{-\frac{c_{\chi_{-}}}{h}})\,,\\
&
\langle v_{j}^{(0)}\,,
 \Delta_{f,h}^{D,(0)}(\Omega_{+}) v_{j}^{(0)}\rangle_{L^{2}(\Omega_{+})}=\mathcal{O}(e^{-\frac{c_{\chi_{-}}}{h}})\,,\quad
\langle v_{k}^{(1)}\,,
\Delta_{f,h}^{D,(1)}(\Omega_{+}) v_{k}^{(1)}\rangle_{L^{2}(\Omega_{+})}=\mathcal{O}(e^{-\frac{c_{\chi_{-}}}{h}})\,,
\end{align*}
where the $\mathcal{O}(e^{-\frac{c_{\chi_{-}}}{h}})$ remainders can be
bounded from above by $C_{\chi_{-}}e^{-\frac{c_{\chi_{-}}}{h}}$
for some constants $C_{\chi_{-}},c_{\chi_{-}}>0$ independent of $h \in
[0, h_0]$. Here and in the following, $\Id_{m}$ denotes the identity
matrix of size $m \times m$.
\end{proposition}
\begin{proof}
Let $\psi$ be a $L^2(\Omega_-)$-normalized eigenvector of $\Delta_{f,h}^{N,(p)}(\Omega_{-})$,
$p=0,1$, associated with the eigenvalue
$\lambda=\mathcal{O}(e^{-\frac{c_{0}}{h}})$,
 and set $v=\chi_{-} \psi$. Since $\chi_{-}$ belongs to $\mathcal{C}^{\infty}_{0}(\Omega_{-})$ the
form $v=\chi_{-} \psi$ belongs to
$D(\Delta_{f,h}^{D,(p)}(\Omega_{+}))$.

The $W^{1,2}(\Omega_-)$-estimates as well as the result on the Gram matrices are
consequences of:
\begin{equation}
  \label{eq.expdecpsiO-}
\| \psi - v\|_{W^{1,2}(\Omega_{-})}=\|(1-\chi_{-})\psi\|_{W^{1,2}(\Omega_{-})}\leq \|\psi\|_{W^{1,2}(\Omega_{-}\setminus\left\{\chi_{-}=1\right\})}
\leq C'_{\chi_{-}}e^{-\frac{c_{\chi_{-}}'}{h}}
\end{equation}
for some
constants
$c'_{\chi_{-}}>0$ and $C'_{\chi_{-}}>0$. The estimate~\eqref{eq.expdecpsiO-} is
derived from Proposition~\ref{pr.decayO-}  by using the fact that there
exists $c>0$ such that, for all $x \in
\Omega_{-}\setminus\left\{\chi_{-}=1\right\}$, $d_{Ag}(x,U) \ge c$
(this is a consequence of~\eqref{eq:U}  and~\eqref{eq:fx-fy}).

For the last estimate of Proposition~\ref{pr.almostorhtO-}, we use the estimate on $\Delta_{f,h}^{D}$ in Lemma~\ref{le.Agmon}  with $\varphi=0$
and successively $\Omega=\Omega_{+}$, $\omega=v=\chi_{-} \psi$ and 
$\Omega=\Omega_{-}$, $\omega=\psi$. This yields
\begin{equation}\label{eq:1}
\begin{aligned}
  \langle \chi_{-} \psi\,,\,
 \Delta_{f,h}^{D}(\Omega_+)  \chi_{-}\psi\rangle_{L^{2}(\Omega_{+})}
&=
h^{2}\|d\chi_{-}\psi\|_{L^{2}(\Omega_{+})}^{2}+
h^{2}\|d^{*}\chi_{-}\psi\|_{L^{2}(\Omega_{+})}^{2}
\\
&\quad +\langle (|\nabla f|^{2}+h{\mathcal L}_{\nabla
  f}+h{\mathcal L}_{\nabla
  f}^{*})\chi_{-}\psi\,,\,
\chi_{-}\psi\rangle_{L^{2}(\Omega_{+})}+ 0
\end{aligned}
\end{equation}
and (since $\frac{\partial f}{\partial n} > 0$ on $\partial \Omega_-$)
\begin{equation}\label{eq:2}
\begin{aligned}
e^{-\frac{c_{0}}{h}}\geq
\lambda\geq
h^{2}\|d\psi\|_{L^{2}(\Omega_{-})}^{2}+
h^{2}\|d^{*}\psi\|_{L^{2}(\Omega_{-})}^{2}
+\langle (|\nabla f|^{2}+h{\mathcal L}_{\nabla
  f}+h{\mathcal L}_{\nabla
  f}^{*})\psi\,,\,
\psi\rangle_{L^{2}(\Omega_{-})}\,.
\end{aligned}
\end{equation}
By considering the difference between~\eqref{eq:1} and~\eqref{eq:2}, we
thus have:
\begin{align*}
  &\langle \chi_{-} \psi\,,\, \Delta_{f,h}^{D}(\Omega_+)   \chi_{-}\psi\rangle_{L^{2}(\Omega_{+})}
\\
&\le
e^{-\frac{c_{0}}{h}} + h^2
\left(\|d\chi_{-}\psi\|_{L^{2}(\Omega_{+})}^{2} -
  \|d\psi\|_{L^{2}(\Omega_{-})}^{2} \right) + h^2 \left(
  \|d^{*}\chi_{-}\psi\|_{L^{2}(\Omega_{+})}^{2} -
  \|d^{*}\psi\|_{L^{2}(\Omega_{-})}^{2} \right) \\
& + \left( \langle (|\nabla f|^{2}+h{\mathcal L}_{\nabla
  f}+h{\mathcal L}_{\nabla
  f}^{*})\chi_{-}\psi\,,\,
\chi_{-}\psi\rangle_{L^{2}(\Omega_{+})} - \langle (|\nabla f|^{2}+h{\mathcal L}_{\nabla
  f}+h{\mathcal L}_{\nabla
  f}^{*})\psi\,,\,
\psi\rangle_{L^{2}(\Omega_{-})} \right).
\end{align*}
The last three terms in the right-hand side are all of the order
$\mathcal{O}(e^{-\frac{c_{\chi_{-}}}{h}})$. Indeed, for the first term
(the two other terms are estimated in the same way):
\begin{align*}
\left|\|d\chi_{-}\psi\|^{2}_{L^{2}(\Omega_{+})}-
\|d\psi\|^{2}_{L^{2}(\Omega_{-})}\right|&=\left|\langle d(1-\chi_{-})\psi\,,\,
d(1+\chi_{-})\psi\rangle_{L^{2}(\Omega_{-})}\right|\\
&\leq C_{\chi_{-}}'' \|\psi\|^{2}_{W^{1,2}(\Omega_{-}\setminus\left\{\chi_{-}=1\right\})}
\leq C^{(3)}_{\chi_{-}}e^{-\frac{2c_{\chi_{-}}'}{h}},
\end{align*}
using again~\eqref{eq.expdecpsiO-}. This proves the last estimate.
\end{proof}
According to the terminology of \cite{Lep1}, the property on the  Gram matrices in
Proposition~\ref{pr.almostorhtO-} is equivalent to
the almost orthonormality of the family $(v_{j}^{(p)})_{1\leq j\leq
  m_{p}^{N}(\Omega_{-})}$, $p=0,1$ in $L^2(\Omega_+)$.
\begin{definition}
  \label{de.almost} A finite family of $h$-dependent vectors
  $(u_{k}^{h})_{1\leq k\leq N}$ in a  Hilbert space $\mathcal{H}$ is
   almost orthonormal if the Gram matrix satisfies
$$
\left(\langle u_{j}^{h}\,,\,
  u_{k}^{h}\rangle\right)_{1\leq j,k\leq N}=\Id_{N}+
\mathcal{O}(e^{-\frac{c}{h}})
$$
for some $c>0$ independent of $h$.
\end{definition}

We end this paragraph with some remarks on the spectrum of
$\Delta_{f,h}^{N,(0)}(\Omega_{-})$, that we denote in the following  (as usual in
increasing order and with multiplicity) 
$\left(\mu_k^{(0)}(\Omega_{-})\right)_{k \ge 1}$.
The first eigenvalue of
$\Delta_{f,h}^{N,(0)}(\Omega_{-})$ is $\mu_{1}^{(0)}(\Omega_{-})=0$ associated with the eigenvector 
$$
\psi_{1}^{(0)}=\frac{e^{-\frac{f}{h}}}{\left(\int_{\Omega_{-}}e^{-\frac{2f(x)}{h}}~dx\right)^{1/2}}.
$$
One can prove that the second eigenvalue $\mu_2^{(0)}(\Omega_{-})$ of $\Delta_{f,h}^{N,(0)}(\Omega_{-})$
is exponentially large as compared
to  $e^{-\frac{2\kappa_{f}}{h}}$, where we recall $\kappa_{f}=\min_{\partial
  \Omega_{+}}f-\min_{\Omega_{+}}f=\min_{\partial
  \Omega_{+}}f-\min_{\Omega_{-}}f$. 
\begin{proposition}
 \label{pr.gapO-}
Let $\mathrm{cvmax}$ be the maximum critical value of $f$
in $\Omega_{-}$:
$$
\mathrm{cvmax}=\max\left\{f(x), x\in \Omega_{-}, \nabla f(x)=0\right\}\,.
$$
Then the second eigenvalue $\mu_{2}^{(0)}(\Omega_-)$ of $\Delta_{f,h}^{N,(0)}(\Omega_{-})$
satisfies
$$
\liminf_{h\to 0}h\log \left( \mu_{2}^{(0)}(\Omega_-) \right) \geq
-2(\mathrm{cvmax}-\min_{\Omega_{-}}f) \geq -2 \kappa_f + 2 c_0,
$$
where $c_0$ denotes the positive constant used in Hypothesis~\ref{hyp.3}.
\end{proposition}
\begin{proof}
The second inequality $-2(\mathrm{cvmax}-\min_{\Omega_{-}}f) \geq -2 \kappa_f
+ 2 c_0$ is of course a consequence of Hypothesis~\ref{hyp.3}. To
prove the first inequality, let us reason by contradiction and  assume that there exists
$\varepsilon_{0}>0$ and a sequence $h_{n}$ 
such that $\lim_{n \to \infty} h_n = 0$ and
$$
\min \left\{
  \sigma(\Delta_{f,h_{n}}^{N,(0)}(\Omega_{-}))\setminus\left\{0\right\}
  \right\}\leq C
e^{-2\frac{\mathrm{cvmax}-\min_{\Omega_{-}}f+\varepsilon_{0}}{h_{n}}}\,.
$$
To simplify the notation, let us drop the subscript
$_{n}$ in $h_{n}$\,. 
Let $\psi_{2}^{(0)}$ be a normalized eigenfunction of $\Delta_{f,h}^{N,(0)}(\Omega_{-})$  associated with
$\mu_{2}^{(0)}(\Omega_-)>0$. It is orthogonal to $\psi_{1}^{(0)}$ in $L^2(\Omega_-)$ and it
satisfies: for any $\Omega\subset \Omega_{-}$,
$$
\|d_{f,h}\psi_{2}^{(0)}\|_{L^{2}(\Omega)}^{2}\leq 
\|d_{f,h}\psi_{2}^{(0)}\|_{L^{2}(\Omega_{-})}^{2}=\langle
\psi_{2}^{(0)}\,,\,
\Delta_{f,h}^{N,(0)}(\Omega_{-})\psi_{2}^{(0)}\rangle_{L^{2}(\Omega_{-})}
=\lambda_{2}^{(0)}\leq Ce^{-2\frac{\mathrm{cvmax}-\min_{\Omega_{-}}f+\varepsilon_{0}}{h}}\,.
$$
In particular, for $\Omega=\left\{x\in \Omega_{-}, f(x)<
  \mathrm{cvmax}+\frac{\varepsilon_{0}}{2}\right\}$,
this gives
\begin{align*}
\left\|d \left(e^{\frac{f-\min_{\Omega_{-}}f}{h}}\psi_{2}^{(0)}
  \right) \right\|_{L^{2}(\Omega)}^{2}& \leq 
h^{-2} \max_{x\in \Omega}
\left|e^{\frac{f(x)-\min_{\Omega_{-}}f}{h}}\right|^{2} \left\|d_{f,h}\psi_{2}^{(0)}
  \right\|_{L^{2}(\Omega)}^{2}\\
&\leq Ch^{-2}
e^{-2\frac{\mathrm{cvmax}-\min_{\Omega_{-}}f+\varepsilon_{0}}{h}}
\max_{x\in \Omega} \left|e^{\frac{f(x)-\min_{\Omega_{-}}f}{h}}\right|^{2}\leq C'e^{-\frac{\varepsilon_{0}}{h}}\,.
\end{align*}
Using the spectral gap estimate for the Neumann Laplacian in $\Omega$
(or equivalently the Poincar\'e-Wirtinger inequality on $\Omega$), there is a
constant $C_{h}$ (depending on $\psi_2^{(0)}$) such that
\begin{equation*}
\left\|\psi_{2}^{(0)}-C_{h}e^{-\frac{f-\min_{\Omega_{-}}f}{h}} \right\|_{L^{2}(\Omega)}=\mathcal{O}(e^{-\frac{\varepsilon_{0}}{2h}})\,.
\end{equation*}
Equivalently, there is a constant $C_h$ such that
\begin{equation}
  \label{eq.estimpsi21}
\left\|\psi_{2}^{(0)}-C_{h}  \psi_{1}^{(0)} \right\|_{L^{2}(\Omega)}=\mathcal{O}(e^{-\frac{\varepsilon_{0}}{2h}})\,.
\end{equation}
Besides, using Proposition~\ref{pr.decayO-} with $U=\left\{x\in \Omega_{-}, f(x)<
  \mathrm{cvmax}+\frac{\varepsilon_{0}}{4}\right\}\subset
\Omega\,,
$ and  a lower bound on $d_{Ag}(x,U)$ (see~\eqref{eq.expdecpsiO-} for
a similar reasoning),
one obtains
\begin{equation}
  \label{eq.estimpsi22}
\|\psi_{1}^{(0)}\|_{L^{2}(\Omega_{-}\setminus \Omega)}+\|\psi_{2}^{(0)}\|_{L^{2}(\Omega_{-}\setminus \Omega)}\leq
C_{\varepsilon_{0}}e^{-\frac{c_{\varepsilon_{0}}}{h}}\,.
\end{equation}
The two estimates \eqref{eq.estimpsi21} and \eqref{eq.estimpsi22}
contradict the orthogonality of  $\psi_{2}^{(0)}$
and $\psi_{1}^{(0)}$ in $L^2(\Omega_-)$, in the limit $h \to 0$
(actually $n \to \infty$).
\end{proof}

\subsection{Exponential decay for the eigenvectors of
  $\Delta_{f,h}^{D,(p)}(\Omega_{+}\setminus\overline{\Omega_{-}})$}
In this section, we check that
$\sigma(\Delta_{f,h}^{D,(0)}(\Omega_{+}\setminus\overline{\Omega_{-}}))\cap
[0,\nu(h)]=\emptyset$ and provide the same results as in the previous
section for the
eigenvectors of
$\Delta_{f,h}^{D,(1)}(\Omega_{+}\setminus\overline{\Omega_{-}})$. Let us start with an equivalent of Proposition~\ref{pr.decayO-}.
\begin{proposition}\label{prop:exp_decay}
Let $\mathcal{V}$ be a  subset of $\overline{\Omega_+} \setminus
\Omega_-$ such that $\partial\Omega_{+} \subset \mathcal{V}$
and let $d_{Ag}(x,\mathcal{V})$
be the Agmon distance to ${\mathcal V}$ defined for $x \in \Omega_+
\setminus \overline{\Omega_-}$. There exists a constant $C>0$
independent of $h \in [0,h_0]$ such that
every normalized eigenvector $\psi$ of
$\Delta_{f,h}^{D,(1)}(\Omega_{+}\setminus\overline{\Omega_{-}})$
associated with an eigenvalue $\lambda \in [0,\nu(h)]$ satisfies
$$
\|e^{\frac{d_{Ag}(.,
    \mathcal{V})}{h}}\psi\|_{W^{1,2}(\Omega_{+}\setminus\overline{\Omega_{-}})}\leq
\frac{C}{h}.$$
\end{proposition}
\begin{proof}
The proof follows ideas from~\cite{DiSj}. Using Lemma~\ref{le.Agmon}, the fact that $\lambda \leq
h$ and the assumption on the sign of the normal derivative of $f$
on $\partial \Omega_-$ stated in Hypothesis~\ref{hyp.0},  we have
\begin{equation}
\label{eq.estimbord}
\begin{aligned}
0
& \geq h^{2}\|de^{\frac{\varphi}{h}}\psi\|_{L^{2}(\Omega_{+}\setminus\overline{\Omega_{-}})}^{2}
+
h^{2}\|d^{*}e^{\frac{\varphi}{h}}\psi\|_{L^{2}(\Omega_{+}\setminus\overline{\Omega_{-}})}^{2}
+\langle (|\nabla f|^{2}-|\nabla
\varphi|^{2})e^{\frac{\varphi}{h}}\psi\,,\,e^{\frac{\varphi}{h}}\psi\rangle_{L^{2}(\Omega_{+}\setminus\overline{\Omega_{-}})}
\\
& \quad - hC_{f}\|e^{\frac{\varphi}{h}}\psi\|_{L^{2}(\Omega_{+}\setminus\overline{\Omega_{-}})}^{2}
-h\int_{\partial\Omega_{+}}
\langle \psi\,,\,\psi\rangle_{\bigwedge
  T_{\sigma}^{*}\Omega_{+}}e^{\frac{2\varphi
    (\sigma)}{h}}\frac{\partial f}{\partial n}(\sigma)~d\sigma\,.
\end{aligned}
\end{equation}
Using the trace theorem, there exists a constant $C_{\mathcal{V}}$ such that for any $\omega
\in \bigwedge W^{1,2}(\mathcal{V})$,
$$
\int_{\partial \Omega_{+}}\langle
\omega\,,\,\omega\rangle_{\bigwedge T_{\sigma}^{*}\Omega_{+}}~d\sigma
\leq C_{\mathcal{V}}\left[\|\omega\|_{L^{2}(\mathcal{V})}^{2}+
  \|\omega\|_{W^{1,2}(\mathcal{V})}
\|\omega\|_{L^{2}(\mathcal{V})}\right]\,.
$$
By applying this inequality to $\omega=e^{\frac{\varphi}{h}}\psi$ and
using 
$$\|\omega\|_{W^{1,2}(\mathcal{V})}^2 \le C_{\mathcal{V}}\left[ \|\omega\|_{L^2
  (\mathcal{V})}^2 + \|d \omega\|_{L^2
  (\mathcal{V})}^2 + \|d^*\omega \|_{L^2
  (\mathcal{V})}^2 \right],$$ the last term of \eqref{eq.estimbord} is
estimated by
\begin{align*}
\left|
h\int_{\partial\Omega_{+}}
\langle \psi\,,\,\psi\rangle_{\bigwedge
  T_{\sigma}^{*}\Omega_{+}}e^{\frac{2\varphi
    (\sigma)}{h}}\frac{\partial f}{\partial n}(\sigma)~d\sigma
\right|
&\leq
\frac{h^{2}}{2}\left[\|de^{\frac{\varphi}{h}}\psi\|_{L^{2}(\mathcal{V})}^{2}
+
\|d^{*}e^{\frac{\varphi}{h}}\psi\|_{L^{2}(\mathcal{V})}^{2}
\right]
+C_{f,\mathcal{V}}\|e^{\frac{\varphi}{h}}\psi\|_{L^{2}(\mathcal{V})}^{2}\\
&\leq
\frac{h^{2}}{2}\left[\|de^{\frac{\varphi}{h}}\psi\|_{L^{2}(\Omega_{+}\setminus\overline{\Omega_{-}})}^{2}
+
\|d^{*}e^{\frac{\varphi}{h}}\psi\|_{L^{2}(\Omega_{+}\setminus\overline{\Omega_{-}})}^{2}
\right]
+C_{f,\mathcal{V}}
\end{align*}
since $\varphi \equiv 0$ sur ${\mathcal V}$. 
Taking $\varphi=(1-\alpha h)d_{Ag}(x,\mathcal{V})$
in~\eqref{eq.estimbord} gives (using $|\nabla \varphi|^2 \le
(1 - \alpha h) |\nabla f|^2$ and the inequality
$\|e^{\frac{\varphi}{h}}\psi\|_{L^{2}(\Omega_{+}\setminus\overline{\Omega_{-}})}^{2}
=
\|e^{\frac{\varphi}{h}}\psi\|_{L^{2}({\mathcal V})}^{2}
+
\|e^{\frac{\varphi}{h}}\psi\|_{L^{2}(\Omega_{+}\setminus\overline{\Omega_{-}
  \cup {\mathcal V}})}^{2} \le C_{\mathcal V}' + \|e^{\frac{\varphi}{h}}\psi\|_{L^{2}(\Omega_{+}\setminus\overline{\Omega_{-}
  \cup {\mathcal V}})}^{2}$)
$$
C_{f,\mathcal{V}}' \geq
\frac{h^{2}}{2}\left[\|de^{\frac{\varphi}{h}}\psi\|_{L^{2}(\Omega_{+}\setminus\overline{\Omega_{-}})}^{2}+
\|d^{*}e^{\frac{\varphi}{h}}\psi\|_{L^{2}(\Omega_{+}\setminus\overline{\Omega_{-}})}^{2}\right]
+h \left(\alpha\min_{x \in \Omega_{+}\setminus\overline{\Omega_{-}\cup
  \mathcal{V}}}|\nabla
f(x)|^{2}-C_{f} \right)\|e^{\frac{\varphi}{h}}\psi\|_{L^{2}(\Omega_{+}\setminus\overline{\Omega_{-}\cup
  \mathcal{V}})}^{2}\,.
$$
By taking $\alpha$ large enough, this yields the exponential decay estimate:
$$
\|e^{\frac{d_{Ag}(., \mathcal{V})}{h}}\psi\|_{W^{1,2}(\Omega_{+}\setminus\overline{\Omega_{-}})}\leq \frac{C_{f,\mathcal{V}}''}{h}\,.
$$
\end{proof}
We are now in position to state the main result of this section, which
can be seen as an equivalent of Proposition~\ref{pr.decayO-} for 
$\Delta_{f,h}^{D,(p)}(\Omega_{+}\setminus\overline{\Omega_{-}})$.
\begin{proposition}
\label{pr.decayO+-}
{\em 1)}  There is a constant $c>0$ such that 
\begin{equation}\label{eq:specLaplD0}
\forall h\in (0,h_{0})\,,\quad
\sigma(\Delta_{f,h}^{D,(0)}(\Omega_{+}\setminus\overline{\Omega_{-}}))\cap [0,c]=\emptyset\,.
\end{equation}
{\em 2)}
  Let $(\psi_{k}^{(1)})_{m_{1}^{N}(\Omega_{-})+1\leq k\leq
  m_{1}^{N}(\Omega_{-})+m_{1}^{D}(\Omega_{+}\setminus\overline{\Omega_{-}})}$ 
be an orthonormal basis of eigenvectors of
$\Delta_{f,h}^{D,(1)}(\Omega_{+}\setminus\overline{\Omega}_{-})$
associated with the eigenvalues in $[0,\nu(h)]$, and
let $\chi_{+}\in \mathcal{C}^{\infty}(\overline{\Omega_{+}})$ be
such that $\chi_{+}\equiv 1$ in a neighborhood of $\partial
\Omega_{+}$ and $0$ in a neighborhood of $\overline{\Omega_{-}}$\,. For all $k\in
\left\{m_{1}^{N}(\Omega_{-})+1,\ldots, m_{1}^{N}(\Omega_{-})+
  m_{1}^{D}(\Omega_{+}\setminus\overline{\Omega_{-}})\right\}$, set
$v_{k}^{(1)}=\chi_{+}\psi_{k}^{(1)}$. Then, the one-forms
$v_{k}^{(1)}$ are close to $\psi_{k}^{(1)}$, for $k\in
\left\{m_{1}^{N}(\Omega_{-})+1,\ldots, m_{1}^{N}(\Omega_{-})+
  m_{1}^{D}(\Omega_{+}\setminus\overline{\Omega_{-}})\right\}$:
\begin{equation}\label{eq:exp_decay_psi1}
 \sum_{k=m_{1}^{N}(\Omega_{-})+1}^{m_{1}^{N}(\Omega_{-})+m_{1}^{D}(\Omega_{+}\setminus\overline{\Omega_{-}})} 
\|\psi_{k}^{(1)}-v_{k}^{(1)}\|_{W^{1,2}(\Omega_{+}\setminus\overline{\Omega_{-}})}
= \mathcal{O}(e^{-\frac{c_{\chi_{+}}}{h}}).
\end{equation}
They are almost orthonormal in $L^2(\Omega_+)$:
$$
\left(\langle v_{k}^{(1)}\,,\,
  v_{k'}^{(1)}\rangle_{L^2(\Omega_+)}\right)_{k,k'}=\Id_{m_{1}^{D}(\Omega_{+}\setminus
    \overline{\Omega_{-}})}+\mathcal{O}(e^{-\frac{c_{\chi_{+}}}{h}}).
$$
Moreover, they belong to $D(\Delta_{f,h}^{D,(1)}(\Omega_{+}))$ 
and they satisfy
\begin{align*}
&\langle v_{k}^{(1)}\,,\, \Delta_{f,h}^{D,(1)}(\Omega_{+})
v_{k}^{(1)}\rangle_{L^{2}(\Omega_{+})} = \mathcal{O}(e^{-\frac{c_{\chi_{+}}}{h}})\,,\\
& d_{f,h}^{*}v_{k}^{(1)}\equiv 0 \quad \text{in~}~\left\{\chi_{+}= 1\right\}.
\end{align*}
All the $\mathcal{O}(e^{-\frac{c_{\chi_{+}}}{h}})$ remainders can be
bounded from above by $C_{\chi_{+}}e^{-\frac{c_{\chi_{+}}}{h}}$
for some constants $C_{\chi_{+}},c_{\chi_{+}}>0$ independent of $h \in
[0, h_0]$. 
\end{proposition}
\begin{proof}
1) The lower bound on the spectrum of
$\Delta_{f,h}^{D,(0)}(\Omega_{+}\setminus\overline{\Omega_{-}})$
comes from Lemma~\ref{le.Agmon} used with $\varphi=0$, and Hypothesis~\ref{hyp.0}: for any function
$\omega\in D(\Delta_{f,h}^{D,(0)}(\Omega_{+}\setminus\overline{\Omega_{-}}))$,
\begin{align*}
 \langle \omega\,,\,
 \Delta_{f,h}^{D,(0)}(\Omega_{+}\setminus\overline{\Omega_{-}}) \omega\rangle_{L^{2}(\Omega_{+}\setminus\overline{\Omega_{-}})}
&=
h^{2}\|d\omega\|_{L^{2}(\Omega_{+}\setminus\overline{\Omega_{-}})}^{2}+
h^{2}\|d^{*}\omega\|_{L^{2}(\Omega_{+}\setminus\overline{\Omega_{-}})}^{2}
\\
&\quad +\langle (|\nabla f|^{2}+h{\mathcal L}_{\nabla
  f}+h{\mathcal L}_{\nabla
  f}^{*})\omega\,,\,
\omega\rangle_{L^{2}(\Omega_{+}\setminus\overline{\Omega_{-}})}
\geq C_{f}\|\omega\|_{L^{2}(\Omega_{+}\setminus\overline{\Omega_{-}})}^{2}\,.
\end{align*}
2)  Let us start by proving that $d_{f,h}^{*}v_{k}^{(1)}\equiv 0$ in
$\left\{\chi_{+}= 1\right\}$. Let $\psi$ be  an eigenvector  of
$\Delta_{f,h}^{D,(1)}(\Omega_{+}\setminus\overline{\Omega_{-}})$
associated with an eigenvalue $\lambda\in [0,\nu(h)]$. Then,
$d_{f,h}^{*}\psi$ belongs to
$D(\Delta_{f,h}^{D,(0)}(\Omega_{+}\setminus\overline{\Omega_{-}}))$ and
$$
\Delta_{f,h}^{D,(0)}(d_{f,h}^{*}\psi)=\lambda d_{f,h}^{*}\psi\,,
$$
according to \cite{HeNi} (see also~\eqref{eq:vp_p_p+1} below). 
Using now~\eqref{eq:specLaplD0} and $\lambda\leq\nu(h)\leq h$,  this
implies 
\begin{equation}\label{eq:d*_zero}
d_{f,h}^{*}\psi\equiv 0
\end{equation}
and thus $d_{f,h}^{*}v\equiv 0$ in $\left\{\chi_{+}\equiv
  1\right\}$ .


All the other estimates are proved like in
Proposition~\ref{pr.almostorhtO-} as consequences of the
 exponential decay estimate for the
eigenvector $\psi$, stated in Proposition~\ref{pr.decayO+-}, using a neighborhood
 $\mathcal{V} \subset \overline{\Omega_+} \setminus \Omega_-$ of $\partial\Omega_{+}$ such
that $\chi_{+}\equiv 1$ in a neighborhood of
$\overline{\mathcal{V}}$.

For example, for~\eqref{eq:exp_decay_psi1}, using $d_{Ag}(x,\mathcal{V})\geq 2c_{\chi_{+}}'>0$ for $x\in \supp
(1-\chi_{+})$, Proposition~\ref{pr.decayO+-} provides
\begin{equation}
  \label{eq.boundbord}
\|(1-\chi_{+})\psi\|_{W^{1,2}(\Omega_{+}\setminus\overline{\Omega_{-}})}
\leq C_{\chi_{+}}'e^{-\frac{c'_{\chi_{+}}}{h}}.
\end{equation} The proofs of the two other
estimates on
$\langle v_{k}^{(1)}\,,\,
  v_{k'}^{(1)}\rangle_{L^2(\Omega_+)}$ and $\langle v_{k}^{(1)}\,,\,
 \Delta_{f,h}^{D,(1)}(\Omega_{+})  v_{k}^{(1)}\rangle_{L^{2}(\Omega_{+})}$ follow
the same lines as in the proof of
Proposition~\ref{pr.almostorhtO-}.
\end{proof}

\subsection{Exponential decay for the eigenvectors of
  $\Delta_{f,h}^{D,(p)}(\Omega_{+})$, ($p=0,1$)}
\label{se.decayO+}

We will use the two operators $\Delta_{f,h}^{N}(\Omega_{-})$ and
$\Delta_{f,h}^{D}(\Omega_+\setminus\overline{\Omega_{-}})$ to analyze the spectrum of $\Delta_{f,h}^{D}(\Omega_{+})$.
\begin{definition}
 \label{de.dirsom}
On $\bigwedge L^{2}(\Omega_{+})=\bigwedge L^{2}(\Omega_{-})\oplus
\bigwedge L^{2}(\Omega_{+}\setminus \overline{\Omega_{-}})$, the
self-adjoint operator $\Delta_{f,h}^{N}(\Omega_{-})\oplus
\Delta_{f,h}^{D} (\Omega_{+}\setminus \overline{\Omega_{-}})$ is 
denoted by $\Delta_{f,h}^{\oplus}(\Omega_{+})$\,.
\end{definition}
In other words, for any form $u$ such that $u 1_{\Omega_-} \in
D(\Delta_{f,h}^{N}(\Omega_{-}))$ and $u 1_{\Omega_{+}\setminus
  \overline{\Omega_{-}}} \in D(\Delta_{f,h}^{D} (\Omega_{+}\setminus
\overline{\Omega_{-}}))$ (namely if $u \in D(\Delta_{f,h}^{\oplus}(\Omega_{+}))$),
$$\Delta_{f,h}^{\oplus}(\Omega_{+}) u = \Delta_{f,h}^{N}(\Omega_{-})
\left( u 1_{\Omega_-} \right) + \Delta_{f,h}^{D} (\Omega_{+}\setminus
\overline{\Omega_{-}}) \left(u 1_{\Omega_{+}\setminus
  \overline{\Omega_{-}}} \right).$$
It is easy to check that the spectrum of $\Delta_{f,h}^{\oplus, (p)}(\Omega_{+})$ is the union
of the two spectra $\sigma(\Delta_{f,h}^{N,(p)}(\Omega_{-}))$ and
$\sigma(\Delta_{f,h}^{D,(p)}(\Omega_{+}\setminus\overline{\Omega_{-}}))$. Bases
of eigenvectors are given by the direct sum structure. In particular,
we have
$$m_p^{\oplus}(\Omega_+)=m_{p}^{N}(\Omega_{-})+m_{p}^{D}(\Omega_{+}\setminus\overline{\Omega_{-}})$$
where
$m_p^{\oplus}(\Omega_+)=\#\left[\sigma(\Delta_{f,h}^{\oplus}(\Omega_{+})) \cap [0,\nu(h)] \right]
$
denotes the number of small eigenvalues of $\Delta_{f,h}^{\oplus}(\Omega_{+})$.

\begin{proposition}
\label{pr.decayO+}
Let $U$ be an open set satisfying~\eqref{eq:U}.
 Let $(\psi_{k}^{(p)})_{1\leq k\leq m^D_p(\Omega_+)}$, $p=0$ or $1$, 
be an orthonormal basis of eigenvectors of
$\Delta_{f,h}^{D,(p)}(\Omega_{+})$ associated with the eigenvalues in
$[0,\nu(h)]$, and
let $\chi\in \mathcal{C}^{\infty}(\overline{\Omega_{+}})$ be
such that $\chi\equiv 1$ in a neighborhood of $\partial
\Omega_{+}\cup U$ and $0$ in a neighborhood of $\partial\Omega_{-}$\,. For all $k\in
\left\{1,\ldots, m^D_p(\Omega_+) \right\}$, set
$v_{k}^{(p)}=\chi\psi_{k}^{(p)}$\,. The forms $v_{k}^{(p)}$
are close to  $\psi_{k}^{(p)}$, for $k \in
\left\{1,\ldots, m^D_p(\Omega_+)\right\}$:
$$
\sum_{k=1}^{m^D_p(\Omega_+) }
\|\psi_{k}^{(p)}-v_{k}^{(p)}\|_{W^{1,2}(\Omega_{+})}
=\mathcal{O}(e^{-\frac{c_{\chi}}{h}}).
$$
They are almost orthonormal in $L^2(\Omega_+)$:
$$(\langle v_{k}^{(p)}\,,\,
  v_{k'}^{(p)}\rangle_{L^2(\Omega_+)} )_{k,k'}=\Id_{m^D_p(\Omega_+)}+\mathcal{O}(e^{-\frac{c_{\chi}}{h}}).$$
Moreover, they belong to the domain $ D(\Delta_{f,h}^{\oplus,
  (p)}(\Omega_{+}))$ and they satisfy
\begin{equation*}
\langle v_{k}^{(p)}\,,\, \Delta_{f,h}^{\oplus,(p)}(\Omega_{+}) 
v_{k}^{(p)}\rangle_{L^{2}(\Omega_{+})}=\mathcal{O}(e^{-\frac{c_{\chi}}{h}}).
\end{equation*}
All the $\mathcal{O}(e^{-\frac{c_{\chi}}{h}})$ remainders can be
bounded from above by $C_{\chi}e^{-\frac{c_{\chi}}{h}}$
for some constants $C_{\chi},c_{\chi}>0$ independent of $h \in
[0, h_0]$.
\end{proposition}
\begin{proof}
The proof for $p=0$ follows the same lines as the proofs of
Proposition~\ref{pr.decayO-} and Proposition~\ref{pr.almostorhtO-},
because the boundary term in Lemma~\ref{le.Agmon} disappears for
functions vanishing along $\partial \Omega_{+}$\,.\\
For $p=1$, the boundary term has to be taken into account as we did in
the proofs of Proposition~\ref{prop:exp_decay} and Proposition~\ref{pr.decayO+-}. A neighborhood
$\mathcal{V}$ of $\partial \Omega_{+}$ has to be introduced and the
function $\varphi$ used in Lemma~~\ref{le.Agmon}  is $\varphi(x)=(1-\alpha h)d_{Ag}(x,U\cup
\mathcal{V})$ with $\alpha>0$ large enough.
\end{proof}

Notice that  the number
  $m_p^D(\Omega_+)$ of small eigenvalues for
  $\Delta^{D,(p)}_{f,h}(\Omega_+)$ is {\em a priori} depending on
  $h$. We did not indicate explicitly this dependency since the
  result of the next section is that   $m_p^D(\Omega_+)$ is actually
  independent of $h$.

\subsection{On the number of small eigenvalues of $\Delta_{f,h}^{D,(p)}(\Omega_{+})$}
\label{se.conseq}
Using the results of the three previous sections, one can show that
the number $m_{p}^{D}(\Omega_{+})$  of eigenvalues  of
$\Delta_{f,h}^{D,(p)}(\Omega_{+})$ in $[0,\nu(h)]$,
 is actually independent of $h\in (0,h_{0})$.
\begin{proposition}
  \label{pr.number}
For $p\in \left\{0,1\right\}$, the number of eigenvalues of $\Delta_{f,h}^{D,(p)}(\Omega_{+})$ lying
in $[0,\nu(h)]$ is given by
$$
m_{p}^{D}(\Omega_{+})=m_{p}^{N}(\Omega_{-})+m_{p}^{D}(\Omega_{+}\setminus\overline{\Omega_{-}})\,,
$$
where we recall (see Equation~\eqref{eq:specLaplD0})
$m_{0}^{D}(\Omega_{+}\setminus\overline{\Omega_{-}})=0$\,.
Moreover all these eigenvalues are exponentially small, i.e. there
exists $c_{0}'>0$ such that
$$
\forall h\in (0,h_{0})\,,\quad
\sigma(\Delta_{f,h}^{D,(p)}(\Omega_{+}))\cap [0,\nu(h)]\subset
\left[0,e^{-\frac{c_{0}'}{h}} \right]
,\quad p=0,1\,.
$$
\end{proposition}
\begin{proof}
This is obtained as an application of the min-max principle. Indeed,
we know that the spectrum of $\Delta_{f,h}^{D,(p)}(\Omega_{+})$ is
given by the formula: for $k \ge 1$,
$$\lambda^{(p)}_k(\Omega_+)=\sup_{\{\omega_1, \ldots, \omega_{k-1}\}}
Q(\omega_1, \ldots, \omega_{k-1})$$
where
$$Q(\omega_1, \ldots, \omega_{k-1})=
\inf_{v} \left\{ \frac{\langle v \,,\, \Delta_{f,h}^{D,(p)}(\Omega_{+})
v \rangle_{L^{2}(\Omega_{+})}}{\|v\|^2_{L^2(\Omega_+)}}, v \in
D(\Delta_{f,h}^{D,(p)}(\Omega_{+})),  \,
v \in {\rm Span}(\omega_1, \ldots, \omega_{k-1})^{\perp}  \right\}.
$$
By convention, for $k=1$, the supremum is taken over an empty
set (and can thus be neglected). Using
Proposition~\ref{pr.almostorhtO-} and Proposition~\ref{pr.decayO+-},
one can build
$m_{p}:=m_{p}^{N}(\Omega_{-})+m_{p}^{D}(\Omega_{+}\setminus\overline{\Omega_{-}})$
almost orthonormal vectors for which the Rayleigh quotients associated with
$\Delta_{f,h}^{D,(p)}(\Omega_{+})$ are exponentially small. Let us fix
$\varepsilon >0$ and consider $\{\omega_1, \ldots, \omega_{m_{p}-1}\}$
such that $\lambda^{(p)}_{m_{p}}(\Omega_+) \le Q(\omega_1, \ldots,
\omega_{m_{p}-1}) +\varepsilon$. Since, in the limit $h \to 0$, the $m_{p}$
vectors built in Proposition~\ref{pr.almostorhtO-} and
Proposition~\ref{pr.decayO+-} are linearly independent, there exist a
linear combination $v \in D(\Delta_{f,h}^{D,(p)}(\Omega_{+}))$ of these vectors which is in ${\rm
  Span}(\omega_1, \ldots, \omega_{m_{p}-1})^{\perp}$. Using the estimates
on the Rayleigh quotients and the almost orthonormality of these
vectors, one obtains that $\frac{\langle v \,,\, \Delta_{f,h}^{D,(p)}(\Omega_{+})
v \rangle_{L^{2}(\Omega_{+})}}{\|v\|^2_{L^2(\Omega_+)}}={\mathcal
O}(e^{-\frac{c}{h}})$ for some positive constant $c$. This implies
that $Q(\omega_1, \ldots, \omega_{k-1})={\mathcal
O}(e^{-\frac{c}{h}})$ and thus $\lambda^{(p)}_{m_{p}}(\Omega_+)={\mathcal
O}(e^{-\frac{c}{h}})$. Therefore, one gets $m_{p}^{D}(\Omega_{+})\geq m_p=
m_{p}^{N}(\Omega_{-})+m_{p}^{D}(\Omega_{+}\setminus\overline{\Omega_{-}})$.

 A similar reasoning on $\Delta^{\oplus,(p)}_{f,h}(\Omega_+)$ using
 Proposition~\ref{pr.decayO+} gives the opposite inequality
$m_p^{\oplus}(\Omega_+) = m_{p}^{N}(\Omega_{-})+m_{p}^{D}(\Omega_{+}\setminus\overline{\Omega_{-}})\geq
m_{p}^{D}(\Omega_{+})$. This ends the proof.
\end{proof}

\section{Quasimodes for $\Delta_{f,h}^{D,(0)}(\Omega_{+})$}
\label{se.quasimode}

In this section, we specify the quasimodes which will be useful for
the analysis of the spectrum of
$\Delta_{f,h}^{D,(0)}(\Omega_{+})$ lying in $[0,\nu(h)]$. In our
context, a quasimode for
$\Delta_{f,h}^{D,(0)}(\Omega_{+})$ is simply a function $v$ in the
domain $D(\Delta_{f,h}^{D,(0)}(\Omega_{+}))$ such that
$\displaystyle{\frac{\langle v ,\Delta_{f,h}^{D,(0)}(\Omega_{+})  v
    \rangle_{L^2(\Omega_+)}}{\|v\|_{L^2(\Omega_+)}^2}} ={\mathcal O}(e^{-\frac{c}{h}})$. Quasimodes for
$\Delta_{f,h}^{D,(0)}(\Omega_{+})$ will be built from the eigenvectors
of $\Delta_{f,h}^{D,(1)}(\Omega_{+} \setminus \overline{\Omega_-})$ and
of $\Delta_{f,h}^{N,(p)}(\Omega_{-})$
($p=0,1$).

\subsection{The restricted differential $\beta$}
\label{se.restric}

We recall here basic properties of boundary  Witten Laplacians.
\begin{proposition}
 \label{pr.restrict}
Let $\Omega$ be a regular bounded domain of $(M,g)$ and consider the Dirichlet
(resp. Neumann) realization $A=\Delta_{f,h}^{D}(\Omega)$
(resp. $A=\Delta_{f,h}^{N}(\Omega)$) of the Witten Laplacian with form domain
$Q(A)=W^{1,2}_{D}(\Omega)$ (resp. $Q(A)=W^{1,2}_{N}(\Omega)$). The
differential $d_{f,h}$ and codifferential $d_{f,h}^{*}$ satisfy the
commutation property: $\forall z\in \cz\setminus \sigma(A), \, \forall u\in Q(A),$
$$
d_{f,h}(z-A)^{-1}u=(z-A)^{-1}d_{f,h}u\quad\text{and}\quad
d_{f,h}^{*}(z-A)^{-1}u=(z-A)^{-1}d_{f,h}^{*}u\,.
$$
Consequently, for any $\ell \in \rz_{+}$,
\begin{equation}\label{eq:commut}
\begin{aligned}
 d_{f,h}\circ 1_{[0,\ell]}(A^{(p)})&= 1_{[0,\ell]}(A^{(p+1)}) \circ d_{f,h}\\
 d_{f,h}^{*}\circ 1_{[0,\ell]}(A^{(p)})&= 1_{[0,\ell]}(A^{(p-1)}) \circ d_{f,h}^{*}
\end{aligned}
\end{equation}
where $A^{(p)}$ denotes the restriction of $A$ to $p$-forms.
Moreover if $F_{\ell}^{(p)}$ denotes the spectral subspace $\Ran
1_{[0,\ell]}(A^{(p)})$, the chain complex
\begin{equation}
  \label{eq.redcompl}
\xymatrix{
0\ar[r]&
F_{\ell}^{(0)}
\ldots
\ar[r]&F_{\ell}^{(p-1)}
\ar[r]^-{d_{f,h}}&
F_{\ell}^{(p)}
\ar[r]^-{d_{f,h}}&
F_{\ell}^{(p+1)}
\ldots
\ar[r]&F_{\ell}^{(d)}\ar[r]&0
}
\end{equation}
is quasi-isomorphic to the relative (resp. absolute) Hodge-de~Rham
chain complex.
The Witten codifferential $d_{f,h}^{*}$ implements the dual chain-complex.
\end{proposition}
%
%
%
%
%
%
%
%
%
%

We refer to \cite{ChLi,HeNi,Lep3} for the adaptation to
boundary cases of these well known properties of Witten
Laplacians~\cite[Chapter 11]{CFKS}.

Let us give two consequences of that result that are useful in our
context. First, the following property which was
already used in the proof of Proposition~\ref{pr.decayO+-} holds
(using the notation of Proposition~\ref{pr.restrict}):
\begin{equation}\label{eq:vp_p_p+1}
A^{(p)}\psi=\lambda\psi \Rightarrow
\left\{
\begin{array}[c]{l}
 A^{(p+1)}d_{f,h}\psi=\lambda d_{f,h}\psi
\\
A^{(p-1)}d_{f,h}^{*}\psi=\lambda d_{f,h}^{*}\psi
\end{array}
\right.
\end{equation}
with the convention $A^{(-1)}=A^{ (d+1)}=0$. Secondly,
the following orthogonal decompositions hold:
%
%
\begin{align}
  F_{\ell}&=\Ker[A\big|_{F_{\ell}}]\mathop{\oplus}^{\perp}\Ran[d_{f,h}\big|_{F_{\ell}}]
\mathop{\oplus}^{\perp}\Ran[d_{f,h}^{*}\big|_{F_{\ell}}] \label{eq:decomp}\\
 \Ran[d_{f,h}^{*}\big|_{F_{\ell}}]^{\perp}&=\Ker[d_{f,h}\big|_{F_{\ell}}]=
\Ker[A\big|_{F_{\ell}}]\mathop{\oplus}^{\perp}\Ran[d_{f,h}\big|_{F_{\ell}}]\nonumber\\
 \Ran[d_{f,h}\big|_{F_{\ell}}]^{\perp}&=\Ker[d_{f,h}^{*}\big|_{F_{\ell}}]=
\Ker[A\big|_{F_{\ell}}]\mathop{\oplus}^{\perp}\Ran[d_{f,h}^{*}\big|_{F_{\ell}}] \nonumber
\end{align}
where $F_\ell=\oplus_{p=0}^d F_\ell^{(p)}$.
In our problem, we shall use the following notation.
\begin{definition}
  \label{de.pibeta}
Consider the Dirichlet realization $\Delta_{f,h}^{D}(\Omega_{+})$ of
$\Delta_{f,h}$ on $\Omega_{+}$. For $p=0,1$, the operators
$\Pi^{(p)}$ are the  spectral projections 
$$
\Pi^{(p)}=1_{[0,\nu(h)]}(\Delta_{f,h}^{D,(p)}(\Omega_{+}))\,,\quad p=0,1\,
$$
and their range is denoted by $F^{(p)}$. Moreover, the Witten differential $d_{f,h}$
restricted to $F^{(0)}$ is  written $\beta =d_{f,h}\big|_{F^{(0)}}:
F^{(0)}\to F^{(1)} $, so that $\Delta_{f,h}^{D,(0)}(\Omega_{+})|_{F^{(0)}}=
\beta^* \beta$, where $\beta^* =d^*_{f,h}\big|_{F^{(1)}}:
F^{(1)}\to F^{(0)} $.
\end{definition}
A consequence of the commutation properties~\eqref{eq:commut} is the identity:
\begin{equation}
  \label{eq.betaproj}
\beta=\Pi^{(1)}d_{f,h}=d_{f,h}\Pi^{(0)}=\Pi^{(1)}\beta \Pi^{(0)}\,.
\end{equation}
Moreover,~\eqref{eq:decomp} rewrites:
$$F^{(0)}=\Ker[\Delta_{f,h}^{D,(0)}(\Omega_{+})] \mathop{\oplus}^{\perp} \Ran[\beta^*] \text{ and } \Ker(\beta)=\{0\},$$
since $\beta u = d_{f,h} u =0$ and $u=0$ on $\partial \Omega$ imply
$u=0$, and
\begin{align}
 F^{(1)}&=\Ker(\beta^*) \mathop{\oplus}^{\perp} \Ran(\beta)   \label{eq.betadecomp}
\\
&=\Ker[ \Delta_{f,h}^{D,(1)}(\Omega_{+})
] \mathop{\oplus}^{\perp} \Ran(\beta) \mathop{\oplus}^{\perp}
\Ran[ d_{f,h}^{*}\big|_{F^{(2)}}]. \nonumber
\end{align}

\subsection{Truncated eigenvectors}
\label{se.treig}

Let us recall the eigenvectors which have been introduced in
Propositions~\ref{pr.almostorhtO-} and~\ref{pr.decayO+-}.
\begin{itemize}
\item $(\psi_{j}^{(0)})_{1\leq j\leq m_{0}^{N}(\Omega_{-})}$ are
  eigenvectors for the
  operator $\Delta_{f,h}^{N,(0)}(\Omega_{-})$ associated with the  eigenvalues
$0= \mu_{1}^{(0)} (\Omega_-)\leq C_0 e^{-2\frac{\kappa_f -c_0}{h}}\leq \mu_{2}^{(0)} (\Omega_-) \ldots \leq
\mu_{m_{0}^{N}(\Omega_{-})}^{(0)} (\Omega_-)\leq e^{-\frac{c_{0}}{h}} \le
\nu(h)$. The first eigenvector
$\psi_{1}^{(0)}$ associated with the eigenvalue  $\mu_{1}^{(0)}
(\Omega_-) =0$ is 
$
\psi_{1}^{(0)}=\frac{e^{-\frac{f}{h}} \, 1_{\Omega_{-}} } {\left(\int_{\Omega_{-}}e^{-2\frac{f(x)}{h}}~dx\right)^{\frac{1}{2}}}.
$
The lower bound on
$\mu_{2}^{(0)}(\Omega_-)$ stated above is valid for sufficiently small
$h$ and was proven in Proposition~\ref{pr.gapO-}.
\item $(\psi_{k}^{(1)})_{1\leq
  k\leq m_{1}^{N}(\Omega_{-})}$  are
  eigenvectors for the operator
$\Delta_{f,h}^{N,(1)}(\Omega_{-})$ associated with the
$m_{1}^{N}(\Omega_{-})$ eigenvalues
smaller than $\nu(h)$. 
 Using~\eqref{eq:vp_p_p+1}, those
eigenvectors can be labelled such that 
$$
\psi_{k}^{(1)}=(\mu_{k+1}^{(0)}(\Omega_-))^{-1/2}
d_{f,h}\psi_{k+1}^{(0)}=(\mu_{k+1}^{(0)}(\Omega_-) )^{-1/2}
\beta\psi_{k+1}^{(0)}
\qquad\text{for}\qquad
k\in\left\{1,\ldots, m_{0}^{N}(\Omega_{-})-1\right\}.
$$
Notice that we may have
$m_{1}^{N}(\Omega_{-}) = m_{0}^{N}(\Omega_{-})-1$. If not,
using~\eqref{eq.betadecomp}, $\beta^*\psi_{k}^{(1)}
=d_{f,h}^{*}\psi_{k}^{(1)}=0$\,, for $k\geq m_{0}^{N}(\Omega_{-})$\,.
%
%
\item $(\psi_{k}^{(1)})_{m_{1}^{N}(\Omega_{-})+1\leq k\leq
  m_{1}^{N}(\Omega_{-})+m_{1}^{D}(\Omega_{+}\setminus\overline{\Omega_{-}})}$
are eigenvectors for the operator
$\Delta_{f,h}^{D,(1)}(\Omega_{+}\setminus\overline{\Omega_{-}})$ associated
with the $m_{1}^{D}(\Omega_{+}\setminus\overline{\Omega_{-}})$ eigenvalues smaller than $\nu(h)$.
From~\eqref{eq:d*_zero} in the
proof of Proposition~\ref{pr.decayO+-}, we know that
$d_{f,h}^{*}\psi_{k}^{(1)}=\beta^*\psi_{k}^{(1)} =0$. 
\end{itemize}

It has been proven in Proposition~\ref{pr.number}  that
  $m_{0}^{D}(\Omega_{+})=m_{0}^{N}(\Omega_{-})$ and
  $m_{1}^{D}(\Omega_{+})=m_{1}^{N}(\Omega_{-})+m_{1}^{D}(\Omega_{+}\setminus\overline{\Omega_{-}})$. The
  family $(\psi_{j}^{(0)})_{1\leq j\leq
    m_{0}^{D}(\Omega_{+})}$ (respectively $(\psi_{k}^{(1)})_{1\leq k\leq
    m_{1}^{D}(\Omega_{+})}$) is an orthonormal basis of eigenvectors
  for $\Delta_{f,h}^{\oplus,(0)}(\Omega_{+})$ (respectively $\Delta_{f,h}^{\oplus,(1)}(\Omega_{+})$) restricted to the
  spectral range $[0,\nu(h)]$. These two families will be used to
  construct quasimodes for the operator
  $\Delta_{f,h}^{D,(p)}(\Omega_{+})$ restricted to the
  spectral range $[0,\nu(h)]$. This will require some appropriate
  truncations or extrapolations, detailed below.

Let us start with $\psi_{1}^{(0)}$ and let us introduce
\begin{equation}
  \label{eq.expo}
  \tilde{\psi}_{1}^{(0)}= \frac{e^{-\frac{f}{h}} 1_{\Omega_{+}}(x)}{\left(\int_{\Omega_{+}}e^{-2\frac{f(x)}{h}}~dx\right)^{\frac{1}{2}}}.
\end{equation}
These two functions are exponentially close in $L^{2}(\Omega_{+})$:
$$
\|\psi_{1}^{(0)}-\tilde{\psi}_{1}^{(0)}\|_{L^{2}(\Omega_{+})}\leq Ce^{-\frac{c}{h}}\,,
$$
owing to $\forall x \in  \overline{\Omega_+} \setminus
\Omega_-,\,f(x)\geq \min_{\partial \Omega_{-}}f>\min_{\Omega_{+}}f$ and the following upper
and lower bounds of the integral factor.
\begin{lemme}
\label{le.intminmaj} Let $\Omega$ be a regular bounded  domain of $(M,g)$ and
let $f$ belong to  $\mathcal{C}^{\infty}(\overline{\Omega})$ such that
$\min_{\overline{\Omega}}f$ is achieved in $\Omega$. Then
there exists a constant $C_{f}>0$ such that
$$
\frac{1}{C_{f}}h^{d/2}e^{-2\frac{\min_{\Omega}f}{h}}\leq
\int_{\Omega}e^{-2\frac{f(x)}{h}}~dx\leq {\rm Vol}_{g}(\Omega) e^{-2\frac{\min_{\Omega}f}{h}}\,,
$$ 
where ${\rm Vol}_{g}(\Omega)$ denotes the volume of $\Omega$ for
the metric $g$\,.
\end{lemme}
\begin{proof}
The upper bound is obvious since $e^{-2\frac{f(x)}{h}}\leq
e^{-2\frac{\min_{\Omega}f}{h}}$ for all $x\in \Omega$. For the lower bound, write
\begin{align*}
\int_{\Omega}e^{-2\frac{f(x)}{h}}~dx&=\int_{\Omega}\int_{2f(x)}^{+\infty}e^{-\frac{t}{h}}~\frac{dt}{h}dx
=\int_{2\min_{\Omega}f}^{+\infty}{\rm Vol}_{g}(2f<
t)e^{-\frac{t}{h}}~\frac{dt}{h}
\\
&=e^{-2\frac{\min_{\Omega}f}{h}}\int_{0}^{+\infty}
{\rm Vol}_{g}(2f<2\min_{\Omega}f+hs)e^{-s}~ds.
\end{align*}
We assumed the existence of $x_{0}\in \Omega$ such that
$f(x_{0})=\min_{\Omega}f$. Using the Taylor expansion of $f$ around
$x_{0}$, there exists  $r>0$, $h_0>0$ and $s_{0}>0$ such that the ball
$B(x_{0}, \frac{(hs)^{1/2}}{r})$ 
is included
in $\left\{f<\min_{\Omega}f+\frac{hs}{2}\right\}$ for all $s< s_{0}$ and
$h<h_{0}$.
Since ${\rm Vol}_g\left[B(x_{0}, \frac{(hs)^{1/2}}{r})\right] \geq
\frac{1}{C_{r}}(hs)^{d/2}$, we get
$$
\int_{\Omega}e^{-2\frac{f(x)}{h}}~dx\geq
\frac{1}{C_{r}}e^{-2\frac{\min_{\Omega}f}{h}}\int_{0}^{s_{0}}e^{-s}(hs)^{d/2}~ds\geq \frac{h^{d/2}e^{-2\frac{\min_{\Omega}f}{h}}}{C_{f}}\,.
$$
\end{proof}
Compared to the standard Laplace estimate, the interest of
Lemma~\ref{le.intminmaj} is that it holds even if the minimum of $f$
is degenerate.

In all what follows, $U$ denotes a fixed subset of $\Omega_-$
satisfying~\eqref{eq:U}. Let us introduce various cut-off
functions, which all satisfy $0\leq \chi\leq 1$. We refer to Figure~\ref{fig:cutoff}
for an illustration of these cut-off functions, with respect to the
three sets $U \subset \Omega_- \subset \Omega_+$.
\begin{itemize}
\item $\chi_{-}^{(0)}$ and $\chi_{-}^{(1)}$ are two cut-off functions
  like $\chi_{-}$ in Proposition~\ref{pr.almostorhtO-}, that is
  $\chi_{-}^{(p)}\in \mathcal{C}^{\infty}_{0}(\Omega_{-})$ and
  $\chi_{-}^{(p)}\equiv 1$ in a neighborhood of $U$, with the
  additional condition that $\chi_{-}^{(0)}\equiv 1$ in a neighborhood
  of $\supp\chi_{-}^{(1)}$. 
\item $\chi_{+}$ is chosen as in Proposition~\ref{pr.decayO+-}, that
  is $\chi_{+}\in \mathcal{C}^{\infty}(\overline{\Omega_{+}})$,
  $\chi_{+}\equiv 1$ in a neighborhood of $\partial \Omega_{+}$
and $\chi_{+}\equiv 0$ in a neighborhood of
$\overline{\Omega_{-}}$. Let us introduce $c_+>0$ such that
$\chi_{+}\equiv 1$ on $\left\{x\in \overline{\Omega_{+}}, \, d(x,\partial\Omega_{+})\leq
 c_{+} \right\}$.
\item $\chi_{0}$ belongs to $\mathcal{C}^{\infty}_{0}(\Omega_{+})$,
  is equal to $1$ in a neighborhood of $\overline{\Omega_{-}}$ and is
  chosen in such a way that its gradient is supported in $\left\{x\in \Omega_{+},
    d(x,\partial\Omega_{+})\leq \delta_{+}\right\}$, where
  $\delta_{+} \in (0,c_+)$ will be fixed further.
\end{itemize}
\begin{figure}[htbp]
\centerline{\includegraphics[width=0.6\textwidth]{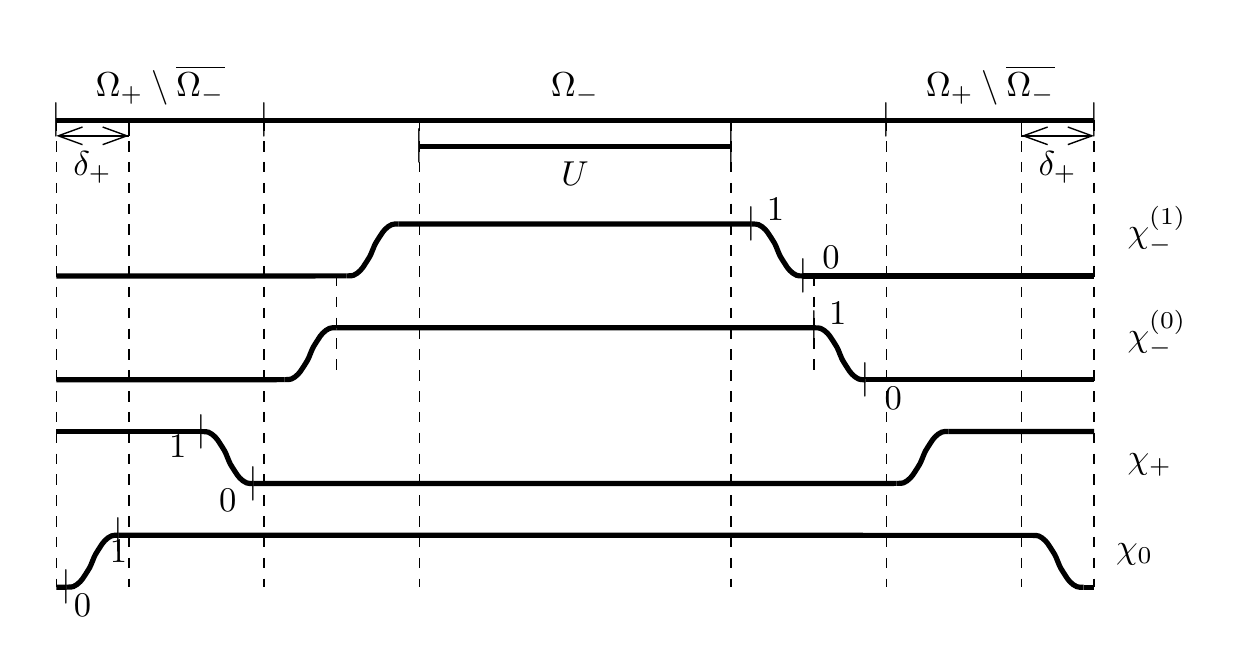}}
\caption{Positions of the domains $\Omega_{+}$, $\Omega_{-}$, $U$
and of the supports of the cut-off functions $\chi_{-}^{(0)}$, $\chi_{-}^{(1)}$,
$\chi_{+}$, $\chi_{0}$.}\label{fig:cutoff}
\end{figure}
We are now in position to introduce a family of quasimodes for the operator
$\Delta^{D,(p)}_{f,h}(\Omega_+)$.
\begin{definition} \label{de.quasimodes} 
Let $\chi_{-}^{(0)}$, $\chi_{-}^{(1)}$,
  $\chi_{+}$ and $\chi_{0}$ be the cut-off functions defined above.
  Let $( \psi_{j}^{(0)})_{1\leq j\leq m_{0}^{D}(\Omega_{+})}$ 
and $(\psi_{k}^{(1)})_{1\leq k\leq m_{1}^{D}(\Omega_{+})}$ be the previously gathered families of eigenvectors of
  $\Delta_{f,h}^{N,(p)}(\Omega_{-})$ and
  $\Delta_{f,h}^{D,(1)}(\Omega_{+}\setminus\overline{\Omega_{-}})$,
  and finally let $\tilde{\psi}_{1}^{(0)}$ be given by
  \eqref{eq.expo}.
The families of vectors $(v_{j}^{(0)})_{1\leq
    j\leq m_{0}^{D}(\Omega_{+})}$ and $(v_{k}^{(1)})_{1\leq k\leq
    m_{1}^{D}(\Omega_{+})}$ are defined by:
  \begin{itemize}
  \item $v_{1}^{(0)}=\chi_{0}\tilde{\psi}_{1}^{(0)}$\,;
  \item $v_{j}^{(0)}=\chi_{-}^{(0)}\psi_{j}^{(0)}$ for $j\in
    \left\{2,\ldots, m_{0}^{D}(\Omega_{+})\right\}$\,;
  \item $v_{k}^{(1)}= \chi_{-}^{(1)}\psi_{k}^{(1)}$ for $k\in
    \left\{1,\ldots, m_{1}^{N}(\Omega_{-})\right\}$\,;
\item $v_{k}^{(1)}=\chi_{+}\psi_{k}^{(1)}$ for $k\in
  \left\{m_{1}^{N}(\Omega_{-})+1,\ldots, m_{1}^{D}(\Omega_{+})\right\}$\,.
  \end{itemize}
\end{definition}

\begin{proposition}  \label{pr.quasi}
The families $(v_{j}^{(0)})_{1\leq j\leq m_{0}^{D}(\Omega_{+})}$ and
$(v_{k}^{(1)})_{1\leq k\leq m_{1}^{D}(\Omega_{+})}$ of
Definition~\ref{de.quasimodes} satisfy the
following properties:
\begin{itemize}
\item[\em 1)] They are almost orthonormal in $L^2(\Omega_+)$:
  \begin{align*}
\left(\langle v_{j}^{(0)}\,,\, v_{j'}^{(0)}\rangle_{L^2(\Omega_+)}\right)_{1\leq
  j,j'\leq m_{0}^{D}(\Omega_{+})}&=\Id_{m_{0}^{D}(\Omega_{+})}+
\mathcal{O}(e^{-\frac{c}{h}})
\,,\\
\left(\langle v_{k}^{(1)}\,,\, v_{k'}^{(1)}\rangle_{L^2(\Omega_+)}\right)_{1\leq
  k,k'\leq m_{1}^{D}(\Omega_{+})}&=\Id_{m_{1}^{D}(\Omega_{+})}+
\mathcal{O}(e^{-\frac{c}{h}})\,,
\end{align*}
for some constant $c>0$ independent of $\delta_{+}$.
\item[\em 2)] The elements $v_{j}^{(0)}$, $1\leq j\leq
  m_{0}^{D}(\Omega_{+})$ (resp. $v_{k}^{(1)}$, $1\leq k\leq
  m_{1}^{D}(\Omega_{+})$) belong to
  $D(\Delta_{f,h}^{D,(0)}(\Omega_{+}))$
  (resp. $\Delta_{f,h}^{D,(1)}(\Omega_{+})$) and satisfy
  \begin{align*}
    &\langle v_{j}^{(0)}\,,\,
    \Delta_{f,h}^{D,(0)}(\Omega_{+})v_{j}^{(0)}\rangle_{L^2(\Omega_+)}=\mathcal{O}(e^{-\frac{c}{h}})\,,
\\
\text{resp. }
&
\langle v_{k}^{(1)}\,,\,
    \Delta_{f,h}^{D,(1)}(\Omega_{+})v_{k}^{(1)}\rangle_{L^2(\Omega_+)}=\mathcal{O}(e^{-\frac{c}{h}})\,,
  \end{align*}
for some constant $c>0$ independent of $\delta_{+}$\,.
\item[\em 3)] Let us consider the spectral projections $\Pi^{(0)}$ and $\Pi^{(1)}$ 
 associated with $\Delta_{f,h}^{D}(\Omega_{+})$ introduced in
 Definition~\ref{de.pibeta}. The elements $v_{j}^{(0)}$, $1\leq j\leq
  m_{0}^{D}(\Omega_{+})$ (resp. $v_{k}^{(1)}$, $1\leq k\leq
  m_{1}^{D}(\Omega_{+})$) satisfy:
  \begin{align*}
&
    \|v_{j}^{(0)}-\Pi^{(0)}v_{j}^{(0)}\|_{L^{2}(\Omega_{+})}=\mathcal{O}(e^{-\frac{c}{h}})\,,\\
\text{resp. }
&
 \|v_{k}^{(1)}-\Pi^{(1)}v_{k}^{(1)}\|_{L^{2}(\Omega_{+})}=\mathcal{O}(e^{-\frac{c}{h}})\,,
  \end{align*}
for some constant $c>0$ independent of $\delta_{+}$.
\end{itemize}
\end{proposition}
\begin{proof}
\noindent{1)}
The family $(\psi_{j}^{(0)})_{1\leq j \leq m_{0}^{D}(\Omega_{+})}$
(resp. $(\psi_{k}^{(1)})_{1\leq k \leq m_{1}^{D}(\Omega_{+})}$) is an
orthonormal basis of eigenvectors of $\Delta_{f,h}^{\oplus, (0)}$
(resp. of $\Delta_{f,h}^{\oplus, (1)}$).
Proposition~\ref{pr.almostorhtO-} implies that
$(\chi_{-}^{(0)}\psi_{j}^{(0)})_{1\leq j\leq m_{0}^{D}(\Omega_{+})}$
is almost orthonormal. The estimate
$\|\chi_{0}\tilde{\psi}_{1}^{(0)}-\chi_{-}^{(0)}\psi_{1}^{(0)}\|_{L^{2}(\Omega_{+})}\leq
Ce^{-\frac{c}{h}}$ (which is a consequence of Lemma~\ref{le.intminmaj}
 and $\forall x \in \overline{\Omega_+} \setminus \Omega_-, \, f(x)\geq \min_{\partial \Omega_{-}}f>\min_{\Omega_{+}}f$) ends the proof of the almost orthonormality of
$(v_{j}^{(0)})_{1\leq j\leq m_{0}^{D}(\Omega_{+})}$\,. For $p=1$, the
two families $(v_{k}^{(1)}=\chi_{-}^{(1)}\psi_{k}^{(1)})_{1\leq k\leq
  m_{1}(\Omega_{-})}$ and $(v_{k}^{(1)}=\chi_{+}\psi_{k}^{(1)})_{m_{1}^{N}(\Omega_{-})+1\leq
k\leq m_{1}^{D}(\Omega_{+})}$ have disjoint supports and lie therefore in
orthogonal subspaces of $L^2(\Omega_+)$. Besides, the almost orthonormality of both
families is again a consequence of the exponential decay of the
$\psi_{k}^{(1)}$, see Proposition~\ref{pr.almostorhtO-} and
Proposition~\ref{pr.decayO+}.\\
\noindent{2)} With the chosen truncations all the vectors
$v_{j}^{(0)}$ (resp. $v_{k}^{(1)}$) belong to the domain
$D(\Delta_{f,h}^{D,(0)}(\Omega_{+}))$
(resp. $D(\Delta_{f,h}^{D,(1)}(\Omega_{+}))$).
In all cases except $p=0$ and $k=1$,  we write for $v=\chi \psi$ (we
omit the index $k$ and the superscript $(p)$) and
$A\psi=\lambda\psi$, 
$A=\Delta_{f,h}^{N}(\Omega_{-})$ or 
$A=\Delta_{f,h}^{D}(\Omega_{+}\setminus\overline{\Omega_{-}})$\,,
$$
\langle v\,,\, \Delta_{f,h}^{D}(\Omega_{+})v\rangle_{L^{2}(\Omega_{+})}= \|d_{f,h}v\|_{
  L^{2}}^{2}+\|d_{f,h}^{*}v\|_{L^{2}(\Omega_{+})}^{2}
\leq \langle \psi\,, A \psi\rangle + C\|\psi\|^{2}_{W^{1,2}(\left\{\chi \neq 1\right\})}
\leq Ce^{-\frac{c}{h}}\,,
$$
owing to $\langle\psi\,,\,
A\psi\rangle=\lambda=\mathcal{O}(e^{-\frac{c_{0}}{h}})$ and to the
estimates on $\psi-v$ given in Proposition~\ref{pr.almostorhtO-} and
Proposition~\ref{pr.decayO+}. For $p=0$ and $k=1$, it is even simpler
because $d_{f,h}\tilde{\psi}_{1}^{(0)}=0$ implies
$$
\langle v_{1}^{(0)},
\Delta_{f,h}^{D,(0)}(\Omega_{+})v_{1}^{(0)}\rangle_{L^{2}(\Omega_{+})}=\|d_{f,h}(\chi_{0}\tilde\psi_{1}^{(0)})\|_{L^{2}(\Omega_{+})}^{2}=\|(hd\chi_{0})\tilde{\psi}_{1}^{(0)}\|_{L^{2}(\Omega_{+})}^{2}\leq C e^{-\frac{c}{h}},
$$
as a consequence of Lemma~\ref{le.intminmaj}
(see~\eqref{eq.estim_delta_v10} below for a more precise estimate).\\
\noindent{3)} All the $v_{j}^{(0)}$'s and $v_{k}^{(1)}$'s satisfy $\langle v\,,\, A
v\rangle_{L^{2}(\Omega_{+})}=\mathcal{O}(e^{-\frac{c}{h}})$ with
$A=\Delta_{f,h}^{D,(0)}(\Omega_{+})$ or
$A=\Delta_{f,h}^{D,(1)}(\Omega_{+})$, and recall that $\Pi^{(0)}$ and
$\Pi^{(1)}$ are the spectral projectors $1_{[0,\nu(h)]}(A)$. The last
estimates are consequences of
$$
\nu(h)\|1_{(\nu(h),+\infty)}(A)v\|_{L^{2}(\Omega_{+})}^{2}\leq \langle
v\,,\, Av\rangle_{L^{2}(\Omega_{+})}\leq Ce^{-\frac{c}{h}}\,,
$$
together with the fact that $\lim_{h\to 0}h\log\nu(h)= 0$, see~\eqref{eq.hypnu}.
\end{proof}
We will need in the following the coefficients 
$\langle v_{k}^{(1)}\,,\, d_{f,h}v_{j}^{(0)}\rangle_{L^2(\Omega_+)}$,
for $j\in \left\{1,\ldots, m_{0}^{D}(\Omega_{+})\right\}$ and $k\in
\left\{1,\ldots, m_{1}^{D}(\Omega_{+})\right\}$.
\begin{proposition}
  \label{pr.dfv}
The coefficients $\langle v_{k}^{(1)}\,,\,
d_{f,h}v_{j}^{(0)}\rangle_{L^2(\Omega_+)}$,
$j\in \left\{1,\ldots, m_{0}^{D}(\Omega_{+})\right\}$, $k\in
\left\{1,\ldots, m_{1}^{D}(\Omega_{+})\right\}$ satisfy:
\begin{itemize}
\item[1)] For $j=1$ and $k\in \left\{1,\ldots,m_{1}^{N}(\Omega_{-})\right\}$: $\langle
  v_{k}^{(1)}\,,\, d_{f,h}v_{1}^{(0)}\rangle_{L^2(\Omega_+)}=0$.
\item[2)] For $j=1$ and $k\in \left\{m_{1}^{N}(\Omega_{-})+1,\ldots,
    m_{1}^{D}(\Omega_{+})
\right\}$: $$\langle
  v_{k}^{(1)}\,,\,
  d_{f,h}v_{1}^{(0)}\rangle_{L^2(\Omega_+)}=-\frac{h\int_{\partial\Omega_{+}}e^{-\frac{f(\sigma)}{h}}\mathbf{i}_{n}\psi_{k}^{(1)}
    (\sigma) d\sigma}{\left(\int_{\Omega_{+}}
e^{-\frac{2f(x)}{h}}~dx\right)^{1/2}}$$ where  
$d\sigma$ is the infinitesimal volume on $\partial \Omega_{+}$
and $n(\sigma)$ the outward normal vector at $\sigma\in \partial
\Omega_{+}$.
\item[3)] For $j\in \left\{2,\ldots, m_{0}^{D}(\Omega_{+})\right\}$ and $k\in
  \left\{1,\ldots, m_{1}^{N}(\Omega_{-})\right\}$:
$$\langle  v_{k}^{(1)}\,,\,
d_{f,h}v_{j}^{(0)}\rangle_{L^2(\Omega_+)}=\sqrt{\mu_{j}^{(0)} (\Omega_-)}
\left(\delta_{k,j-1}+\mathcal{O}(e^{-\frac{c}{h}}) \right).$$
\item[4)] For $j\in \left\{2,\ldots, m_{0}^{D}(\Omega_{+})\right\}$ and $k\in
  \left\{m_{1}^{N}(\Omega_{-})+1,\ldots,
    m_{1}^{D}(\Omega_{+})\right\}$:
$\langle v_{k}^{(1)}\,,\, d_{f,h}v_{j}^{(0)}\rangle_{L^2(\Omega_+)}=0$.
\end{itemize}
\end{proposition}
\begin{proof}
The cases 1) and 4) are due to  the disjoint supports of $d_{f,h}v_{j}^{(0)}$
and $v_{k}^{(1)}$ (see Figure~\ref{fig:cutoff}).

The case 3) comes from the computation
\begin{align*}
d_{f,h}v_{j}^{(0)}
&=
d_{f,h}(\chi_{-}^{(0)}\psi_{j}^{(0)})=\chi_{-}^{(0)}d_{f,h}\psi_{j}^{(0)}+(hd\chi_{-}^{(0)})\wedge
\psi_{j}^{(0)}
\\
&= \sqrt{\mu_{j}^{(0)} (\Omega_-)}\chi_{-}^{(0)}\psi_{j-1}^{(1)}+
\psi_{j}^{(0)}
hd\chi_{-}^{(0)}.
\end{align*}
The condition $\chi_{-}^{(0)}\equiv 1$ in a neighborhood of $\supp
\chi_{-}^{(1)}$ then leads to
\begin{align*}
\langle v_{k}^{(1)}\,,\, d_{f,h}v_{j}^{(0)}\rangle_{L^2(\Omega_+)}
&=\left\langle \chi_{-}^{(1)}\psi_{k}^{(1)}\,,\,
\sqrt{\mu_{j}^{(0)} (\Omega_-)}\psi_{j-1}^{(1)}\right\rangle_{L^2(\Omega_-)}\\
&=
\sqrt{\mu_{j}^{(0)} (\Omega_-)}\delta_{k, j-1}+
\sqrt{\mu_{j}^{(0)} (\Omega_-)}\|(1-\chi_{-}^{(1)})\psi_{k}^{(1)}\|_{L^{2}(\Omega_{-})}\,,
\end{align*}
and we conclude with the exponential decay of $\psi_{k}^{(1)}$ given
by \eqref{eq.expdecpsiO-}
in the proof Proposition~\ref{pr.almostorhtO-}.

For the case 2), we use first
$$
d_{f,h}v_{1}^{(0)}=d_{f,h}(\chi_{0}\tilde{\psi}_{1}^{(0)})=\frac{e^{-\frac{f}{h}}}{\left(\int_{\Omega_{+}}
e^{-\frac{2f(x)}{h}}~dx\right)^{1/2}}hd\chi_{0}\,.
$$
The assumption on the supports of $\chi_{0}$ and $\chi_{+}$ (see Figure~\ref{fig:cutoff}) implies that
$d\chi_{0}$ is supported in the interior of 
$\left\{x\in
  \overline{\Omega_{+}},~\chi_{+}(x)=1\right\}$\,, so that
$$
\left(\int_{\Omega_{+}}
e^{-\frac{2f(x)}{h}}~dx\right)^{1/2}
\langle v_{k}^{(1)}\,,\, d_{f,h}v_{j}^{(0)}\rangle=
\langle
\chi_{+}\psi_{k}^{(1)}\,,\, e^{-\frac{f}{h}}\,hd\chi_{0}\rangle
= 
\langle \psi_{k}^{(1)}\,,\, e^{-\frac{f}{h}} \, hd\chi_{0}\rangle\,.
$$
The definition of the Hodge $\star$ operation, gives
$$
\langle \psi_{k}^{(1)}\,,\, e^{-\frac{f}{h}} \, hd\chi_{0}\rangle
=
h\int_{\Omega_{+}}d\chi_{0}\wedge 
\left[\star\left(e^{-\frac{f}{h}}\psi_{k}^{(1)}\right)\right]
=-h\int_{\Omega_{+}\setminus \overline{\Omega_{-}}}d(1-\chi_{0})\wedge \left[\star\left(e^{-\frac{f}{h}}\psi_{k}^{(1)}\right)\right]\,.
$$
We recall (see~\eqref{eq:d*_zero} in the proof of Proposition~\ref{pr.decayO+-}) that
$d_{f,h}^{*}\psi_{k}^{(1)}=0$ in $\Omega_{+}\setminus \overline{\Omega_{-}}$\,, which means
$$
d\left[\star\left(e^{-\frac{f}{h}}\psi_{k}^{(1)}\right)\right]
=(-1)^{1+1}\star
\left[\frac{e^{-\frac{f}{h}}}{h}d_{f,h}^{*}\psi_{k}^{(1)}\right]=0
\quad\text{in}~\Omega_{+}\setminus \overline{\Omega_{-}}\,.
$$
Hence we get
$$
d(1-\chi_{0})\wedge
\left[\star\left(e^{-\frac{f}{h}}\psi_{k}^{(1)}\right)\right]
=d\left[
(1-\chi_{0})\wedge
\left[\star\left(e^{-\frac{f}{h}}\psi_{k}^{(1)}\right)\right]
\right]
$$
and Stokes' formula yields
$$
\langle \psi_{k}^{(1)}\,,\, e^{-\frac{f}{h}}hd\chi_{0}\rangle
=
-h\int_{\partial\Omega_{+}}e^{-\frac{f}{h}}\star \psi_{k}^{(1)}=
-h\int_{\partial\Omega_{+}}e^{-\frac{f}{h}}\mathbf{t}(\star \psi_{k}^{(1)})\,.
$$
Using the relations \eqref{eq.tn2} $\mathbf{t}\star=\star\mathbf{n}$\,,
and \eqref{eq.tn4} $ \omega_{1}\wedge
(\star \mathbf{n}\omega_{2})=\langle \omega_{1}\,,\,
\mathbf{i}_{n}\omega_{2}\rangle_{\bigwedge^{p-1}T^{*}_{\sigma}\Omega_{+}}~d\sigma$
along $\partial \Omega_{+}$ (where $d\sigma$ is the infinitesimal volume on $\partial \Omega_{+}$
and $n(\sigma)$ the outward normal vector at $\sigma\in \partial
\Omega_{+}$) with $p=1$,  $\omega_{1}=1$ and
$\omega_{2}=\psi_{k}^{(1)}$, we get:
$$
\langle \psi_{k}^{(1)}\,,\, e^{-\frac{f}{h}}hd\chi_{0}\rangle
=-h\int_{\partial
  \Omega_{+}}e^{-\frac{f(\sigma)}{h}}\mathbf{i}_{n}\psi_{k}^{(1)} (\sigma)~d\sigma\,.
$$
This concludes the proof of case 2), and of Proposition~\ref{pr.dfv}.
\end{proof}
\section{Analysis of  the restricted differential $\beta$}
\label{se.matrix}

It is in this section that the assumption~\eqref{eq.decayhyp} is
used. In the following, we assume that the open subset $U$ of
$\Omega_{-}$ which has been used to build the cut-off functions in the
previous section satisfies (in addition to~\eqref{eq:U}):
\begin{equation}\label{eq:U'}
U\cup \mathcal{V}_{-} = \overline{\Omega_{-}},
\end{equation}
where $\mathcal{V}_{-}$ is the neighborhood of $\partial \Omega_-$
introduced in  the assumption~\eqref{eq.decayhyp}.

The main result of this section is the following:
\begin{proposition}
\label{pr.singbeta}
  The singular values of $\beta=d_{f,h}\big|_{F^{(0)}}:F^{(0)}\to
  F^{(1)}$, labelled in decreasing order, are given by
  \begin{align*}
s_{j}(\beta)&=\sqrt{\mu_{m_{0}^{D}(\Omega_{+})+1-j}^{(0)}(\Omega_-) } \, (1+\mathcal{O}(e^{-\frac{c}{h}}))\,,
\quad \text{for}~j \in \{1, \ldots,m_{0}^{D}(\Omega_{+})-1\}\\
s_{m_{0}^{D}(\Omega_{+})}(\beta)&=\frac{h\sqrt{
\sum_{k=m_{1}^{N}(\Omega_{-})+1}^{m_{1}^{D}(\Omega_{+})}\left|\int_{\partial
  \Omega_{+}}e^{-\frac{f(\sigma)}{h}}\mathbf{i}_{n}\psi_{k}^{(1)}(\sigma) d\sigma\right|^{2}}}{
\sqrt{\int_{\Omega_{+}}e^{-2\frac{f(x)}{h}}~dx}
} \, (1+\mathcal{O}(e^{-\frac{c}{h}}))\,,
  \end{align*}
for some $c>0$.
\end{proposition}
According to the notation of Section~\ref{se.treig},
$(\mu_{j}^{(0)}(\Omega_-))_{1\le j \le m_0^D(\Omega_+)}$ are the eigenvalues of
$\Delta_{f,h}^{N,(0)}(\Omega_{-})$ and
$(\psi_{k}^{(1)})_{m_{1}^{N}(\Omega_{-})+1 \le k \le {m_{1}^{D}(\Omega_{+})}}$ are the
eigenvectors of
$\Delta_{f,h}^{D,(1)}(\Omega_{+}\setminus\overline{\Omega_{-}})$.
Notice that, contrary to the eigenvalues of the operators considered
in the previous sections which were labelled in increasing order, the
singular values are naturally labelled in decreasing
order.
 Of course, the singular values of
$\beta$ are related to the small eigenvalues of $\Delta_{f,h}^{D,(0)}(\Omega_+)$
through the relation:
\begin{equation}\label{eq:VS_VP}
\sigma(\Delta_{f,h}^{D,(0)}(\Omega_{+}))\cap[0,\nu(h)]=\left\{s_{k}(\beta)^{2}\,,~1\leq
k\leq m_{0}^{D}(\Omega_{+})\right\},
\end{equation}
since $\Delta_{f,h}^{D,(0)}|_{F^{(0)}}= \beta^* \beta$. 
Proposition~\ref{pr.singbeta} will thus be instrumental in proving Theorem~\ref{th.main}.

The idea of the proof of Proposition~\ref{pr.singbeta} follows the linear algebra argument used in
\cite{HKN,HeNi,Lep3,LNV} and well summarized in
\cite{Lep1}. Notice that
$\beta=d_{f,h}\big|_{F^{(0)}}$ is a {\em finite dimensional} linear
operator. The proof then relies on the following fundamental property
for singular values of matrices. Let us denote
$s_{k}(B)$, $k\in \left\{1,\ldots,\max(n_{0},n_{1})\right\}$, the
singular values of a matrix $B\in
\mathcal{M}_{n_{1},n_{0}}(\cz)$. Then, for any matrices $C_{0} \in
\mathcal{M}_{n_{0}}(\cz)$ and $C_1  \in \mathcal{M}_{n_{1}}(\cz)$,
\begin{equation}\label{eq.Fan1}
s_{k}(B \, C_{0})\leq s_{k}(B)\|C_{0}\| \, ,
 \quad s_{k}(C_{1} \, B)\leq \|C_{1}\|s_{k}(B)\,,
\end{equation}
and for any matrices $C_0 \in GL_{n_{0}}(\cz)$ and $C_1 \in GL_{n_{1}}(\cz)$,
\begin{equation}\label{eq.Fan2}
\frac{1}{\|C_{0}^{-1}\|\|C_{1}^{-1}\|}s_{k}(B)\leq
s_{k}(C_{1}B C_{0})
\leq \|C_{0}\|\|C_{1}\|s_{k}(B)\, ,
\end{equation}
where $\|A\|=\left(\max \sigma(AA^T) \right)^{1/2}$ denotes the
spectral radius of a matrix $A$. The inequalities \eqref{eq.Fan1}
are specific and simple cases of the Ky~Fan inequalities (see for example \cite{Sim} for
a generalization). In particular, when
$C_{p}^{*}C_{p}=\Id_{n_{p}}+\mathcal{O}(\varepsilon)$ ($p=0,1$), the
$k$-th singular value of $B$ is close to the $k$-th singular value of
$C_{1}BC_{0}$: $s_k(C_1 B C_0)=s_k(B)(1+\mathcal{O}(\varepsilon))$. In particular computing the singular values of $\beta$ in almost
orthonormal bases (according to the Definition~\ref{de.almost}), changes
every $s_{k}(\beta)$ into
$s_{k}(\beta)(1+\mathcal{O}(e^{-\frac{c}{h}}))$. To analyze the
singular values of $\beta$, we will use the almost orthonormal bases
built in the previous section.
\begin{remarque}
Our approach, which emphasizes the differential $d_{f,h}$
and allows almost
orthonormal changes of bases, is very close to the work~\cite{BiZh}  by Bismut
and Zhang (see in particular the Section~6) where an
  isomorphism between
the Thom-Smale complex and the Witten complex is constructed\footnote{The second author thanks
  J.M.~Bismut for mentioning this point.}. The interest of our
technique, following~\cite{HeNi,Lep3,LNV}, is that
 the hierarchy of  long range tunnel effects can be analyzed
 accurately, using a Gauss elimination algorithm (see \cite{Lep1}). 
This makes more explicit the inductive process which was used in the former works~\cite{BEGK,BGK} of
Bovier, Eckhoff, Gayrard and Klein.  Actually, the present analysis
shows that the Thom-Smale transversality condition and the Morse
condition are not necessary: introducing the suitable block structure
associated with the assumed geometry of the tunnel effect (see in
particular Hypothesis~\ref{hyp.3}) suffices.
\end{remarque}

\subsection{Structure of  $\beta$}
\label{se.strucbeta}
The estimates
$\|v_{j}^{(0)}-\Pi^{(0)}v_{j}^{(0)}\|_{L^{2}(\Omega_{+})}=\mathcal{O}(e^{-\frac{c}{h}})$
and
$\|v_{k}^{(1)}-\Pi^{(1)}v_{k}^{(1)}\|_{L^{2}(\Omega_{+})}=\mathcal{O}(e^{-\frac{c}{h}})$
of Proposition~\ref{pr.quasi}, together with the results stated in Proposition~\ref{pr.quasi}-1) ensure that 
$$
\mathcal{B}^{(0)}=\left(\Pi^{(0)}v_{j}^{(0)}\right)_{1\leq j\leq m_{0}^{D}(\Omega_{+})}
\quad\text{and}\quad 
\mathcal{B}^{(1)}= \left(\Pi^{(1)}v_{k}^{(1)}\right)_{1\leq k\leq m_{1}^{D}(\Omega_{+})}
$$
are almost orthonormal bases of $F^{(0)}$ and $F^{(1)}$. The same
holds for their dual bases (in $L^2(\Omega_+)$) denoted by $\mathcal{B}^{(0),*}$ and
$\mathcal{B}^{(1),*}$. The matrix of $\beta=d_{f,h}\big|_{F^{(0)}}:F^{(0)}\to F^{(1)}$ in the
bases $\mathcal{B}^{(0)}$, $\mathcal{B}^{(1), *}$ is given by
$$
M(\beta,\mathcal{B}^{(0)},\mathcal{B}^{(1),*})=B=
\begin{pmatrix}
 b_{k,j}
\end{pmatrix}_{\scriptsize
  \begin{array}[l]{l}
    1\leq k\leq m_{1}^{D}(\Omega_{+})\\
    1\leq j\leq m_{0}^{D}(\Omega_{+})
  \end{array}
}\text{ with }
b_{k,j}= \langle \Pi^{(1)}v_{k}^{(1)}\,,\, \beta \Pi^{(0)}v_{j}^{(0)}\rangle_{L^2(\Omega_+)}.
$$
Remember that the coefficients are equivalently written, by using \eqref{eq.betaproj},
\begin{equation}\label{eq:bkj}
b_{k,j}=\langle \Pi^{(1)}v_{k}^{(1)}\,,\, \beta \Pi^{(0)}v_{j}^{(0)}\rangle_{L^2(\Omega_+)}=
\langle \Pi^{(1)}v_{k}^{(1)}\,,\, d_{f,h}v_{j}^{(0)}\rangle_{L^2(\Omega_+)}
=\langle v_{k}^{(1)}\,,\,
d_{f,h}\Pi^{(0)}v_{j}^{(0)}\rangle_{L^2(\Omega_+)}.
\end{equation}
Following the various cases discussed in Proposition~\ref{pr.dfv}
where  the scalar products $\langle
v_{k}^{(1)},d_{f,h}v_{j}^{(0)}\rangle_{L^2(\Omega_+)}$ were studied,
we shall write the matrix $B$ in  block
form:
\begin{align*}
B=
\begin{pmatrix}
  B_{1,1}& B_{1,2}\\
  B_{2,1}& B_{2,2}
\end{pmatrix}
\text{where }&
B_{1,1}=
\begin{pmatrix}
\langle \Pi^{(1)}v_{k}^{(1)}\,,\, d_{f,h}v_{1}^{(0)}\rangle_{L^2(\Omega_+)}  
\end{pmatrix}_{1\leq k\leq m_{1}^{N}(\Omega_{-})}\,,\\
&
B_{1,2}=
\begin{pmatrix}
\langle \Pi^{(1)}v_{k}^{(1)}\,,\, d_{f,h}v_{j}^{(0)}\rangle_{L^2(\Omega_+)}  
\end{pmatrix}_{\scriptsize
  \begin{array}[c]{c}
2\leq j\leq m_{0}^{D}(\Omega_{+})\\1\leq k\leq m_{1}^{N}(\Omega_{-})
\end{array}
}\,,
\\
&
B_{2,1}=
\begin{pmatrix}
 \langle \Pi^{(1)}v_{k}^{(1)}\,,\, d_{f,h}v_{1}^{(0)}\rangle_{L^2(\Omega_+)} 
\end{pmatrix}_{m_{1}^{N}(\Omega_{-})+1\leq k\leq
  m_{1}^{D}(\Omega_{+})}\,,
\\
&
B_{2,2}=
\begin{pmatrix}
\langle \Pi^{(1)}v_{k}^{(1)}\,,\, d_{f,h}v_{j}^{(0)}\rangle_{L^2(\Omega_+)}  
\end{pmatrix}_{\scriptsize
  \begin{array}[c]{c}
2\leq j\leq m_{0}^{D}(\Omega_{+})\\ m_{1}^{N}(\Omega_{-})+1\leq k\leq m_{1}^{D}(\Omega_{+})
\end{array}
}\,.
\end{align*}
In the following, we will give some estimates of each of these blocks
in the asymptotic regime $h \to 0$.  In the following, we denote
\begin{equation}\label{eq:C0}
C_0=2 \|\nabla f\|_{L^\infty(\supp(\nabla \chi_0))}.
\end{equation}
Notice that $C_0>0$. We assume that $\delta_+>0$ is chosen such that
\begin{equation}\label{eq:hyp_delta+}
\delta_+<\frac{\kappa_f}{C_0}.
\end{equation}
 The assumption~\eqref{eq.decayhyp} will be useful to study the blocks $B_{1,2}$ and $B_{2,2}$
and the parameter $\delta_{+}>0$ (see Figure~\ref{fig:cutoff}) will be
further adjusted when
considering the blocks $B_{1,1}$ and $B_{2,1}$.

\subsection{The blocks $B_{1,2}$ and $B_{2,2}$}
\label{se.M12}
Estimates for both blocks rely on the assumption~\eqref{eq.decayhyp}.
Let us start with $B_{1,2}$.
\begin{lemme}
\label{le.M12}
The coefficients of
$B_{1,2}$ satisfy:
$$
b_{k,j}=\langle \Pi^{(1)}v_{k}^{(1)}\,,\, d_{f,h}v_{j}^{(0)} \rangle_{L^2(\Omega_+)}
= \sqrt{\mu_{j}^{(0)} (\Omega_-)}\left(\delta_{k, j-1}+ \mathcal{O}(e^{-\frac{c}{h}})\right)
$$
for $j\in\left\{2,\ldots, m_{0}^{D}(\Omega_{+})\right\}$ and $k\in
\left\{1,\ldots, m_{1}^{N}(\Omega_{-})\right\}$\,.
\end{lemme}
\begin{proof}
Let us first estimate $\|d_{f,h}v_{j}^{(0)}\|_{L^{2}(\Omega_{+})}$ by writing:
$$
d_{f,h}v_{j}^{(0)}= d_{f,h}(\chi_{-}^{(0)}\psi_{j}^{(0)})=\chi_{-}^{(0)}d_{f,h}\psi_{j}^{(0)}+
h\psi_{j}^{(0)}d\chi_{-}^{(0)}=
\chi_{-}^{(0)}\sqrt{\mu_{j}^{(0)} (\Omega_-)}\psi_{j-1}^{(1)}+
h\psi_{j}^{(0)}d\chi_{-}^{(0)}\,.
$$
Since $\supp d\chi_{-}^{(0)}\subset \Omega_{-}\setminus U\subset
\mathcal{V}_{-}$ (see~\eqref{eq:U'}), the assumption~\eqref{eq.decayhyp} implies
$\|d_{f,h}v_{j}^{(0)}\|_{L^{2}(\Omega_{+})}=\tilde{\mathcal{O}}\left(\sqrt{\mu_{j}^{(0)}(\Omega_-)}\right)$\,.
The difference 
$$
\left| \langle \Pi^{(1)}v_{k}^{(1)}\,,\,
d_{f,h}v_{j}^{(0)}\rangle_{L^2(\Omega_+)}-\langle v_{k}^{(1)}\,,\,
d_{f,h}v_{j}^{(0)}\rangle_{L^2(\Omega_+)} \right|
$$
is thus bounded from above by
$$
\|\Pi^{(1)}v_{k}^{(1)}-v_{k}^{(1)}\|_{L^{2}(\Omega_{+})}
\tilde{\mathcal{O}} \left(\sqrt{\mu_{j}^{(0)}(\Omega_-) } \right)
\leq Ce^{-\frac{c'}{2h}}\sqrt{\mu_{j}^{(0)}(\Omega_-)}\,,
$$
owing to the estimate
$\|\Pi^{(1)}v_{k}^{(1)}-v_{k}^{(1)}\|=\mathcal{O}(e^{-\frac{c'}{h}})$
obtained in Proposition~\ref{pr.quasi}-3). 
The result then comes from the expression of $\langle v_{k}^{(1)}\,,\,
d_{f,h}v_{j}^{(0)}\rangle_{L^2(\Omega_+)}$ given in Proposition~\ref{pr.dfv}-3).
\end{proof}
The estimate of the block $B_{2,2}$ follows the same lines.
\begin{lemme}
  \label{le.M22}
The coefficients of $B_{2,2}$ satisfy:
$$
b_{k,j}=\langle \Pi^{(1)}v_{k}^{(1)}\,,\, d_{f,h}v_{j}^{(0)} \rangle_{L^2(\Omega_+)}
= \mathcal{O}\left(\sqrt{\mu_{j}^{(0)}(\Omega_-) }\,e^{-\frac{c}{h}}\right)
$$
for $j\in\left\{2,\ldots, m_{0}^{D}(\Omega_{+})\right\}$ and $k\in
\left\{m_{1}^{N}(\Omega_{-})+1,\ldots, m_{1}^{D}(\Omega_{+})\right\}$.
\end{lemme}
\begin{proof}
Using again
$\|d_{f,h}v_{j}^{(0)}\|=\tilde{\mathcal{O}}
\left(\sqrt{\mu_{j}^{(0)}(\Omega_-) } \right)$,
$\|\Pi^{(1)}v_{k}^{(1)}-v_{k}^{(1)}\|=\mathcal{O}(e^{-\frac{c'}{h}})$
and, according to Proposition~\ref{pr.dfv}-4), $\langle v_{k}^{(1)}\,,\, d_{f,h}v_{j}^{(0)}\rangle=0$
we get
$
|b_{k,j}|\leq Ce^{-\frac{c'}{2h}}\sqrt{\mu_{j}^{(0)} (\Omega_-)}
$.
\end{proof}

\subsection{The block $B_{1,1}$}
\label{se.M11}
In this section, the value of the parameter $\delta_{+}$ is
adjusted. This value will be again possibly changed three other times:
for the estimate of the block $B_{2,1}$ and in the final proof of
Theorem~\ref{th.main}, see Sections~\ref{se.appu10} and \ref{se.appdfu}.
Remember that the constant $c$ occurring in the remainders
$\mathcal{O}(e^{-\frac{c}{h}})$ introduced in Proposition~\ref{pr.quasi} do
not depend on $\delta_{+}>0$.
\begin{lemme}
\label{le.M11} For any $k\in \left\{1,\ldots,
  m_{1}^{N}(\Omega_{-})\right\}$, the matrix element $b_{k,1}$
satisfies 
$$
b_{k,1}=\langle \Pi^{(1)}v_{k}^{1}\,,\,
  d_{f,h}v_{1}^{(0)}\rangle_{L^2(\Omega_+)}
=\mathcal{O} \left(e^{-\frac{\kappa_{f}+c-C_0\delta_{+}}{h}} \right)
$$
where $\kappa_{f}=\min_{\partial \Omega_{+}}f-\min_{\Omega_{+}}f$, and
the constants $c>0$ and $C_0>0$ (defined by~\eqref{eq:C0}) are independent of $\delta_{+}>0$. In
particular
when $\delta_{+}>0$ is chosen smaller than $\frac{c}{C_0}$, one gets
$$
b_{k,1}=\mathcal{O} \left(e^{-\frac{\kappa_{f}+c}{h}} \right),
$$
for a positive constant $c$, which depends on $\delta_{+}$.
\end{lemme}
\begin{proof}
Remember that
$v_{1}^{(0)}=\chi_{0}\tilde{\psi}_{1}^{(0)}=\frac{\chi_{0}e^{-\frac{f}{h}}}{\left(\int_{\Omega_{+}}e^{-\frac{2f(x)}{h}}~dx\right)^{1/2}}$
where $\nabla\chi_{0}$ is supported in $\left\{x\in
  \Omega_{+},~d(x,\partial \Omega_{+})< \delta_{+}\right\}$ (see Figure~\ref{fig:cutoff}).
The Witten differential of $v_{1}^{(0)}$ satisfies
$$
d_{f,h}v_{1}^{(0)}= \frac{e^{-\frac{f}{h}}}{\left(\int_{\Omega_{+}}e^{-\frac{2f(x)}{h}}~dx\right)^{1/2}}(hd\chi_{0})
$$
and its $L^{2}$-norm can be estimated by:
$$
\|d_{f,h}v_{1}^{(0)}\|_{L^{2}(\Omega_{+})}^{2}\leq
C_{\chi_{0}}\frac{\int_{\supp (\nabla\chi_{0})}e^{-\frac{2f(x)}{h}}~dx}{\int_{\Omega_{+}}e^{-2\frac{f(x)}{h}}~dx}\,.
$$
With $f(x)\geq \min_{\partial\Omega_{+}}f  - \frac{C_0}{2}\delta_{+}$ for $x\in
\supp (\nabla\chi_{0})$ (where $C_0$ is defined by~\eqref{eq:C0} and
does not depend on~$\delta_+$) and the lower bound
$\int_{\Omega_{+}}e^{-2\frac{f(x)}{h}}~dx \geq
\frac{h^{d/2}e^{-2\frac{\min_{\Omega_{+}}f}{h}}}{C_1}$ of Lemma~\ref{le.intminmaj}, we get
\begin{equation}\label{eq.estim_delta_v10}
\|d_{f,h}v_{1}^{(0)}\|_{L^{2}(\Omega_{+})}^{2}
\leq C_1 h^{-d/2}e^{-2\frac{\kappa_{f}-(C_0\delta_{+}/2)}{h}}\leq C_2e^{-2\frac{\kappa_{f}-C_0\delta_{+}}{h}}\,,
\end{equation}
provided that $h$ is small enough. Then using like in Lemma~\ref{le.M12}
$$
|b_{k,1}-\langle v_{k}^{(1)}\,,\,
d_{f,h}v_{1}^{(0)}\rangle_{L^2(\Omega_+)}|
\leq \|\Pi^{(1)}v_{k}^{(1)}-v_{k}^{(1)}\|_{L^{2}(\Omega_{+})}\|d_{f,h}v_{1}^{(0)}\|_{L^{2}(\Omega_{+})}
\leq C_{3}e^{-\frac{c'}{h}}e^{-\frac{\kappa_{f}-C_0\delta_{+}}{h}}\,,
$$
the equality $\langle v_{k}^{(1)}\,,\,
d_{f,h}v_{1}^{(0)}\rangle=0$ (see Proposition~\ref{pr.dfv}-1)) yields
the result.
\end{proof}
\begin{remarque}\label{re.m1D}
If $m_1^D(\Omega_+ \setminus \Omega_-)=0$ (and thus $m_1^N(\Omega_-)=m_1^D(\Omega_+)$), the previous lemma shows
that:
\begin{align*}
\langle \Pi^{(0)} v_1^{(0)}, \beta^* \beta \Pi^{(0)} v_1^{(0)}
\rangle_{F^{(0)}}
&= \|  \beta \Pi^{(0)} v_1^{(0)} \|^2_{F^{(0)}} \\
&= \sum_{k=1}^{m_1^N(\Omega_-)} |b_{k,1}|^2 \left( 1 +
  \mathcal{O}(e^{-c/h}) \right)\\
&={\mathcal O} \left(e^{-\frac{\kappa_f+c}{h}} \right).
\end{align*}
This implies that $\beta^*\beta$ (and therefore
$\Delta_{f,h}^{D,(0)}(\Omega_+)$) has an eigenvalue of the order
${\mathcal O}(e^{-\frac{\kappa_f+c}{h}})$, which contradicts the
Lemma~\ref{le.infsigD0} below. Therefore, $m_1^D(\Omega_+ \setminus \Omega_-)$ is
not zero.
\end{remarque}

\subsection{The block $B_{2,1}$}
\label{se.M21}
We shall first provide an approximate expression for the coefficients
of the column $B_{2,1}$.
\begin{proposition}
\label{pr.M21}
For any $k\in \left\{m_{1}^{N}(\Omega_{-})+1,\ldots, m_{1}^{D}(\Omega_{+})\right\}$, the
matrix element $b_{k,1}=\langle \Pi^{(1)}v_{k}^{1}\,,\, d_{f,h}v_{1}^{(0)}\rangle_{L^2(\Omega_+)}$ satisfies
\begin{equation}\label{eq.M21}
b_{k,1}=
-\frac{h\int_{\partial\Omega_{+}}e^{-\frac{f(\sigma)}{h}}\mathbf{i}_{n}
    \psi_{k}^{(1)} (\sigma)~d\sigma}{\left(\int_{\Omega_{+}}e^{-2\frac{f(x)}{h}}~dx\right)^{1/2}}
+\mathcal{O} \left(e^{-\frac{\kappa_{f}+c}{h}} \right)\,,
\end{equation}
where $c$ is a positive constant which depends on $\delta_{+}>0$
chosen sufficiently small, and
$\kappa_{f}=\min_{\partial\Omega_{+}}f-\min_{\Omega_{+}}f$\,.
Moreover, these coefficients $b_{k,1}$ satisfy
\begin{equation}
  \label{eq.minlog}
\lim_{h\to 0}
h\log\left[\sum_{k=m_{1}^{N}(\Omega_{-})+1}^{m_{1}^{D}(\Omega_{+})}|b_{k,1}|^{2}\right]
= -2\kappa_{f}\,.
\end{equation}
\end{proposition}
The estimate~\eqref{eq.minlog} shows that the
approximation~\eqref{eq.M21} is meaningful, in the sense that some of
the coefficients $b_{k,1}$ are indeed larger than the error term $\mathcal{O} \left(e^{-\frac{\kappa_{f}+c}{h}} \right)$. In
particular, we have:
\begin{equation}\label{eq:expansion_bk1}
\sum_{k=m_{1}^{N}(\Omega_{-})+1}^{m_{1}^{D}(\Omega_{+})}|b_{k,1}|^{2}=\frac{
h^{2}\sum_{k=m_{1}^{N}(\Omega_{-})+1}^{m_{1}^{D}(\Omega_{+})}
\left(\int_{\partial\Omega_{+}}e^{-\frac{f(\sigma)}{h}}\mathbf{i}_{n}
\psi_{k}^{(1)}(\sigma) \,d\sigma\right)^{2}}{
\int_{\Omega_{+}}e^{-2\frac{f(x)}{h}}\,dx}
(1+\mathcal{O}(e^{-\frac{c}{h}})).
\end{equation}
\begin{proof}
The first statement is proved like in Lemma~\ref{le.M11} after recalling 
$$
\langle v_{k}^{(1)}\,,\, d_{f,h}v_{1}^{(0)}\rangle=
-
\frac{h\int_{\partial\Omega_{+}}e^{-\frac{f(\sigma)}{h}}\mathbf{i}_{n}
    \psi_{k}^{(1)}(\sigma)~d\sigma}{\left(\int_{\Omega_{+}}e^{-2\frac{f(x)}{h}}~dx\right)^{1/2}}\,,
$$
according to Proposition~\ref{pr.dfv}-2).

For the equality \eqref{eq.minlog}, the upper bound
$$
\limsup_{h\to
  0}h\log\left[\sum_{k=m_{1}^{N}(\Omega_{-})+1}^{m_{1}^{D}(\Omega_{+})}|b_{k,1}|^{2}\right]\leq -2\kappa_{f}\,,
$$
is a consequence of
\begin{equation*}
 \left| \frac{\int_{\partial \Omega_{+}}
e^{-\frac{f(\sigma)}{h}}\mathbf{i}_{n}
    \psi_{k}^{(1)}(\sigma)~d\sigma}{\left(\int_{\Omega_{+}}e^{-2\frac{f(x)}{h}}~dx\right)^{1/2}}
\right|
\leq C\frac{\left(\int_{\partial
      \Omega_{+}}|\mathbf{i}_{n}\psi_{k}^{(1)}(\sigma)|^{2}~d\sigma\right)^{1/2}}{\left(\int_{\Omega_{+}}e^{-2\frac{f(x)-\min_{\Omega_{+}}f}{h}}dx\right)^{1/2}}
e^{-\frac{\kappa_{f}}{h}},
\end{equation*}
where the denominator is bounded from below by
Lemma~\ref{le.intminmaj}. The numerator is estimated by 
$\|\psi_{k}^{(1)}\big|_{\partial \Omega_{+}}\|_{L^{2}(\partial
  \Omega_{+})}\leq
C\|\psi_{k}^{(1)}\|_{W^{1,2}(\mathcal{V})}=\mathcal{O}(h^{-1})=\tilde{\mathcal{O}}(1)$
owing to Proposition~\ref{prop:exp_decay} since $d_{Ag}(x,\mathcal{V})=0$
for $x \in \mathcal{V}$.
Using Lemma~\ref{le.M11}, the lower bound for \eqref{eq.minlog} is
equivalent to 
\begin{equation}\label{eq:LB}
\liminf_{h\to
  0}h\log\left[\sum_{k=1}^{m_{1}^{D}(\Omega_{+})}|b_{k,1}|^{2}\right]\geq -2\kappa_{f}\,.
\end{equation}
Since $b_{k,1}=\langle \Pi^{(1)}v_{k}^{(1)}\,,\,
d_{f,h}\Pi^{(0)}v_{1}^{(0)}\rangle_{L^2(\Omega_+)}$ is the $k$-th component of
$d_{f,h}\Pi^{(0)}v_{1}^{(0)}\in F^{(1)}$ in the almost orthonormal
basis $\mathcal{B}^{(1),*}$ of $F^{(1)}$, the inequality
\eqref{eq:LB} is equivalent to
$$
\liminf_{h\to
  0}h\log\left(\|d_{f,h}\Pi^{(0)}v_{1}^{(0)}\|_{L^2(\Omega_+)}^{2}\right)
=\liminf_{h\to
  0}h\log\left(\langle \Pi^{(0)}v_{1}^{(0)}\,,\, \Delta_{f,h}^{D,(0)}(\Omega_{+})\Pi^{(0)}v_{1}^{(0)}\rangle_{L^2(\Omega_+)}\right)
\geq -2\kappa_{f}\,.
$$
With 
$\|\Pi^{(0)}v_{1}^{(0)}\|_{L^{2}(\Omega_{+})}=1+\mathcal{O}(e^{-\frac{c}{h}})$,
the last inequality is a consequence of 
$$
\liminf_{h\to 0}h\log\left[\min
  \sigma(\Delta_{f,h}^{D,(0)}(\Omega_{+}))\right]\geq -2\kappa_{f}\,,
$$
which is proved in the next lemma.
\end{proof}
\begin{remarque}
Using Lemma~\ref{le.M11},  the asymptotic result~\eqref{eq.minlog} is
actually equivalent to
$$\lim_{h\to 0}
h\log\left[\sum_{k=1}^{m_{1}^{D}(\Omega_{+})}|b_{k,1}|^{2}\right]
= -2\kappa_{f}.
$$
\end{remarque}
We end this section with an estimate on the bottom of the spectrum of
$\Delta_{f,h}^{D,(0)}(\Omega_{+})$, which was used to conclude the proof of
Proposition~\ref{pr.M21} above.
\begin{lemme}
\label{le.infsigD0}
The bottom of the spectrum of $\Delta_{f,h}^{D,(0)}(\Omega_{+})$ satisfies
$$
\lim_{h\to 0}h\log\left[\min
  \sigma(\Delta_{f,h}^{D,(0)}(\Omega_{+}))\right]= -2\kappa_{f}.
$$
In particular, we have:
$$
\forall \varepsilon>0\,, \exists C_{\varepsilon}>1\,,\exists h_{\varepsilon}>0\,,\forall h\in (0,h_{\varepsilon}],\quad
\min
  \sigma(\Delta_{f,h}^{D,(0)}(\Omega_{+}))\geq \frac{1}{C_{\varepsilon}}
e^{-2\frac{\kappa_{f}+\varepsilon}{h}}.
$$
\end{lemme}
\begin{proof}
Let us introduce a function  $w_1^{(0)}$ defined similarly to
$v_1^{(0)}$ by $w_1^{(0)}=\tilde\chi_0 \tilde\psi_1^{(0)}$
where
$\tilde{\chi}_0$ is a $\mathcal{C}_0^\infty(\Omega_+)$ function, equal
  to $1$ in a neighborhood of $\Omega_-$ and such that $d\tilde{\chi}_0$ is
  supported in $\{x \in \Omega_+, d(x, \partial \Omega_+) \le \delta \}$.
The estimate $\limsup_{h\to 0}h\log\left[\min
  \sigma(\Delta_{f,h}^{D,(0)}(\Omega_{+}))\right] \le
-2\kappa_{f}$ is then a consequence of the computation:
\begin{equation}\label{eq:estim_w10}
\langle w_1^{(0)}, \Delta_{f,h} w_1^{(0)}\rangle_{L^2(\Omega_+)}
= \| d_{f,h} w_1^{(0)}\|_{L^2(\Omega_+)}^2=\tilde{\mathcal O}
\left(e^{-\frac{2(\kappa_f - C_0 \delta)}{h}} \right)
\end{equation}
by considering $\delta$ arbitrarily small. The last
equality  is proved like~\eqref{eq.estim_delta_v10} above. 

It remains to prove that $\liminf_{h\to 0}h\log\left[\min
  \sigma(\Delta_{f,h}^{D,(0)}(\Omega_{+}))\right] \ge
-2\kappa_{f}$. The proof is very similar to the one of Proposition~\ref{pr.gapO-}. Assume on the contrary that there exists
$\varepsilon_{0}>0$ and a sequence $h_{n}$
such that $\lim_{n \to \infty} h_n=0$ and
$$
\min
  \sigma(\Delta_{f,h_{n}}^{D,(0)}(\Omega_{+}))\leq C
e^{-2\frac{\kappa_{f}+\varepsilon_{0}}{h_{n}}}\,.
$$
To simplify the notation, let us drop the subscript
$_{n}$ in $h_{n}$. The previous inequality means that there exists
$v_{h}\in L^{2}(\Omega_{+})$ and $\lambda_{h}\geq 0$ such that
\begin{align}
\label{eq.interieur}
&\Delta_{f,h}^{D,(0)}v_{h}=\lambda_{h}v_{h}~\text{in}~\Omega_{+}
,\quad v_{h}\big|_{\partial\Omega_{+}}=0, \quad \|v_{h}\|_{L^{2}(\Omega_{+})}=1,
\\
\label{eq.dfhmaj}
&\lambda_{h}=\langle v_{h}\,,\,
\Delta_{f,h}^{D,(0)}(\Omega_{+})v_{h}\rangle _{L^{2}(\Omega_{+})}=
\|d_{f,h}v_{h}\|_{L^{2}(\Omega_{+})}^{2}
\leq Ce^{-2\frac{\kappa_{f}+\varepsilon_{0}}{h}}.
\end{align}
For a small $t>0$, let us consider the domain
$$
\Omega_{t}=\left\{x\in \Omega_{+},\, f(x)<
  \min_{\partial \Omega_{+}}f +t\right\}.
$$
With $d_{f,h}=e^{-\frac{f-\min_{\Omega_{+}} f}{h}}(hd)e^{\frac{f-\min_{\Omega_{+}} f}{h}}$\,,
the estimate~\eqref{eq.dfhmaj} implies
\begin{equation}\label{eq:gradient_petit}
\left\|d \left(e^{\frac{f-\min_{\Omega_{+}} f}{h}}v_{h}\right)
\right\|_{L^{2}(\Omega_{t})}
\leq h^{-1} \max_{x \in \Omega_t} e^{{\frac{ f(x)- \min_{\Omega_{+}} f}{h}}} \|d_{f,h}v_{h}\|_{L^{2}(\Omega_{t})}
\leq 
Ch^{-1}e^{-\frac{\varepsilon_{0}-t}{h}}
=\mathcal{O}(e^{-\frac{\varepsilon_{0}}{2h}})
\end{equation}
as soon as $t<\frac{\varepsilon_{0}}{2}$.

For a given $t\in (0,\frac{\varepsilon_{0}}{2})$, let us now prove
that $\|v_{h}\|_{L^{2}(\Omega_{t})}$ is close to one, using the same
reasoning as in the proof of Proposition~\ref{pr.decayO-}. There
exists is an open neighborhood $\mathcal{V}$  of $\left\{x\in \Omega_{-}\,,\;\nabla
  f(x)=0\right\}$ such that $\mathcal{V} \subset \Omega_t$ and 
\begin{equation}\label{eq:min_par_c}
d_{Ag}(\Omega_{+}\setminus \Omega_{t},
\mathcal{V})\geq c>0,
\end{equation} where $c$ can be chosen independently of $t$
and $\varepsilon_0$ and is positive according to Hypothesis~\ref{hyp.3}.
 Applying
Lemma~\ref{le.Agmon} with $\Omega=\Omega_{+}$ and
$\varphi=(1-\alpha h)d_{Ag}(.,\mathcal{V})$, one gets for $h < 1/\alpha$ (similarly to~\eqref{eq:ineq1}):
$$0 \ge h^2 \|d(e^{\frac{\varphi}{h}} v_h)\|_{L^2(\Omega_+)}^2 + h
\left[ \alpha \langle e^{\frac{\varphi}{h}} v_h, |\nabla f|^2 e^{\frac{\varphi}{h}} v_h \rangle_{L^2(\Omega_+)}
- C_f \|e^{\frac{\varphi}{h}} v_h\|^2_{L^2(\Omega_+)} \right].$$
By choosing $\alpha$ sufficiently large so that $\alpha \min_{\Omega_+
  \setminus \mathcal{V}} |\nabla f|^2  \ge 2 C_f$, we get
$$0 \ge h^2 \|d(e^{\frac{\varphi}{h}} v_h)\|_{L^2(\Omega_+)}^2 + h
\left[ C_f \|e^{\frac{\varphi}{h}} v_h\|^2_{L^2( \Omega_+ \setminus {\mathcal V})} - C_f
  \|e^{\frac{\varphi}{h}} v_h\|^2_{L^2( {\mathcal
      V})} \right].$$
Using the fact that $ \|e^{\frac{\varphi}{h}} v_h\|^2_{L^2( {\mathcal
      V})} =  \| v_h\|^2_{L^2( {\mathcal
      V})}\le 1$, we obtain, by adding $2 C_f h \| v_h\|^2_{L^2( {\mathcal
      V})}$ on both sides of the previous inequality,
$$2 C_f h \ge 2 C_f h \| v_h\|^2_{L^2( {\mathcal
      V})} \ge h^2 \|d(e^{\frac{\varphi}{h}} v_h)\|_{L^2(\Omega_+)}^2 + h
C_f \|e^{\frac{\varphi}{h}} v_h\|^2_{L^2( \Omega_+) }.$$
This implies in particular
$$ \|e^{\frac{ d_{Ag}(.,\mathcal{V})}{h}} v_h\|^2_{L^2(
  \Omega_+) } \le 2$$
and thus, using~\eqref{eq:min_par_c},
$$ \| v_h\|^2_{L^2(
  \Omega_+ \setminus \Omega_t) } \le Ce^{-\frac{c}{h}}.$$
This implies
\begin{equation}\label{eq:proche_de_un}
\|e^{\frac{f-\min_{\Omega_{+}} f}{h}}v_{h}\|_{L^{2}(\Omega_{t})}\geq
\|v_{h}\|_{L^{2}(\Omega_{t})}
\geq 1-Ce^{-\frac{c}{h}}
\end{equation}
where, we recall, $c$ is independent of $t$ and $\varepsilon_{0}$, supposed to be small
enough.

The two estimates~\eqref{eq:gradient_petit} and~\eqref{eq:proche_de_un} lead to a contradiction. Indeed, let us now set $t=\frac{\varepsilon_{0}}{4}$. 
The Poincar{\'e}-Wirtinger inequality, or equivalently the spectral gap
estimate for the Neumann Laplacian in
$\Omega_{\frac{\varepsilon_{0}}{4}}$,  implies that there exists a constant
$C_{h}$ such that
$$
\|(e^{\frac{f-\min_{\Omega_{+}} f}{h}}v_{h})-C_{h}\|_{L^{2}(\Omega_{\frac{\varepsilon_{0}}{4}})}=\mathcal{O}(e^{-\frac{\varepsilon_{0}}{2h}})\,,
$$
and therefore 
$$
\|(e^{\frac{f-\min_{\Omega_{+}} f}{h}}v_{h})-C_{h}\|_{W^{1,2}(\Omega_{\frac{\varepsilon_{0}}{4}})}=\mathcal{O}(e^{-\frac{\varepsilon_{0}}{2h}})\,.
$$
Since $\Omega_{\frac{\varepsilon_{0}}{4}}\cap \partial\Omega_{+}$ has
a non empty interior $U_{\varepsilon_{0}}$, the trace theorem implies
$$
\|(e^{\frac{f-\min_{\Omega_{+}} f}{h}}v_{h})-C_{h}\|_{L^{2}(U_{\varepsilon_{0}})}=
\mathcal{O}(e^{-\frac{\varepsilon_{0}}{2h}})\,.
$$
Since $v_{h}\big|_{\partial \Omega_{+}}\equiv 0$ and since
$U_{\varepsilon_{0}}$ is fixed by $\varepsilon_{0}$ and
independent of $h$, this implies
$C_{h}=\mathcal{O}(e^{-\frac{\varepsilon_{0}}{2h}})$.
We are led to
$$
1-Ce^{-\frac{c}{h}}\leq
\|v_{h}\|_{L^{2}(\Omega_{\frac{\varepsilon_{0}}{4}})}
\leq 
\|e^{\frac{f-\min_{\Omega_{+}} f}{h}}v_{h}\|_{L^{2}(\Omega_{\frac{\varepsilon_{0}}{4}})}
\leq
\|C_{h}\|_{L^{2}(\Omega_{\frac{\varepsilon_{0}}{4}})}+Ce^{-\frac{\varepsilon_{0}}{2h}}\leq 
C'e^{-\frac{\varepsilon_{0}}{2h}}\,,
$$
which is impossible when $h$ is small enough.
\end{proof}
This Lemma shows the equality~\eqref{eq.applambda10} stated in Theorem~\ref{th.main}.

\subsection{Singular values of $\beta$}
\label{se.singbet}
We are now in position to complete the proof of
Proposition~\ref{pr.singbeta}.

\begin{proof}[Proof of Proposition~\ref{pr.singbeta}]
Let $e^{(0)}=(e^{(0)}_1,\ldots,e^{(0)}_{m_0^D(\Omega_+)})$
(resp. $e^{(1)}= (e^{(1)}_1,\ldots,e^{(1)}_{m_1^D(\Omega_+)})$) denote an
orthonormal basis of $F^{(0)}$ (resp. of $F^{(1)}$) and let $C_{0}$
(resp. $C_{1}$) be the matrix of the change of basis from
$e^{(0)}$
(resp. from $\mathcal{B}^{(1),*}$) to
$\mathcal{B}^{(0)}$ (resp. to $e^{(1)}$). Let $A=M(\beta,e^{(0)},e^{(1)})$ denote the
matrix of $\beta$ in the bases $e^{(0)}$ and $e^{(1)}$, so that
$$A= C_1 B C_0$$
where, we recall, $B=M(\beta,\mathcal{B}^{(0)},\mathcal{B}^{(1),*})$.
Using the fact that $\mathcal{B}^{(0)}$ and $\mathcal{B}^{(1)}$ are almost
orthonormal basis, the matrices $C_0$ and $C_1$ satisfy
$C_p^*C_p=\Id+\mathcal{O}(\varepsilon)$ so that (according to~\eqref{eq.Fan2}):
$$
s_{j}(\beta)=s_{j}(A) =s_{j}(C_1 B C_0)=s_{j}(B) (1+\mathcal{O}(e^{-\frac{c}{h}}))\,.
$$
The singular values of $\beta$ can be understood from those of $B$, up
to exponentially small relative errors.

Now Lemma~\ref{le.M12}, \ref{le.M22}, \ref{le.M11} and
Proposition~\ref{pr.M21} can be gathered into (using the block
structure of $B$ introduced in Section~\ref{se.strucbeta}): in the
asymptotic regime $h \to 0$,
$$
B=
\left(
\begin{array}{c|cccc}
  \mathcal{O}(b_{k_{0},1}e^{-\frac{c}{h}}) &
  b_{1,2}&\mathcal{O}(b_{2,3}e^{-\frac{c}{h}})&\ldots&
  \mathcal{O}(b_{m_{0}^{D}-1,m_{0}^{D}}e^{-\frac{c}{h}})\\ 
\vdots&\mathcal{O}(b_{1,2}e^{-\frac{c}{h}})& b_{2,3}&
\ldots&\vdots\\
\vdots&\vdots&\vdots&\ddots&\mathcal{O}(b_{m_{0}^{D}-1,m_{0}^{D}}e^{-\frac{c}{h}})
\\
\mathcal{O}(b_{k_{0},1}e^{-\frac{c}{h}}) &
 \mathcal{O}(b_{1,2}e^{-\frac{c}{h}})&
  \mathcal{O}(b_{2,3}e^{-\frac{c}{h}})&\ldots
&b_{m_{0}^{D}-1,m_{0}^{D}}\\
 \mathcal{O}(b_{k_{0},1}e^{-\frac{c}{h}}) &
 \mathcal{O}(b_{1,2}e^{-\frac{c}{h}})&
  \mathcal{O}(b_{2,3}e^{-\frac{c}{h}})&\ldots
&\mathcal{O}(b_{m_{0}^{D}-1,m_{0}^{D}}e^{-\frac{c}{h}})
\\
\vdots&\vdots&\vdots&\vdots&\vdots
\\
\mathcal{O}(b_{k_{0},1}e^{-\frac{c}{h}}) &
 \mathcal{O}(b_{1,2}e^{-\frac{c}{h}})&
  \mathcal{O}(b_{2,3}e^{-\frac{c}{h}})&\ldots 
&\mathcal{O}(b_{m_{0}^{D}-1,m_{0}^{D}}e^{-\frac{c}{h}})\\[3pt]
\hline b_{m_{1}^{N}+1,1}&\mathcal{O}(b_{1,2}e^{-\frac{c}{h}})&
  \mathcal{O}(b_{2,3}e^{-\frac{c}{h}})&\ldots
&\mathcal{O}(b_{m_{0}^{D}-1,m_{0}^{D}}e^{-\frac{c}{h}})\\
\vdots& \vdots&\vdots&\vdots&\vdots\\
b_{m_{1}^{D},1}&\mathcal{O}(b_{1,2}e^{-\frac{c}{h}})&
 \mathcal{O}(b_{2,3}e^{-\frac{c}{h}})&\ldots
&\mathcal{O}(b_{m_{0}^{D}-1,m_{0}^{D}}e^{-\frac{c}{h}})\\
\end{array}
\right)
$$
where we used $m_{0}^{D}$ (resp $m_{1}^{N}$, $m_{1}^{D}$) instead of
$m_{0}^{D}(\Omega_{+})$ (resp. $m_{1}^{N}(\Omega_{-})$,
$m_{1}^{D}(\Omega_{+})$)  and where
$k_{0}$ is a (possibly $h$-dependent) index such that
$|b_{k_{0},1}|=\max_{m_{1}^{N}+1\leq k\leq m_{1}^{D}}|b_{k,1}|$.
%
%
%
 By Gaussian elimination (see~\cite{Lep1} for more details), one can
find a matrix $R\in \mathcal{M}_{m_{1}^{D}}(\R)$ with $\|R\|=\mathcal{O}(e^{-\frac{c}{h}})$
such that
\begin{align*}
&(\Id_{m_{1}^{D}}+R) B=\tilde{B}=
\left(
\begin{array}{c|c}
 0(m_{0}^{D}-1,1) &\tilde{B}_{1,2}\\
0(m_{1}^{N}-m_{0}^{D}+1,1)&0(m_{1}^{N}-m_{0}^{D}+1, m_{0}^{D}-1)\\[3pt]
\hline\vspace{0.1cm}
\tilde{B}_{3,1}&0(m_{1}^{D}-m_{1}^{N},m_{0}^{D}-1)
\end{array}
\right)\\
\text{with }&
\tilde{B}_{3,1}=
\begin{pmatrix}
  b_{m_{1}^{N}+1,1}\\
\vdots\\
b_{m_{1}^{D},1}
\end{pmatrix}\text{ and }
\tilde{B}_{1,2}=
\begin{pmatrix}
  b_{1,2}(1+\mathcal{O}(e^{-\frac{c}{h}}))&0&\ldots&0\\
0&\ddots&&\vdots\\
\vdots&&\ddots&0\\
0&\ldots&0&b_{m_{0}^{D}-1,m_{0}^{D}}(1+\mathcal{O}(e^{-\frac{c}{h}}))
\end{pmatrix}\,,
\end{align*}
where $0(i,j)$ is the null matrix in $\mathcal{M}_{i,j}(\R)$\,.
We deduce that the singular values of $B$ are approximated (up to
exponentially small relative error terms) by the ones
of $\tilde{B}$ which are given by its block structure. We find (recall
that the singular values are labelled in decreasing order):
\begin{align*}
&
s_{j}(B)=|b_{m_{0}^{D}-j,m_{0}^{D}-j+1}|(1+\mathcal{O}(e^{-\frac{c}{h}}))\quad\text{for
}j\in 
  \left\{1,\ldots, m_{0}^{D}-1\right\}\\
\text{and }
&
s_{m_{0}^{D}}(B)^{2}=\left[\sum_{k=m_{1}^{N}+1}^{m_{1}^{D}}|b_{k,1}|^{2}\right]
(1+\mathcal{O}(e^{-\frac{c}{h}}))
\,.  
\end{align*}
We conclude the proof of Proposition~\ref{pr.singbeta} using the
approximate values of $b_{k,k+1}$ ($k \in \{1,\ldots,m_0^D-1\}$) and
$b_{k,1}$ ($k \in \{m_1^N+1,\ldots,m_1^D \}$)
given in Lemma~\ref{le.M12} and Proposition~\ref{pr.M21}:
$$|b_{m_{0}^{D}-j,m_{0}^{D}-j+1}|=\sqrt{\mu^{(0)}_{m_{0}^{D}-j+1}(\Omega_-)}(1+\mathcal{O}(e^{-\frac{c}{h}}))
\quad\text{for
}j\in 
  \left\{1,\ldots, m_{0}^{D}-1\right\},$$
$$\sum_{k=m_{1}^{N}+1}^{m_{1}^{D}}|b_{k,1}|^{2}=\frac{h^2 \sum_{k=m_{1}^{N}+1}^{m_{1}^{D}}\left(\int_{\partial\Omega_{+}}e^{-\frac{f(\sigma)}{h}}\mathbf{i}_{n}
    \psi_{k}^{(1)}(\sigma)~d\sigma\right)^2}{\int_{\Omega_{+}}e^{-2\frac{f(x)}{h}}~dx}
+\mathcal{O} \left(e^{-\frac{2\kappa_{f}+c}{h}} \right).$$
In particular, for $h$ small enough, we indeed have:
$$|b_{m_0^D-1,m_0^D}|^2 \ge \ldots \ge |b_{1,2}|^2 \ge \sum_{k=m_{1}^{N}+1}^{m_{1}^{D}}|b_{k,1}|^{2}$$
the last inequality being a consequence of~\eqref{eq.minlog} and $|b_{1,2}|^{2}=\mu_{2}^{(0)}(\Omega_-)(1+\mathcal{O}(e^{-\frac{c}{h}}))\geq
C_{\varepsilon}e^{-2\frac{\kappa_f - c_0}{h}}$
using Proposition~\ref{pr.gapO-}.
\end{proof}

\section{Proof of Theorem~\ref{th.main} and two corollaries}
\label{se.proof}
Proposition~\ref{pr.singbeta} already provides a precise asymptotic
result on the exponentially small eigenvalues of
$\Delta_{f,h}^{D,(0)}(\Omega_{+})$, using~\eqref{eq:VS_VP}:
\begin{align}
 \lambda_{j}^{(0)}(\Omega_{+})&=s_{m_{0}^{D}(\Omega_{+})+1-j}(\beta)^{2}=\mu_{j}^{(0)}(\Omega_-)
(1+\mathcal{O}(e^{-\frac{c}{h}}))\,,\quad \text{for}~j\in
\left\{2,\ldots, m_{0}^{D}(\Omega_{+})\right\}, 
\label{eq:expansion_lambdaj0} \\ 
\lambda_{1}^{(0)}(\Omega_{+})&=s_{m_{0}^{D}(\Omega_{+})}(\beta)^{2}=
\frac{
h^{2}\sum_{k=m_{1}^{N}(\Omega_{-})+1}^{m_{1}^{D}(\Omega_{+})}
\left(\int_{\partial\Omega_{+}}e^{-\frac{f(\sigma)}{h}}\mathbf{i}_{n}
\psi_{k}^{(1)}(\sigma) \,d\sigma\right)^{2}}{
\int_{\Omega_{+}}e^{-2\frac{f(x)}{h}}\,dx}
(1+\mathcal{O}(e^{-\frac{c}{h}})),\label{eq:expansion_lambda10_bis}
\end{align}
the second estimate being a consequence of Proposition~\ref{pr.M21} (see~\eqref{eq:expansion_bk1}).
This is essentially the result of Theorem~\ref{th.main} about
$\lambda_{1}^{(0)}(\Omega_{+})$
(see~Equation~\eqref{eq:expansion_lambda10}): it remains to show that the basis
$(\psi_{k}^{(1)})_{m_{1}^{N}(\Omega_{-})+1\leq k \leq
    m_{1}^{D}(\Omega_{+})}$ in~\eqref{eq:expansion_lambda10_bis}
  (which was introduced in
Section~\ref{se.treig}) can be replaced by {\em any} orthonormal basis
$(u_k^{(1)})_{1\leq k \leq
    m_{1}^{D}(\Omega_{+} \setminus \Omega_{-})}$ of $\Ran
1_{[0,\nu(h)]}(\Delta_{f,h}^{D,(1)}(\Omega_{+}\setminus
\overline{\Omega_{-}}))$. This will be done in Section~\ref{se.chanba}.

In addition, it also remains to prove the estimates~\eqref{eq.appu10} and~\eqref{eq.appdu10} on the eigenvector $u^{(0)}_1$
associated with the smallest eigenvalue
$\lambda_1^{(0)}(\Omega_+)$. This will be the subject of
Sections~\ref{se.appu10} and~\ref{se.appdfu}. We
recall that the spectral subspace associated with $\lambda_{1}^{(0)}(\Omega_{+})$
is one dimensional (since $\lambda_{2}^{(0)}(\Omega_{+})\geq
\lambda_{1}^{(0)}(\Omega_{+}) \, e^{\frac{c}{h}}$). We thus have
\begin{equation}
  \label{eq.defu10}
 u_{1}^{(0)}=\frac{\Pi_{0}v_{1}^{(0)}}{\|\Pi_{0}v_{1}^{(0)}\|_{L^2(\Omega_+)}},
\end{equation}
where $\Pi_0$ denotes the spectral projection associated with $\lambda_{1}^{(0)}(\Omega_{+})$:
\begin{equation}
  \label{eq.defpi0}
\Pi_{0}=1_{\{\lambda_{1}^{(0)}(\Omega_{+})\}}(\Delta_{f,h}^{D,(0)}(\Omega_{+})).
\end{equation}
The fact that $\Pi_{0}v_{1}^{(0)} \neq 0$ follows from the fact that
$\Pi_0 \Pi^{(0)} = \Pi_0$ and the estimate: for small $h$,
\begin{align}
\frac{\langle \Pi^{(0)} v_1^{(0)}, \Delta_{f,h}^{D,(0)}\Pi^{(0)} v_1^{(0)}  \rangle_{L^2(\Omega_+)}}{\|\Pi^{(0)} v_1^{(0)}\|_{L^2(\Omega_+)}^2}&=\frac{\|d_{f,h}\Pi^{(0)} v_1^{(0)}\|_{L^2(\Omega_+)}^2}{\|\Pi^{(0)}
  v_1^{(0)}\|_{L^2(\Omega_+)}^2}=\|\beta \Pi^{(0)}
v_1^{(0)}\|_{L^2(\Omega_+)}^2 (1+{\mathcal O}(e^{-\frac{c}{h}}))
\nonumber \\
&=\sum_{k=m_1^N(\Omega_-)+1}^{m_1^D(\Omega_+)} |b_{k,1}|^2
(1+{\mathcal O}(e^{-\frac{c}{h}}))\nonumber \\
&=\lambda_1^{(0)}(\Omega_+) (1+{\mathcal O}(e^{-\frac{c}{h}}))
\le \lambda_2^{(0)}(\Omega_+)  e^{-c/h} \label{eq:qr_dfhpiv}
\end{align}
for some positive constant $c$. The second and third equalities are consequences
of the almost orthonormality of the bases ${\mathcal B}^{(0)}$ and
${\mathcal B}^{(1),*}$ (see Proposition~\ref{pr.quasi}). The third one
comes from~\eqref{eq:expansion_lambda10_bis}
and~\eqref{eq:expansion_bk1}. The last inequality is a consequence of Equation~\eqref{eq:expansion_lambdaj0} and Proposition~\ref{pr.gapO-}.

Finally, Section~\ref{se.corol} is devoted to two corollaries of Theorem~\ref{th.main}.

\subsection{Approximation of $u_{1}^{(0)}$}
\label{se.appu10}

Let us first prove the estimate~\eqref{eq.appu10} on $u_{1}^{(0)}$.
\begin{proposition}\label{prop:u10}
  There exists $c>0$ such that
$$
\left\|u_{1}^{(0)}-\frac{e^{-\frac{f}{h}}}{\left(\int_{\Omega_{+}}e^{-2\frac{f(x)}{h}}~dx\right)^{1/2}}\right\|_{W^{2,2}(\Omega_{+})}
=\mathcal{O}(e^{-\frac{c}{h}})\,.  
$$
\end{proposition}
\begin{proof}
Since $\left\|v_{1}^{(0)}-\frac{e^{-\frac{f}{h}}}{\left(\int_{\Omega_{+}}e^{-2\frac{f(x)}{h}}~dx\right)^{1/2}}
\right\|_{W^{2,2}(\Omega_{+})}=\mathcal{O}(e^{-\frac{c}{h}})$ (which
is a simple consequence of Lemma~\ref{le.intminmaj}),
it suffices to prove
$\|u_{1}^{(0)}-v_{1}^{(0)}\|_{W^{2,2}(\Omega_{+})}=\mathcal{O}(e^{-\frac{c}{h}})$.

Let us first prove the result in the $L^2(\Omega_+)$-norm. From~\eqref{eq:qr_dfhpiv}, we have $\|d_{f,h}\Pi^{(0)}
v_1^{(0)}\|_{L^2(\Omega_+)}^2 \le \lambda_2^{(0)}(\Omega_+) e^{-c/h}$
and thus
\begin{align*}
\lambda_{2}^{(0)}(\Omega_+)
\left\|1_{[\lambda_{2}^{(0)}(\Omega_{+}),+\infty)}(\Delta_{f,h}^{D,(0)}(\Omega_{+}))
  \Pi^{(0)}  v_{1}^{(0)}\right\|_{L^{2}(\Omega_{+})}^{2}&\leq \langle \Pi^{(0)}v_{1}^{(0)}\,,\,
\Delta_{f,h}^{D,(0)}(\Omega_{+}) \Pi^{(0)}v_{1}^{(0)}\rangle_{L^2(\Omega_+)}\\
&\le \lambda_2^{(0)}(\Omega_+) e^{-c/h}.
\end{align*}
Since $\Pi_0=\Pi_0 \Pi^{(0)}$, we deduce
$$
 \left\|\Pi_{0}v_{1}^{(0)}-\Pi^{(0)}  v_{1}^{(0)}\right\|_{L^{2}(\Omega_{+})}
=\left\|1_{[\lambda_{2}^{(0)}(\Omega_{+}),+\infty)}(\Delta_{f,h}^{D,(0)}(\Omega_{+})) \Pi^{(0)} v_{1}^{(0)}\right\|_{L^{2}(\Omega_{+})}
=\mathcal{O}(e^{-\frac{c}{h}}).
$$
Using in addition the fact that $\left\|\Pi^{(0)}v_{1}^{(0)}-
  v_{1}^{(0)}\right\|_{L^{2}(\Omega_{+})}=
\mathcal{O}(e^{-\frac{c}{h}})$ and $\left\|
  v_{1}^{(0)}\right\|_{L^{2}(\Omega_{+})} = 1 + \mathcal{O}(e^{-\frac{c}{h}})$ (see Proposition~\ref{pr.quasi}), this proves
\begin{equation}\label{eq:L2}
\|u_{1}^{(0)}-v_{1}^{(0)}\|_{L^{2}(\Omega_{+})}=\mathcal{O}(e^{-\frac{c'}{h}})\,.
\end{equation}


The estimate in the $W^{2,2}(\Omega_+)$-norm is then obtained by a
 bootstrap argument, that will be used many times again
below.  The following equations hold: 
$$
\left\{
  \begin{array}[c]{l}
    \Delta_{f,h}^{(0)}u_{1}^{(0)}=\lambda_{1}^{(0)}(\Omega_{+})\,u_{1}^{(0)},\\
   u_{1}^{(0)}\big|_{\partial\Omega_{+}}=0,
  \end{array}
\right.
\text{ and }
\left\{
  \begin{array}[c]{l}
    \Delta_{f,h}^{(0)}v_{1}^{(0)}=g_h,\\
   v_{1}^{(0)}\big|_{\partial\Omega_{+}}=0,
  \end{array}
\right.
$$
where $g_h$ is {\em defined} by the equation $g_h=\Delta_{f,h}^{(0)}v_{1}^{(0)} $ and, using the same arguments
as in the proof of~\eqref{eq.estim_delta_v10},
$\|g_h\|_{L^2(\Omega_+)}=\mathcal{O}(e^{-\frac{\kappa_{f}-C_0\delta_{+}}{h}})$.  Recall that by the assumption~\eqref{eq:hyp_delta+}, $\delta_+$ is small enough so that $C_0 \delta_+
  < \kappa_f$, and thus $\|g_h\|_{L^2(\Omega_+)}=\mathcal{O}(e^{-\frac{c}{h}})$. 
 We then deduce that $(u_{1}^{(0)}-v_{1}^{(0)})$ solves, $\Delta_{\textrm{H}}$ denoting the Hodge Laplacian~\eqref{eq:Hodge}:
$$
\left\{
\begin{array}[c]{l}
 \Delta_{\textrm{H}}^{(0)}(u_{1}^{(0)}-v_{1}^{(0)})=\tilde{g}_h,\\
(u_{1}^{(0)}-v_{1}^{(0)})\big|_{\partial\Omega_{+}}=0.
\end{array}
\right.
$$
Again, $\tilde{g}_h$ is defined by the first equation. Using the
 formula~\eqref{eq:hodge_witten} which relates the Hodge and the
 Witten Laplacians and the estimate~\eqref{eq:L2}, it holds
$\|\tilde{g}_h\|_{L^2(\Omega_+)}=\mathcal{O}(e^{-\frac{c'}{h}})$.
The elliptic regularity of the Dirichlet Hodge Laplacian then implies
$\|u_{1}^{(0)}-v_{1}^{(0)}\|_{W^{2,2}(\Omega_{+})}={\mathcal{O}}(e^{-\frac{c'}{h}})$.

\end{proof}

\subsection{Approximation of $d_{f,h}u_{1}^{(0)}$}
\label{se.appdfu}

We now consider $d_{f,h}u_{1}^{(0)}$. In this section, we will first
prove~\eqref{eq.appdu10} using for the $u_k^{(1)}$'s the special basis
considered in Section~\ref{se.matrix}. This will be generalized to any
orthonormal basis of $\Ran
1_{[0,\nu(h)]}(\Delta_{f,h}^{D,(1)}(\Omega_{+}\setminus
\overline{\Omega_{-}}))$ in the next section.

Let us start with an estimate in the $L^2(\Omega_+)$-norm.
\begin{proposition}
\label{pr.dfL2}
Let $\mathcal{B}_{1}^{*}=(w_{k})_{1\leq k\leq m_{1}^{D}(\Omega_{+})}$
be the basis of $F^{(1)}=\Ran
1_{[0,\nu(h)]}(\Delta_{f,h}^{D,(1)}(\Omega_{+}))$, dual (in $L^2(\Omega_+)$) to
$\mathcal{B}_{1}=(\Pi^{(1)}v_{k}^{(1)})_{1\leq k\leq
  m_{1}^{D}(\Omega_{+})}$\,.
Then the eigenvector
$u_{1}^{(0)}$ of $\Delta_{f,h}^{D,(0)}(\Omega_{+})$ given  by \eqref{eq.defu10} satisfies
\begin{equation}\label{eq.dfL2}
\left\|d_{f,h}u_{1}^{(0)}-\sum_{k=m_{1}^{N}(\Omega_{-})+1}^{m_{1}^{D}(\Omega_{+})}b_{k,1}w_{k}
  \right\|_{L^{2}(\Omega_{+})}
= \mathcal{O}(e^{-\frac{\kappa_{f}+c}{h}})\,,
\end{equation}
for some $c>0$ and where the
coefficients $b_{k,1}$'s are defined by~\eqref{eq:bkj}.
\end{proposition}
\begin{proof}
By definition of the matrix
$B=M(\beta,\mathcal{B}^{(0)},\mathcal{B}^{(1)*})$,
$$
d_{f,h} (\Pi^{(0)}v_{1}^{(0)})=\beta
(\Pi^{(0)}v_{1}^{(0)})=\sum_{k=1}^{m_{1}^{D}(\Omega_{+})}b_{k,1}w_{k}=\sum_{k=m_1^N(\Omega_-)+1}^{m_{1}^{D}(\Omega_{+})}b_{k,1}w_{k}
+ r_h,
$$
with
$\|r_h\|_{L^2(\Omega_+)}=\mathcal{O}(e^{-\frac{\kappa_{f}+c}{h}})$,
this estimate being a consequence of the almost orthonormality of the
one-forms $w_k$, and of Lemma~\ref{le.M11}. Equation~\eqref{eq.dfL2}
is thus equivalent to:
$$\left\|d_{f,h}\left( u_{1}^{(0)}- \Pi^{(0)}v_{1}^{(0)}\right)
  \right\|_{L^{2}(\Omega_{+})}
= \mathcal{O}(e^{-\frac{\kappa_{f}+c}{h}}).$$
Notice that
$$ u_{1}^{(0)}-
\Pi^{(0)}v_{1}^{(0)}=\|\Pi_0v_1^{(0)}\|_{L^2(\Omega_+)}^{-1} \left(
  \Pi_0 - \Pi^{(0)} \right) v_{1}^{(0)} 
+\left(\|\Pi_0v_1^{(0)}\|_{L^2(\Omega_+)}^{-1}-1\right) \Pi^{(0)}v_{1}^{(0)}.$$
We recall that
$\|\Pi_{0}v_{1}^{(0)}\|_{L^2(\Omega_+)}=1+\mathcal{O}(e^{-\frac{c}{h}})$
and $\| d_{f,h}  \Pi^{(0)}v_{1}^{(0)}\|_{L^2(\Omega_+)}=\| \beta
\Pi^{(0)}v_{1}^{(0)}\|_{L^2(\Omega_+)}=\tilde{\mathcal
  O}(e^{-\frac{\kappa_f}{h}})$ (see~\eqref{eq:qr_dfhpiv}). This implies that
$$\left\|d_{f,h}\left( u_{1}^{(0)}- \Pi^{(0)}v_{1}^{(0)}\right)
  \right\|_{L^{2}(\Omega_{+})}
= \left\|d_{f,h}\left(  (\Pi_0-\Pi^{(0)})v_{1}^{(0)} \right)
  \right\|_{L^{2}(\Omega_{+})} (1+\mathcal{O}(e^{-\frac{c}{h}})) +  \mathcal{O}(e^{-\frac{\kappa_{f}+c}{h}}).$$
Moreover, 
using the fact that $\Pi_0 \Pi^{(0)}=\Pi_0$ and $\Pi^{(0)}
-\Pi_0=1_{[\lambda_{2}^{(0)}(\Omega_{+}),+\infty)}(\Delta_{f,h}^{D,(0)}(\Omega_{+}))$
commutes with $\Delta_{f,h}^{D,(0)}(\Omega_{+})$,
\begin{align*}
  \|d_{f,h}( (\Pi_0-\Pi^{(0)})v_{1}^{(0)})\|_{L^{2}(\Omega_{+})}^{2}&= \langle (\Pi^{(0)}-\Pi_{0})v_{1}^{(0)}, \Delta_{f,h}^{D,(0)}(\Omega_{+})(\Pi^{(0)}-\Pi_{0})v_{1}^{(0)}\rangle_{L^2(\Omega_+)}
\\
&=\|\beta
\Pi^{(0)}v_{1}^{(0)}\|_{L^{2}(\Omega_{+})}^{2}-\lambda_{1}^{(0)}(\Omega_+)\|\Pi_0
v_{1}^{(0)}\|^{2}_{
    L^{2}(\Omega_+)}\\
& = \lambda_{1}^{(0)}(\Omega_+) (1+\mathcal{O}(e^{-\frac{c}{h}})) -\lambda_1^{(0)}(\Omega_+) (1+\mathcal{O}(e^{-\frac{c}{h}}))\\
&= \mathcal{O}(e^{-2\frac{\kappa_{f}+c'}{h}}).
\end{align*}
The third equality is obtained from~\eqref{eq:qr_dfhpiv} and the last
one from the estimate on the bottom of the spectrum in
Lemma~\ref{le.infsigD0}. This concludes the proof of~\eqref{eq.dfL2}.
\end{proof}
To perform a bootstrap argument to extend the previous result to
stronger norms, we need the intermediate lemma:
\begin{lemme}
  \label{le.L2W}
For any $n\in \nz$, there exists $C_{n}>0$ and $N_{n}\in\nz$, such
that
$$
\forall u\in F^{(1)}=\Ran 1_{[0,\nu(h)]}(\Delta_{f,h}^{D,(1)}(\Omega_{+}))\,,\quad
\|u\|_{W^{n,2}(\Omega_{+})}\leq C_{n}h^{-N_{n}}\|u\|_{L^{2}(\Omega_{+})}\,.
$$
\end{lemme}
\begin{proof}
Let us introduce an orthonormal basis $(e_{k})_{1\leq k\leq
  m_{1}^{D}(\Omega_{+})}$ of eigenvectors of $\Delta_{f,h}^{D,(1)} (\Omega_+)$
associated with the small eigenvalues $\lambda_{k}^{(1)}(\Omega_+)\leq \nu(h)$:
$\Delta_{f,h}^{D,(1)}e_{k}=\lambda_{k}^{(1)}e_{k}$. We have 
$\|d_{f,h}e_{k}\|^{2}_{L^{2}(\Omega_+)}+\|d_{f,h}^{*}e_{k}\|_{L^{2}(\Omega_{+})}^{2}
= \lambda_{k}^{(1)} \le \nu(h)$.
For any $u \in
F^{(1)}$, there exists some reals $(u_k)_{1 \le k \le m_1^D(\Omega_+)}$
  such that
$$
u=\sum_{k=1}^{m_{1}^{D}(\Omega_{+})}u_{k}e_{k}\quad \text{with}\quad
\sum_{k=1}^{m_{1}^{D}(\Omega_{+})}|u_{k}|^{2}=\|u\|_{L^{2}(\Omega_{+})}^{2}.
$$
Lemma~\ref{le.L2W} will be proven if one can show that for all $n \in
\nz$,  there exists $C_{n}>0$ and $N_{n}\in\nz$, such
that for all $k \in \{1, \ldots,  m_1^D(\Omega_+)\}$, $\|e_{k}\|_{W^{n,2}(\Omega_{+})}\leq
C_{n}h^{-N_{n}}$.
From 
$$
4\||\nabla f|\,e_{k}\|^{2}_{L^{2}(\Omega_{+})}+2\|d_{f,h}e_{k}\|^{2}_{L^{2}(\Omega_+)}+2\|d_{f,h}^{*}e_{k}\|_{L^{2}(\Omega_{+})}^{2}\geq h^{2}
\left[\|d e_{k}\|_{L^{2}(\Omega_{+})}^{2}+\|d^{*}
e_{k}\|_{L^{2}(\Omega_{+})}^{2}
\right]
$$
(which is obtained from the formulas~\eqref{eq:dfh_d} and~\eqref{eq:dfh*_d*} which
relate $d_{f,h}$ to $d$ and $d_{f,h}^*$ to $d^*$)
we deduce $\|e_{k}\|_{W^{1,2}(\Omega_{+})}\leq C h^{-1}$\,. 
Then the equation $\Delta_{f,h}^{D,(1)}(\Omega_{+})e_{k}=\lambda_{k}^{(1)}e_{k}$
can be written
$$
\left\{
\begin{array}[c]{l}
\Delta^{(1)}_H e_{k}=r_{k}(h)\\
\mathbf{t}e_{k}\big|_{\partial \Omega_{+}}=0,\,
\mathbf{t}d^{*}e_{k}\big|_{\partial \Omega_{+}}= \rho_{k}(h), 
\end{array}
\right.
$$
with $\|r_{k}(h)\|_{L^{2}(\Omega_{+})}
+\|\rho_{k}(h)\|_{W^{1/2,2}(\partial \Omega_{+})}=
\mathcal{O}(h^{-2})$. The estimate on $\rho_k(h)$ is a consequence of
$0= \mathbf{t}d^*_{f,h} e_k = h \mathbf{t} d^* e_k +\mathbf{i}_{\nabla
  f} e_k$ so that $ \|\rho_{k}(h)\|_{W^{1/2,2}(\partial \Omega_{+})}=h^{-1}\|\mathbf{i}_{\nabla
  f} e_k
\|_{W^{1/2,2}(\partial \Omega_{+})}\le C h^{-1}
\|e_k\|_{W^{1,2}(\Omega_+)} \le C' h^{-2}$. The estimate on $r_k(h)$
comes from the relation~\eqref{eq:hodge_witten} between the Hodge and the
 Witten Laplacians. The
elliptic regularity of the above system (see for example \cite[Theorem 2.2.6]{Schz}) 
implies $\|e_{k}\|_{W^{2,2}(\Omega_{+})}=\mathcal{O}(h^{-2})$. Finally, the
result for a general $n\in\nz$ is obtained by a bootstrap argument.
\end{proof}
We are now in position to restate the result of Proposition~\ref{pr.dfL2}
in terms of the $W^{n,2}(\mathcal{V})$-norm.
\begin{proposition}
  \label{pr.dfpsi}
Let $(\psi_{k}^{(1)})_{m_{1}^{N}(\Omega_{-})+1\leq k\leq
  m_{1}^{D}(\Omega_{+})}$ be the orthonormal basis of eigenvectors
chosen in Section~\ref{se.treig} and let $\chi_{+}$ be the cut-off
function of Definition~\ref{de.quasimodes}.   For any
$n\in\nz$\,, there exists a constant $C_{n}>0$ such that
$$
\left\|d_{f,h}u_{1}^{(0)}-\sum_{k=m_{1}^{N}(\Omega_{-})+1}^{m_{1}^{D}(\Omega_{+})}
b_{k,1}\psi_{k}^{(1)} \right\|_{W^{n,2}(\mathcal{V})}\leq C_{n}e^{-\frac{\kappa_{f}+c}{h}}\,.
$$
where $\mathcal{V}$ is any neighborhood of
$\partial \Omega_{+}$ contained in $\left\{\chi_{+}=1\right\}$, $c$ is
positive constant and, we recall, the
coefficients $b_{k,1}$'s defined by~\eqref{eq:bkj} satisfy (see Proposition~\ref{pr.M21}):
$$
b_{k,1}=-\frac{h\int_{\partial\Omega_{+}} e^{-\frac{f(\sigma)}{h}}\mathbf{i}_{n}
  \psi_{k}^{(1)}(\sigma)~d\sigma}{\left(\int_{\Omega_{+}}e^{-2\frac{f(x)}{h}}~dx\right)^{1/2}}
+ \mathcal{O}(e^{-\frac{\kappa_f+c}{h}}).
$$
\end{proposition}
\begin{proof}
From Proposition~\ref{pr.dfL2} and Lemma~\ref{le.L2W}, we deduce 
$$
\left\|d_{f,h}u_{1}^{(0)}-\sum_{k=m_{1}^{N}(\Omega_{-})+1}^{m_{1}^{D}(\Omega_{+})}b_{k,1}w_{k}
\right\|_{W^{n,2}(\Omega_{+})}\leq 
C_{n}h^{-N_{n}}e^{-\frac{\kappa_{f}+c}{h}}\leq C_{n}'e^{-\frac{\kappa_{f}+c/2}{h}}\,.
$$
Since, by the almost orthonormality of the family
$(\Pi^{(1)}v_{k}^{(1)})_{1 \le k \le m_1^D(\Omega_+)}$,  $\|w_{k}-\Pi^{(1)}v_{k}^{(1)}\|_{L^{2}(\Omega_{+})}=\mathcal{O}(e^{-\frac{c}{h}})$ and
$\max\left\{|b_{k,1}|, m_{1}^{N}(\Omega_{-})+1 \le k \le
  m_1^D(\Omega_+)\right\}=\tilde{\mathcal{O}}(e^{-\frac{\kappa_{f}}{h}})$
(see Proposition~\ref{pr.M21}),
Lemma~\ref{le.L2W} also leads to
$$
\left\|d_{f,h}u_{1}^{(0)}-\sum_{k=m_{1}^{N}(\Omega_{-})+1}^{m_{1}^{D}(\Omega_{+})}b_{k,1}\Pi^{(1)}v_{k}^{(1)}
\right\|_{W^{n,2}(\Omega_{+})}\leq  C_{n}''e^{-\frac{\kappa_{f}+c/2}{h}}\,.
$$
By recalling the definition of
$v_{k}^{(1)}=\chi_{+}\psi_{k}^{(1)}$\,, it suffices now to check that
$\|v_{k}^{(1)}-\Pi^{(1)}v_{k}^{(1)}\|_{W^{n,2}(\Omega_{+})}$ is of order
$\mathcal{O}(e^{-\frac{c'}{h}})$ for some $c'>0$. We already know
$$
\|v_{k}^{(1)}-\Pi^{(1)}v_{k}^{(1)}\|_{L^{2}(\Omega_{+})}=\mathcal{O}(e^{-\frac{c}{h}})\,,
$$
from Proposition~\ref{pr.quasi}. 

For the $W^{1,2}(\Omega_+)$ estimates, notice that
$\|d_{f,h}v_{k}^{(1)}\|_{L^{2}(\Omega_{+})}^{2}
+\|d_{f,h}^{*}v_{k}^{(1)}\|^{2}_{L^{2}(\Omega_{+})}=\langle v_k^{1},
\Delta_{f,h}^{D,(1)}(\Omega_+) v_k^{(1)} \rangle_{L^2(\Omega_+)}= \mathcal{O}(e^{-\frac{c}{h}})$
(again from Proposition~\ref{pr.quasi}) while $\Pi^{(1)}v_{k}^{(1)}\in
F^{(1)} = \Ran 1_{[0,\nu(h)]}(\Delta_{f,h}^{D,(1)}(\Omega_{+}))$ implies
$$
\|d_{f,h}\Pi^{(1)}v_{k}^{(1)}\|_{L^{2}(\Omega_{+})}^{2}
+\|d_{f,h}^{*}\Pi^{(1)}v_{k}^{(1)}\|^{2}_{L^{2}(\Omega_{+})}=\langle \Pi^{(1)}v_k^{1},
\Delta_{f,h}^{D,(1)}(\Omega_+) \Pi^{(1)}v_k^{(1)} \rangle_{L^2(\Omega_+)}
=\mathcal{O}(e^{-\frac{c}{h}}).
$$
We deduce
\begin{align*}
&\|d(v_{k}^{(1)}-\Pi^{(1)}v_{k}^{(1)})\|_{L^{2}(\Omega_{+})}^{2}+ 
\|d^{*}(v_{k}^{(1)}-\Pi^{(1)}v_{k}^{(1)})\|_{L^{2}(\Omega_{+})}^{2}
\\
&\leq
\frac{2}{h^{2}}\left[\|d_{f,h}(v_{k}^{(1)}-\Pi^{(1)}v_{k}^{(1)})\|_{L^{2}(\Omega_{+})}^{2}
+\|d_{f,h}^{*}(v_{k}^{(1)}-\Pi^{(1)}v_{k}^{(1)})\|^{2}_{L^{2}(\Omega_{+})}+
2\||\nabla f|\, (v_{k}^{(1)}-\Pi^{(1)}v_{k}^{(1)})\|_{L^{2}(\Omega_{+})}^{2}\right]\\
&\leq \frac{Ce^{-2\frac{c}{h}}}{h^{2}}\,.
\end{align*}
This gives the $W^{1,2}$-estimate
$\|v_{k}^{(1)}-\Pi^{(1)}v_{k}^{(1)}\|_{W^{1,2}(\Omega_{+})}=\tilde{\mathcal{O}}(e^{-\frac{c}{h}})$.

The $W^{n,2}$-estimates ($n \ge 2$) are then obtained by an argument
based on the elliptic regularity of the (non-homogeneous) Dirichlet Hodge Laplacian.
On the one hand,
$\|\Pi^{(1)}v_{k}^{(1)}\|_{L^2(\Omega_+)}=1 + {\mathcal
  O}(e^{-\frac{c}{h}})$, $\Pi^{(1)}v_{k}^{(1)} \in F^{(1)}$ and
$\|\Delta_{f,h}^{D,(1)}|_{F^{(1)}}\| = {\mathcal O}(e^{-\frac{c}{h}})$
(see Proposition~\ref{pr.number}) imply
$\|\Delta_{f,h}^{D,(1)}\Pi^{(1)}v_{k}^{(1)}\|_{L^2(\Omega_+)}=\mathcal{O}(e^{-\frac{c}{h}})$.
 Lemma~\ref{le.L2W} can then be used to obtain 
$\|\Delta_{f,h}^{D,(1)}\Pi^{(1)}v_{k}^{(1)}\|_{W^{n,2}(\Omega_+)}=\tilde{\mathcal{O}}(e^{-\frac{c}{h}})$
for any integer~$n$. Here, $\|\Delta_{f,h}^{D,(1)}|_{F^{(1)}}\|=\sup_{
  u \in F^{(1)}} \frac{\|\Delta_{f,h}^{D,(1)} u\|_{L^2(\Omega_+)}}{\| u\|_{L^2(\Omega_+)}}$ is simply the spectral radius of the
finite-dimensional operator $\Delta_{f,h}^{D,(1)}: F^{(1)} \to F^{(1)}$.
On the other hand, Lemma~\ref{lem:dec_exp_Cinfty} below implies
$\|\Delta_{f,h}^{D,(1)}v_{k}^{(1)}\|_{W^{n,2}(\Omega_+)}=\|\Delta_{f,h}^{D,(1)}(\chi_{+}\psi_{k}^{(1)}) \|_{W^{n,2}(\Omega_+)}=\mathcal{O}(e^{-\frac{c}{h}})$
for any integer~$n$, using the arguments of the proofs of
Proposition~\ref{pr.almostorhtO-} or~\ref{pr.decayO+-} to get the estimate on the truncated eigenvector from the
exponential decay of the eigenvector. Thus, for $n \ge 1$, if
$\|(v_{k}^{(1)}-\Pi^{(1)}v_{k}^{(1)})\|_{W^{n,2}(\Omega_{+})}=\tilde{\mathcal{O}}(e^{-\frac{c}{h}})$,
then the difference
$v_{k}^{(1)}-\Pi^{(1)}v_{k}^{(1)}$ satisfies:
$$
\left\{
\begin{array}[c]{l}
\Delta_H^{(1)} (v_{k}^{(1)}-\Pi^{(1)}v_{k}^{(1)})=r_{k}(h)\,,\\
\mathbf{t}(v_{k}^{(1)}-\Pi^{(1)}v_{k}^{(1)})=0, \, \mathbf{t}d^{*}(v_{k}^{(1)}-\Pi_{k}^{(1)}v_{k}^{(1)})=\varrho_{k}(h),
\end{array}
\right.
$$
with $\|r_{k}(h)\|_{W^{n,2}(\Omega_{+})}=\tilde{\mathcal{O}}(e^{-\frac{c}{h}})$ and
$\|\varrho_{k}(h)\|_{W^{n-1/2,2}(\Omega_{+})}=\tilde{\mathcal{O}}(e^{-\frac{c}{h}})$.\\
This
implies $\|(v_{k}^{(1)}-\Pi^{(1)}v_{k}^{(1)})\|_{W^{n+2,2}(\Omega_{+})}=\tilde{\mathcal{O}}(e^{-\frac{c}{h}})$.
A bootstrap argument (induction on $n$) thus shows that for any $n$,
$\|v_{k}^{(1)}-\Pi^{(1)}v_{k}^{(1)}\|_{W^{n,2}(\Omega_{+})}=\tilde{\mathcal{O}}(e^{-\frac{c}{h}})\leq
\mathcal{O}(e^{-\frac{c'}{h}})$ for any $c'<c$.
\end{proof}
We end this section with an estimate on the exponential decay (in a
neighborhood of $\partial \Omega_-$) of the
eigenvectors of $\Delta^{D,(1)}_{f,h}(\Omega_+ \setminus \overline{\Omega_-})$, in ${\mathcal C}^\infty$
norm. This is a refinement of Proposition~\ref{prop:exp_decay}, which
was needed in the previous proof.
\begin{lemme}\label{lem:dec_exp_Cinfty}
For every $\varepsilon\in (0,1)$, there exists a function
$\varphi_{\varepsilon}\in
\mathcal{C}^{\infty}_{0}(\Omega_{+}\setminus\overline{\Omega}_{-})$
such that for all $x\in \overline{\Omega_{+}}\setminus\Omega_{-}$,
\begin{align*}
&
  \quad
  |\nabla \varphi_{\varepsilon}(x)|\leq (1-\varepsilon)|\nabla
  f(x)|\,,\\
& \left(d(x,\partial \Omega_{+}\cup \partial \Omega_{-})\leq
\frac{\varepsilon}{2}\right)\Rightarrow (\varphi_{\varepsilon}(x)=0)\,,\\
&\varphi_{\varepsilon}(x)\geq 0 \text{ and } d_{Ag}(x,\partial
\Omega_{+}\cup \partial \Omega_{-})-C\varepsilon\leq
\varphi_{\varepsilon}(x)\,,
\end{align*}
where $C>0$ is a constant independent of $\varepsilon$\,.
For every $n\in \nz$\,, and once $\varphi_{\varepsilon}$ is fixed, 
 there exists $C_{\varepsilon,n}>0$ and $N_{n}>0$ 
 independent of $h \in [0,h_0]$ such that
every normalized eigenvector $\psi$ of
$\Delta_{f,h}^{D,(1)}(\Omega_{+}\setminus\overline{\Omega_{-}})$
associated with an eigenvalue $\lambda \in [0,\nu(h)]$ satisfies
$$
\|e^{\frac{\varphi_{\varepsilon}}{h}}\psi\|_{W^{n,2}(\Omega_{+}\setminus\overline{\Omega_{-}})}
\leq C_{\varepsilon,n} h^{-N_n}.
$$
\end{lemme}
As explained in the proof, we cannot state this result with $\varphi_\varepsilon$
 equals to the Agmon distance to a neighborhood of $\partial \Omega_+$ as in
 Proposition~\ref{prop:exp_decay} because
 the Agmon distance is not a sufficiently regular function.
\begin{proof}
The function $\varphi_{\varepsilon}\in {\mathcal
  C}^{\infty}_{0}(\Omega_{+}\setminus\overline{\Omega}_{-})$  is built as an accurate enough mollified version of 
$\theta_{\varepsilon}(x)=(1-2\varepsilon)d_{Ag}(x,\mathcal{V}_{+}^{\varepsilon}\cup
\mathcal{V}_{-}^{\varepsilon})$ where 
$$
\mathcal{V}_{\pm}^{\varepsilon}=\left\{x\in
  \overline{\Omega_{+}}\setminus\Omega_{-}\,,\quad d(x,\partial
  \Omega_{\pm})\leq \varepsilon\right\}\,.
$$
Indeed, the function $\theta_{\varepsilon}$ is a Lipschitz function such that
\begin{align*}
  & |\nabla \theta_{\varepsilon}(x)|\leq (1-2\varepsilon)|\nabla
  f(x)|\quad\text{a.e.}~\,,\\
& \left(d(x,\partial \Omega_{+}\cup \partial \Omega_{-})\leq
\varepsilon\right)\Rightarrow (\theta_{\varepsilon}(x)=0)\,,\\
& d(x,\partial \Omega_{+}\cup \partial\Omega_{-})-C_{1}\varepsilon
\leq \theta_{\varepsilon}(x)\leq d(x,\partial \Omega_{+}\cup \partial \Omega_{-})\,,
\end{align*}
hold in $\overline{\Omega_{+}}\setminus \Omega_{-}$, with $C_{1}\geq
0$ independent of $\varepsilon$\,. Since
$\theta_{\varepsilon}$ fulfills uniform Lipschitz estimates and
$|\nabla f(x)|\geq c>0$  on $\overline{\Omega_{+}}\setminus
\Omega_{-}$\,, all the properties of $\varphi_{\varepsilon}$ are obtained
by considering the convolution of $\theta_{\varepsilon}$ with a
mollifier with a sufficiently small compact support.
We cannot simply take $\varphi_\varepsilon=d_{Ag}(x,\partial
\Omega_{+}\cup \partial \Omega_{-})$ or even $\varphi_\varepsilon=d_{Ag}(.,
  \mathcal{V}_{+}^{\varepsilon}\cup \mathcal{V}_{-}^{\varepsilon})$ because
  the argument requires to consider high order
  derivatives of
  $\varphi_\varepsilon$\,.

 Let $\psi$ be a normalized eigenvector of
$\Delta_{f,h}^{D,(1)}(\Omega_{+}\setminus\overline{\Omega_{-}})$
associated with an eigenvalue $\lambda \in [0,\nu(h)]$. 
We already know from
  Proposition~\ref{prop:exp_decay} that
\begin{equation}\label{eq:W12estim}
\|e^{\frac{\varphi_{\varepsilon}}{h}}\psi\|_{W^{1,2}(\Omega_{+}\setminus\overline{\Omega_{-}})}\leq C_{\varepsilon} h^{-1}.
\end{equation}
The argument to obtain the estimates in
$W^{n,2}(\Omega_{+}\setminus\overline{\Omega_{-}})$-norms is based on a
bootstrap argument, using the elliptic regularity of 
non-homogeneous Dirichlet boundary problems for the Hodge Laplacian.

Indeed, we have:
\begin{align*}
e^{-\frac{\varphi_{\varepsilon}}{h}} \Delta_{f,h} e^{\frac{\varphi_{\varepsilon}}{h}} &=
\Delta_{f,h} - h {\mathcal L}_{\nabla \varphi_{\varepsilon}} + h  {\mathcal L}_{\nabla
  \varphi_{\varepsilon}}^* - |\nabla \varphi_{\varepsilon}|^2
\end{align*}
and thus
\begin{align*}
\Delta_{f,h} (e^{\frac{\varphi_{\varepsilon}}{h}} \psi) &=
\lambda e^{\frac{\varphi_{\varepsilon}}{h}} \psi - h e^{\frac{\varphi_{\varepsilon}}{h}} {\mathcal
  L}_{\nabla \varphi_{\varepsilon}} \psi + h  e^{\frac{\varphi_{\varepsilon}}{h}} {\mathcal L}_{\nabla
  \varphi_{\varepsilon}}^*  \psi - |\nabla \varphi_{\varepsilon}|^2 e^{\frac{\varphi_{\varepsilon}}{h}}\psi. 
\end{align*}
Using the fact that $\Delta_{f,h}= h^2 (d d^*+d^*d) + h( {\mathcal
  L}_{\nabla f} + {\mathcal L}_{\nabla f}^* ) + |\nabla f|^2$, we obtain
\begin{equation}\label{eq:deltaHv}
\Delta_H v=h^{-2} \left( \lambda v - h e^{\frac{\varphi_{\varepsilon}}{h}} {\mathcal
  L}_{\nabla \varphi_{\varepsilon}} e^{-\frac{\varphi_{\varepsilon}}{h}} v + h  e^{\frac{\varphi_{\varepsilon}}{h}} {\mathcal L}_{\nabla
  \varphi_{\varepsilon}}^*  e^{-\frac{\varphi_{\varepsilon}}{h}} v - |\nabla \varphi_{\varepsilon}|^2 v -  h  {\mathcal
  L}_{\nabla f} v - h {\mathcal L}_{\nabla f}^*v - |\nabla f|^2 v
 \right)
\end{equation}
where 
$$v=e^{\frac{\varphi_{\varepsilon}}{h}} \psi.$$ 
For the boundary conditions, we have
of course
\begin{equation}\label{eq:BC1}
\mathbf{t} v = 0,
\end{equation}
and
\begin{align*}
0=\mathbf{t} d^*_{f,h} \psi= e^{\frac{\varphi_{\varepsilon}}{h}} \mathbf{t} d^*_{f,h} \psi= \mathbf{t} d^*_{f,h} e^{\frac{\varphi_{\varepsilon}}{h}}  \psi +  e^{\frac{\varphi_{\varepsilon}}{h}}\mathbf{t}
\mathbf{i}_{\nabla \varphi_{\varepsilon}} \psi.
\end{align*}
The condition $\varphi_{\varepsilon}=0$ in a neighborhood of $\partial
\Omega_{+}\cup \partial \Omega_{-}$ implies $\nabla \varphi_{\varepsilon} = 0$ on
$\partial \Omega_+ \cup \partial \Omega_-$ and thus $\mathbf{t}
\mathbf{i}_{\nabla \varphi_{\varepsilon}} \psi=0$.
Since $d^*_{f,h}=hd^*+\mathbf{i}_{\nabla f}$, we thus obtain
\begin{equation}\label{eq:BC2}
\mathbf{t}d^* v  =  -\frac{1}{h}\mathbf{i}_{\nabla f} v.
\end{equation}
By considering the boundary value
problem~\eqref{eq:deltaHv}--\eqref{eq:BC1}--\eqref{eq:BC2} and using
the $W^{1,2}(\Omega_+ \setminus \overline{\Omega_-})$
estimate~\eqref{eq:W12estim}, we
thus obtain by the elliptic regularity of the Dirichlet Hodge
Laplacian:
$$
\|e^{\frac{\varphi_{\varepsilon}}{h}}\psi\|_{W^{2,2}(\Omega_{+}\setminus\overline{\Omega_{-}})}\leq C_{2,\varepsilon} h^{-3}.
$$
This is due to the fact that the right-hand side in~\eqref{eq:deltaHv}
(resp. in~\eqref{eq:BC2}) is
a differential operator of order~1 (resp. of order~0).
The $W^{n,2}(\Omega_+ \setminus \overline{\Omega_-})$ estimates for $n
\ge 3$ are
then obtain by induction on $n$.
\end{proof}

\subsection{Change of basis in $F^{(1)}$}
\label{se.chanba}

In the previous sections, the estimates~\eqref{eq:expansion_lambda10} and~\eqref{eq.appdu10} of the eigenvalue
$\lambda_{1}^{(0)}$ and of
$d_{f,h}u_{1}^{(0)}$ in a neighborhood of $\partial \Omega_{+}$ have been
proven with the specific basis
$(\psi_{k}^{(1)})_{m_{1}^{N}(\Omega_{-})+1\leq k\leq
  m_{1}^{D}(\Omega_{+})}$ of $\Ran
1_{[0,\nu(h)]}(\Delta_{f,h}^{D,(1)}(\Omega_{+}\setminus\overline{\Omega_{-}}))$. The aim of this section is to show that
the estimates~\eqref{eq:expansion_lambda10} and~\eqref{eq.appdu10} are
valid for any almost orthonormal basis (according to Definition~\ref{de.almost}) $(u_{k}^{(1)})_{1\leq k\leq
  m_{1}^{D}(\Omega_{+}\setminus\overline{\Omega_{-}})}$ of
$\Ran
1_{[0,\nu(h)]}(\Delta_{f,h}^{D,(1)}(\Omega_{+}\setminus\overline{\Omega_{-}}))$. The next proposition thus concludes the proof of Theorem~\ref{th.main}.

\begin{remarque}\label{rem:almost_orthonormal}
We thus actually prove a slightly more general result than the one stated in
Theorem~\ref{th.main}, since it is only required that  the $(u_{k}^{(1)})_{1\leq k\leq
  m_{1}^{D}(\Omega_{+}\setminus\overline{\Omega_{-}})}$ is an {\em
  almost} orthonormal basis of
$\Ran
1_{[0,\nu(h)]}(\Delta_{f,h}^{D,(1)}(\Omega_{+}\setminus\overline{\Omega_{-}}))$.
\end{remarque}
\begin{remarque}
All the results below extend to complex valued eigenbases, by simply
replacing the real scalar product by the hermitian scalar product.
\end{remarque}

\begin{proposition}
\label{pr.expralmostorth}
Let $\lambda_{1}^{(0)}$ be the first eigenvalue of
$\Delta_{f,h}^{D,(0)}(\Omega_{+})$
and $u_{1}^{(0)}$ the associated $L^2(\Omega_+)$-normalized non negative eigenfunction. For any almost orthonormal basis
$(u_{k}^{(1)})_{1\leq k \leq
  m_{1}^{D}(\Omega_{+}\setminus\overline{\Omega_{-}})}$
 of 
$\Ran
1_{[0,\nu(h)]}(\Delta_{f,h}^{D,(1)}(\Omega_{+}\setminus\overline{\Omega_{-}}))$,
the approximate expressions \eqref{eq:expansion_lambda10} and
\eqref{eq.appdu10} for  $\lambda_{1}^{(0)}$ and
$d_{f,h}u_{1}^{(0)}$  hold true.
\end{proposition}
\begin{proof}
Let $(u_{k}^{(1)})_{1\leq k \leq
  m_{1}^{D}(\Omega_{+}\setminus\overline{\Omega_{-}})}$ be an almost
orthonormal basis of $\Ran
1_{[0,\nu(h)]}(\Delta_{f,h}^{D,(1)}(\Omega_{+}\setminus\overline{\Omega_{-}}))$. Then,
there exists a matrix $C(h)=(c_{k,k'})_{1\leq k,k' \leq
  m_{1}^{D}(\Omega_{+}\setminus\overline{\Omega_{-}})}$ such that
\begin{align}
&C(h)C(h)^{*}=\Id_{m_{1}^{D}(\Omega_{+}\setminus\overline{\Omega_{-}})}+\mathcal{O}(e^{-\frac{c}{h}})
\quad,\quad
C(h)^{*}C(h)=\Id_{m_{1}^{D}(\Omega_{+}\setminus\overline{\Omega_{-}})}+\mathcal{O}(e^{-\frac{c}{h}})\,,
\label{eq.CC*}
\\
\text{and }&
\psi_{k+m_{1}^{N}(\Omega_{-})}^{(1)}
=\sum_{k'=1}^{m_{1}^{D}(\Omega_{+}\setminus\overline{\Omega_{-}})}c_{k,k'}u_{k'}^{(1)}\,,\quad
\forall k \in
\left\{1,\ldots, m_{1}^{D}(\Omega_{+}\setminus\overline{\Omega_{-}})\right\}
\,.\nonumber
\end{align}
Here, $C(h)^{*}$ denotes the transpose of the matrix $C(h)$.

Let $L_{1}$ (resp. $L_{2}$) be a continuous linear mapping from the
finite dimensional space $\Ran
1_{[0,\nu(h)]}(\Delta_{f,h}^{D,(1)}(\Omega_{+}\setminus\overline{\Omega_{-}}))$
(endowed with the scalar product of $L^2(\Omega_{+}\setminus\overline{\Omega_{-}})$)
to $\R$ (resp. to some vector space~$E$). Then, the following estimate
holds, using~\eqref{eq.CC*}:
\begin{align}
\sum_{k=m_{1}^{N}(\Omega_{-})+1}^{m_{1}^{D}(\Omega_{+})} L_{1}(\psi_{k}^{(1)})L_{2}(\psi_{k}^{(1)})&=
\sum_{k,k_1,k_2=1}^{m_{1}^{D}(\Omega_{+}\setminus\overline{\Omega_{-}})} 
c_{k,k_1} c_{k,k_2} L_{1}(u_{k_1}^{(1)}) L_{2}(u_{k_2}^{(1)}) \nonumber \\
&=
\sum_{k'=1}^{m_{1}^{D}(\Omega_{+}\setminus
\overline{\Omega_{-}})}L_{1}(u_{k'}^{(1)}) L_{2}(u_{k'}^{(1)})+
\mathcal{O}(\|L_{1}\|\|L_{2}\|e^{-\frac{c}{h}})\,, \label{eq.estimL1L2}
\end{align}
where $\|L_{1}\|$ and  $\|L_{2}\|$ denote the operator norms of the
linear mappings $L_1$ and $L_2$. 

The estimate~\eqref{eq:expansion_lambda10} is then a consequence
of~\eqref{eq:expansion_lambda10_bis} and~\eqref{eq.estimL1L2} with
$$
L_{1}=L_{2}~: \Ran
1_{[0,\nu(h)]}(\Delta_{f,h}^{D,(1)}(\Omega_{+}\setminus\overline{\Omega_{-}}))
\ni u\mapsto -\frac{\int_{\partial
    \Omega_{+}}e^{-\frac{f(\sigma)}{h}}\mathbf{i}_{n}u (\sigma)~d\sigma}{
\left(\int_{\Omega_{+}}e^{-\frac{2f(x)}{h}}~dx\right)^{1/2}}\in
\R\,,
$$
with
$\|L_{1}\|=\|L_{2}\|=\tilde{\mathcal{O}}(e^{-\frac{\kappa_{f}}{h}})$
due to $\lambda_{1}^{(0)}(\Omega_{+})=\tilde{O}(e^{-\frac{2\kappa_{f}}{h}})$ (see~\eqref{eq:expansion_lambda10_bis}) and
the orthonormality of the basis
$(\psi_{k+m_{1}^{N}(\Omega_{-})}^{(1)})_{1 \le k \le m_1^D(\Omega_+
  \setminus \Omega_-)}$\,.

The estimate~\eqref{eq.appdu10} is a consequence of Proposition~\ref{pr.dfpsi}
and of~\eqref{eq.estimL1L2} with $L_{1}$ like before and 
$$
L_{2}~: \Ran
1_{[0,\nu(h)]}(\Delta_{f,h}^{D,(1)}(\Omega_{+}\setminus\overline{\Omega_{-}}))
\ni u  \mapsto u\Big|_{\mathcal V}  \in \bigwedge^{1}W^{n,2}(\mathcal{V})
$$
with $\|L_{2}\|=\tilde{\mathcal{O}}(1)$ according to
Lemma~\ref{le.L2W} applied with $\Delta_{f,h}^{D,(1)}(\Omega_{+}\setminus
\overline{\Omega_{-}})$ instead of $\Delta_{f,h}^{D,(1)}(\Omega_{+})$\,.
\end{proof}

\subsection{Corollaries}
\label{se.corol}
The estimate \eqref{eq.appdu10} contains an accurate information about
the trace $\partial_{n}u_{1}^{(0)}\big|_{\partial \Omega_{+}}$\,.
\begin{corollaire}
\label{co.trace}   Let $n: \sigma\mapsto n (\sigma)$ be the outward normal vector field on
$\partial \Omega_{+}$ and let $\partial_{n}=\mathbf{i}_{n}d$ be the
outward normal derivative for functions\,. For any almost orthonormal basis
$(u_{k}^{(1)})_{1\leq k\leq
  m_{1}^{D}(\Omega_{+}\setminus\overline{\Omega_{-}})}$ of $\Ran
1_{[0,\nu(h)]}(\Delta_{f,h}^{D,(1)}(\Omega_{+}\setminus \overline{\Omega_{-}}))$\,, the normal
derivative of the non negative and normalized first eigenfunction $u_{1}^{(0)}$ of
$\Delta_{f,h}^{D,(0)}(\Omega_{+})$ satisfies
\begin{align*}
  &\forall \sigma\in \partial \Omega_{+}\,,\quad \partial_{n}u_{1}^{(0)}(\sigma)\leq 0\,,
\\
\text{and }&\forall n \in \nz, \, 
\left\| \partial_{n}u_{1}^{(0)}\big|_{\partial\Omega_{+}}+ \sum_{k=1}^{m_{1}^{D}(\Omega_{+}\setminus\overline{\Omega_{-}})}\frac{\int_{\partial
    \Omega_{+}}e^{-\frac{f(\sigma)}{h}}\mathbf{i}_{n}
  {u}_{k}^{(1)}(\sigma)~d\sigma}{\left(\int_{\Omega_{+}}e^{-\frac{2f(x)}{h}}~dx\right)^{1/2}}\mathbf{i}_{n}u_{k}^{(1)}\right\|_{W^{n,2}(\partial \Omega_{+})}
=\mathcal{O}(e^{-\frac{\kappa_{f}+c}{h}})
\end{align*}
for some $c>0$ independent on $n$.
\end{corollaire}
\begin{proof}
The sign condition for $\partial_{n}u_{1}^{(0)}(\sigma)$ is a
consequence of $u_{1}^{(0)}\geq 0$ in $\Omega_{+}$ and
$u_{1}^{(0)}\big|_{\partial \Omega_{+}}=0$\,.

The trace theorem with \eqref{eq.appdu10} implies
$$
d_{f,h}u_{1}^{(0)}\big|_{\partial \Omega_{+}}=
-h\sum_{k=1}^{m_{1}^{D}(\Omega_{+}\setminus\overline{\Omega_{-}})}\frac{\int_{\partial
    \Omega_{+}}e^{-\frac{f(\sigma)}{h}}\mathbf{i}_{n}
 {u}_{k}^{(1)}(\sigma)~d\sigma}{\left(\int_{\Omega_{+}}e^{-\frac{2f(x)}{h}}~dx\right)^{1/2}}u_{k}^{(1)}+ \mathcal{O}(e^{-\frac{\kappa_{f}+c}{h}})
$$
in any Sobolev space $W^{n,2}(\partial \Omega_{+})$.
Recalling
$$
d_{f,h}u_{1}^{(0)}= hdu_{1}^{(0)}+ u_{1}^{(0)}df\quad\text{and}\quad
u_{1}^{(0)}\big|_{\partial \Omega_{+}}=0
$$
yields the result.
\end{proof}
Note that the numerators in the estimates~\eqref{eq:expansion_lambda10} and~\eqref{eq.appdu10} of the eigenvalue
$\lambda_{1}^{(0)}(\Omega_+)$ and of
$d_{f,h}u_{1}^{(0)}$  depend only on the values of $f$ and the geometry
of $\Omega_+$ around $\partial\Omega_{+}$\,. More precisely, they do not
change when $f$ is modified inside $\Omega_{-}$\,. This allows to
understand the variations of $\lambda_{1}^{(0)}(\Omega_+)$ and
$\partial_{n}u_{1}^{(0)}\big|_{\partial \Omega_{+}}$ with respect to
$f$, which is needed in the hyperdynamics algorithm (see the introduction).
\begin{corollaire}
\label{co.variaf} Let $f_{1}$ and $f_{2}$ be two functions which fulfill the
Hypotheses~\ref{hyp.0}, \ref{hyp.3}, \ref{hyp.1}, \ref{hyp.2} and let $\lambda_{1}^{(0)}(f_{1})$
(resp. $\lambda_{1}^{(0)}(f_{2})$) be the first eigenvalue of
$\Delta_{f_{1},h}^{D,(0)}(\Omega_{+})$ associated with the non
negative normalized eigenvector
$u_{1}^{(0)}(f_{1})$ (resp. of $\Delta_{f_{2},h}^{D,(0)}(\Omega_{+})$
associated with the
eigenvector $u_{1}^{(0)}(f_{2})$).
Assume additionally $f_{1}=f_{2}$ in $\Omega_{+}\setminus\overline{\Omega_{-}}$\,.
The quantities $\lambda_{1}^{(0)}(f_{1,2})$ and
$\partial_{n}\left[e^{-\frac{f_{1,2}}{h}}u_{1}^{(0)}(f_{1,2})\right]\big|_{\partial
    \Omega_{+}}= e^{-\frac{f_{1,2}}{h}}\left[\partial_{n}u_{1}^{(0)}(f_{1,2})\right]\big|_{\partial
    \Omega_{+}}$ satisfy
\begin{align}
  \label{eq.complam}
&
\frac{\lambda_{1}^{(0)}(f_{2})}{\lambda_{1}^{(0)}(f_{1})}=\frac{\int_{\Omega_{+}}
  e^{-2\frac{f_{1}(x)}{h}}~dx}{\int_{\Omega_{+}}e^{-2\frac{f_{2}(x)}{h}}~dx}
  (1+\mathcal{O}(e^{-\frac{c}{h}}))\,,\\
\label{eq.comppau}
&\frac{\partial_{n}\left[e^{-\frac{f_{2}}{h}}u_{1}^{(0)}(f_{2})\right]\big|_{\partial \Omega_{+}}}{ 
\|\partial_{n}\left[e^{-\frac{f_{2}}{h}}u_{1}^{(0)}(f_{2})\right]\|_{L^{1}(\partial
  \Omega_{+})}}
= 
\frac{\partial_{n}\left[e^{-\frac{f_{1}}{h}}u_{1}^{(0)}(f_{1})\right]\big|_{\partial \Omega_{+}}}{ 
\|\partial_{n}\left[e^{-\frac{f_{1}}{h}}u_{1}^{(0)}(f_{1})\right]\|_{L^{1}(\partial
  \Omega_{+})}}
+
\mathcal{O}(e^{-\frac{c}{h}})\quad \text{in}~L^{1}(\partial \Omega_{+})\,.
\end{align}
\end{corollaire}
\begin{proof}
First, note that the equality
$$
\partial_{n}\left[e^{-\frac{f_{1,2}}{h}}u_{1}^{(0)}(f_{1,2})\right]\big|_{\partial
    \Omega_{+}}= e^{-\frac{f_{1,2}}{h}}\left[\partial_{n}u_{1}^{(0)}(f_{1,2})\right]\big|_{\partial
    \Omega_{+}}
$$
is simply due to the Dirichlet boundary condition $u_{1}^{(0)}\big|_{\partial \Omega_{+}}=0$\,.
The identity \eqref{eq.complam} is then a direct consequence of
\eqref{eq:expansion_lambda10}, since the same basis $(u_k^{(1)})_{1 \le k
  \le m_1^D(\Omega_{+}\setminus\overline{\Omega_{-}})}$ can be picked for $f_1$ and $f_2$, since these two functions coincide on $\Omega_{+}\setminus\overline{\Omega_{-}}$.

Second, for \eqref{eq.comppau}, it is more convenient to write
\eqref{eq.appdu10} with $f_{j}$, $j=1,2$, in the form
$$\left(\int_{\Omega_{+}}e^{-\frac{2f_{j}(x)}{h}}~dx\right)^{1/2} \, d_{f_{j},h}u_{1}^{(0)}(f_{j})
=-h\sum_{k=1}^{m_{1}^{D}(\Omega_{+}\setminus\overline{\Omega_{-}})} \left(\int_{\partial
\Omega_{+}}e^{-\frac{f_{j}(\sigma)}{h}}\mathbf{i}_{n} {u_{k}^{(1)}}(\sigma)~d\sigma\right)u_{k}^{(1)}
+{\mathcal{O}}(e^{-\frac{\min_{\partial
      \Omega_{+}}f_{j}+c}{h}})\,,
$$the estimate being true in any Sobolev space $\bigwedge^1  W^{n,2}(\mathcal{V})$.
Using the fact that
 $f_{1}\equiv f_{2}\equiv f$ in
 $\Omega_{+}\setminus\overline{\Omega_{-}}$\,,  taking the trace along
 $\partial \Omega_{+}$ and multiplying by
 $e^{-\frac{f-\min_{\partial\Omega_{+}}f}{h}}$ which is less than $1$
 on $\partial \Omega_{+}$ and then by
 $e^{\frac{\min_{\partial\Omega_{+}}f}{h}}$, lead to 
\begin{align*}
&\left(\int_{\Omega_{+}}e^{-2\frac{f_{j}(x)}{h}}~dx\right)^{1/2}
e^{-\frac{f-2\min_{\partial\Omega_{+}}f}{h}}
\partial_{n}u_{1}^{(0)}(f_{j})\big|_{\partial \Omega_{+}}
\\
&=
-
\sum_{k=1}^{m_{1}^{D}(\Omega_{+}\setminus\overline{\Omega_{-}})} \left(\int_{\partial
\Omega_{+}}e^{-\frac{f(\sigma)-\min_{\partial \Omega_{+}f}}{h}}\mathbf{i}_{n}
{u_{k}^{(1)}} (\sigma)~d\sigma\right) \, e^{-\frac{f-\min_{\partial \Omega_{+}}f}{h}}\mathbf{i}_{n}u_{k}^{(1)}
+{\mathcal{O}}(e^{-\frac{c}{h}})\, ,
\end{align*}
the estimate being true in $L^{1}(\partial \Omega_{+})$.
The left-hand side is negative and its $L^{1}$-norm is
thus given by the absolute value of its integral. Let us estimate this
norm, using Lemma~\ref{le.intminmaj}  and Lemma~\ref{le.infsigD0}:
for any positive $\varepsilon$,
\begin{align*}
-\left(e^{-\frac{2f_{j}(x)}{h}}~dx\right)^{1/2} &
\int_{\partial \Omega_+}  e^{-\frac{f-2\min_{\partial\Omega_{+}}f}{h}}
\partial_{n}u_{1}^{(0)}(f_{j}) (\sigma) d \sigma
\\
&=
\sum_{k=1}^{m_{1}^{D}(\Omega_{+}\setminus\overline{\Omega_{-}})} \left(\int_{\partial
\Omega_{+}}e^{-\frac{f(\sigma)-\min_{\partial \Omega_{+}f}}{h}}\mathbf{i}_{n}
{u_{k}^{(1)}}(\sigma)~d\sigma \right)^2
+{\mathcal{O}}(e^{-\frac{c}{h}})\\
&= e^{\frac{2\min_{\partial \Omega_{+}f}}{h}} \lambda_1^{(0)}(f_1)
h^{-2} \int_{\Omega_+} e^{-2\frac{f_1(x)}{h}} \, dx
+{\mathcal{O}}(e^{-\frac{c}{h}})\\
&\ge  C_\varepsilon e^{\frac{2\min_{\partial \Omega_{+}f}}{h}} e^{-2
  \frac{\kappa_f +\varepsilon}{h}}
h^{-2} \frac{1}{C_{f_1}} h^{d/2} e^{-\frac{2\min_{\Omega_+} f_1}{h}}
+{\mathcal{O}}(e^{-\frac{c}{h}})\\
&= C_\varepsilon e^{-
  \frac{2\varepsilon}{h}}
\frac{h^{-2+d/2}}{C_{f_1}}
+{\mathcal{O}}(e^{-\frac{c}{h}}) \ge C e^{-\frac{c}{2h}}.
\end{align*}
Thus,
\begin{align*}
&- \frac{ e^{-\frac{f_j}{h}}\partial_{n}u_{1}^{(0)}(f_{j})\big|_{\partial \Omega_{+}}}
{ \| e^{-\frac{f_j}{h}}\partial_{n}u_{1}^{(0)}(f_{j})\big|_{\partial
    \Omega_{+}} \|_{L^1(\partial \Omega_+)}}\\
&=\frac{\sum_{k=1}^{m_{1}^{D}(\Omega_{+}\setminus\overline{\Omega_{-}})}
  \left(\int_{\partial
\Omega_{+}}e^{-\frac{f(\sigma)-\min_{\partial \Omega_{+}f}}{h}}\mathbf{i}_{n}
{u_{k}^{(1)}} (\sigma)~d\sigma\right) e^{-\frac{f-\min_{\partial \Omega_{+}}f}{h}}\mathbf{i}_{n}u_{k}^{(1)}}{\sum_{k=1}^{m_{1}^{D}(\Omega_{+}\setminus\overline{\Omega_{-}})} \left(\int_{\partial
\Omega_{+}}e^{-\frac{f(\sigma)-\min_{\partial \Omega_{+}f}}{h}}\mathbf{i}_{n}
{u_{k}^{(1)}}(\sigma)~d\sigma \right)^2}+{\mathcal{O}}(e^{-\frac{c}{2h}}).
\end{align*}
This concludes the proof since the right-hand side does not depend on $f_j$.
\end{proof}

\section{About the Hypotheses~\ref{hyp.1} and~\ref{hyp.2}}
\label{se.ass}

We have chosen to set the Hypotheses~\ref{hyp.1} and~\ref{hyp.2} in terms of some spectral properties of the Witten Laplacians
$\Delta_{f,h}^{N}(\Omega_{-})$ and 
$\Delta_{f,h}^{D}(\Omega_{-}\setminus\overline{\Omega_{-}})$ in order
to be general enough and to cover possible further advances about the low
spectrum of Witten Laplacians. These hypotheses can actually be translated in very
explicit and simple geometric conditions on the function $f$ when $f$
is a Morse function such that $f\big|_{\partial
  \Omega_{+}}$is a Morse function. We recall that a Morse
function is a
$\mathcal{C}^{\infty}$ function whose all critical points are
non degenerate. Section~\ref{se.Morse} is devoted to a verification of
the Hypotheses~\ref{hyp.1} and~\ref{hyp.2}   when $f$  and $f\big|_{\partial \Omega_{+}}$ are
Morse functions, using the results of \cite{HeNi} and \cite{Lep3}. 
 Theorem~\ref{th.main2} is then obtained as a consequence of the
accurate results under the Morse conditions and the estimates stated
in Corollary~\ref{co.variaf}.

Finally, Section~\ref{se.nonMorse} is devoted to a discussion about
potentials which are not Morse functions. In particular, examples of
functions $f$ which are not Morse functions and for which
Hypotheses~\ref{hyp.1} and~\ref{hyp.2} hold 
are presented.

\subsection{The case of a Morse function $f$}
\label{se.Morse}

\subsubsection{Verifying the Hypotheses~\ref{hyp.1} and~\ref{hyp.2}}
\label{se.assmorse}

Let us first specify the assumptions which allow to use the results of
\cite{HeNi} and \cite{Lep3}, in
addition to the Hypotheses~\ref{hyp.0} and~\ref{hyp.3} which
were already explicitly formulated in terms of the function~$f$:
\begin{hypothese}
\label{hyp.4}
The functions $f$ and $f\big|_{\partial\Omega_{+}}$  are Morse functions.
\end{hypothese}
\begin{hypothese}
\label{hyp.5}
  The critical values of $f$ 
are all distinct and the differences $f(U^{(1)})-f(U^{(0)})$, where
$U^{(0)}$ ranges over the local minima of $f$ and
$U^{(1)}$ ranges over the critical points of $f$
with index $1$, 
are all distinct.
\end{hypothese}
Although $f|_{\partial \Omega_-}$ is not assumed to be a Morse
function (see the discussion below), Hypotheses~\ref{hyp.0},~\ref{hyp.4} and~\ref{hyp.5} ensure that the
results of \cite{HeNi} and \cite{Lep3} on 
small eigenvalues of $\Delta_{f,h}^{D}(\Omega_{+})$,
$\Delta_{f,h}^{N}(\Omega_{-})$ and $\Delta_{f,h}^{D}(\Omega_{+}
\setminus \overline{\Omega_-})$ apply. Following~\cite{Lep3}, the Hypothesis~\ref{hyp.5} is useful to get accurate
scaling rates for the small eigenvalues of
$\Delta_{f,h}^{N,(0)}(\Omega_{-})$. In particular, the information on
the size of the second eigenvalue
$\mu_{2}^{(0)}(\Omega_-)> \mu_{1}^{(0)}(\Omega_-)=0$ of
$\Delta_{f,h}^{N,(0)}(\Omega_{-})$ is important to
prove~\eqref{eq.decayhyp} in Hypothesis~\ref{hyp.1}. Hypothesis~\ref{hyp.5} also implies that $f$ has a
unique global minimum. Hypothesis~\ref{hyp.5} could certainly be relaxed.

Let us recall the general results of~\cite{HeNi,Lep3}, on the
number and the scaling of small eigenvalues for
boundary Witten Laplacians in a regular domain $\Omega$ (see
also~\cite{ChLi,Lau} for related results). The potential
$f$ is assumed to be a Morse
function $f$ on $\Omega$ such that $|\nabla f|\neq 0$ on $\partial
\Omega$ and $f|_{\partial \Omega}$ is also a
Morse function. The notion of critical points with
index $p$ for $f$ has to be extended as follows, in order to take into
account points on the  boundary $\partial \Omega$.
\begin{description}
\item[In the interior $\Omega$:] A generalized critical point with
index $p$ is as usual a critical point at which the Hessian of $f$ 
has $p$ negative eigenvalues. It is a local minimum for $p=0$, a
saddle point for $p=1$ and a local maximum for $p=\dim M=d$\,.
\item[Along the boundary $\partial \Omega$ in the Dirichlet case:] 
A generalized critical point with index $p \ge 1$ is a critical point $\sigma$ of $f|_{\partial
  \Omega}$ with index $p-1$ such that the outward normal derivative is positive
 ($\partial_{n}f(\sigma)>0$). Therefore, along the boundary, there is no
 generalized critical point with index $0$ and critical points with
 index $1$
 coincide with the local minima $\sigma$ of $f\big|_{\partial
   \Omega}$ such that $\partial_{n}f(\sigma)>0$.
Intuitively, this definition can be understood by interpreting
the homogeneous Dirichlet boundary
 conditions as an extension of the  potential by $-\infty$ outside~$\Omega$\,.
\item[Along the boundary $\partial \Omega$ in the Neumann case:] 
A generalized critical point with index $p$ is a critical point
$\sigma$ of $f|_{\partial
  \Omega}$ with index $p$ such that the outward normal derivative is
negative ($\partial_{n}f(\sigma)<0$). Therefore, along the boundary,
a generalized critical point with index $0$
is a local minimum of $f|_{\overline{\Omega}}$ 
and a critical point with index $1$ is a saddle point $\sigma$ of $f|_{\partial \Omega}$ such that
 $\partial_{n}f(\sigma)<0$.
Intuitively, this definition can be understood by interpreting
 the homogeneous Neumann boundary
 conditions as an extension of the  potential by $+\infty$ outside~$\Omega$\,.
\end{description}
  The number of generalized critical points in $\Omega$ with index $p$
  is denoted   $\tilde{m}_{p}^{D}(\Omega)$ or $\tilde{m}_{p}^{N}(\Omega)$, depending on whether the boundary Witten
  Laplacian on $\Omega$
  with Dirichlet or Neumann boundary conditions is considered.

One result of \cite{HeNi,Lep3} says
that for $\nu(h)=h^{6/5}$, one has for the Dirichlet Witten Laplacian
  \begin{align*}
&\sharp\left[\sigma(\Delta_{f,h}^{D,(p)}(\Omega))
\cap
[0,\nu(h)]\right]=\tilde{m}_{p}^{D}(\Omega)\,,\\
&
\sigma(\Delta_{f,h}^{D,(p)}(\Omega))
\cap
[0,\nu(h)]\subset \left[0,e^{-\frac{c_{0}}{h}}\right]\,,
  \end{align*}
and for the Neumann boundary Witten Laplacian:
  \begin{align*}
&\sharp\left[\sigma(\Delta_{f,h}^{N,(p)}(\Omega))
\cap
[0,\nu(h)]\right]=\tilde{m}_{p}^{N}(\Omega)\,,\\
&
\sigma(\Delta_{f,h}^{N,(p)}(\Omega))
\cap
[0,\nu(h)]\subset \left[0,e^{-\frac{c_{0}}{h}}\right]\,,
  \end{align*}
for some positive constant $c_0$.
These results rely, like in \cite{CFKS} for the
boundaryless case, on the
introduction of an $h$-dependent
partition of unity and a rough analysis of boundary local models.

Let us now apply these general results in our context.
Under Hypotheses~\ref{hyp.0} and~\ref{hyp.4}, it holds: 
\begin{itemize}
\item $\tilde{m}_{p}^{N}(\Omega_{-})$ is the number of critical points with
  index $p$ in the interior of $\Omega_{-}$\,;
\item $\tilde{m}_{p}^{D}(\Omega_{+}\setminus\overline{\Omega_{-}})$ is the number of
  critical points $\sigma$ with index $p-1$ of $f\big|_{\partial
    \Omega_{+}}$ such that $\partial_n f(\sigma) > 0$. In particular $\tilde{m}_{0}^{D}(\Omega_{+}\setminus
  \overline{\Omega_{-}})=0$ and $\tilde{m}_{1}^{D}(\Omega_{+}\setminus
  \overline{\Omega_{-}})$ is  the number of local minima of
  $f\big|_{\partial \Omega_{+}}$ with positive normal derivatives\,;
\item $\tilde{m}_{p}^{D}(\Omega_{+})$ is the number of critical points with
  index $p$ in the interior of $\Omega_{-}$ plus the number of
  critical points $\sigma$ of $f\big|_{\partial\Omega_{+}}$ with index $p-1$  such that $\partial_n f(\sigma) > 0$\,. For
  $p=0$, $\tilde{m}_{0}^{D}(\Omega_{+})$ equals $\tilde{m}_{0}^{N}(\Omega_{-})$ while
  $\tilde{m}_{1}^{D}(\Omega_{+})$ is $m_{1}^{N}(\Omega_{-})$ augmented by the
  number of local minima of $f\big|_{\partial \Omega_{+}}$ with
  positive normal derivatives.
\end{itemize}
As already mentioned above, we can use the results
of~\cite{HeNi,Lep3} without assuming that $f|_{\partial \Omega_-}$ is
a Morse function. The reason is that $\partial_n f>0$ on
$\partial \Omega_-$ and thus, there is no generalized critical point
on $\partial \Omega_-$ associated with $\Delta_{f,h}^{N,(p)}(\Omega_-)$ and
$\Delta_{f,h}^{D,(p)}(\Omega_+ \setminus \overline{\Omega_-})$.

In summary, using these results, conditions
\eqref{eq.nombN}, \eqref{eq.tailleN}, \eqref{eq.nombD} and~\eqref{eq.tailleD}
are fulfilled with $\nu(h)=h^{6/5}$, some $c_{0}>0$ and
$m_{p}^{N,D}(\Omega)=\tilde{m}_{p}^{N,D}(\Omega)$, $p\in\{0,1\}$ and
$\Omega=\Omega_{-}$ or $\Omega=\Omega_{+}\setminus\overline{\Omega_{-}}$\,. 
Hence all the conditions of Hypotheses~\ref{hyp.1}
and~\ref{hyp.2} are satisfied except \eqref{eq.decayhyp}. Note in particular that the
two following results
in Theorem~\ref{th.main}:
$$  m_{0}^{D}(\Omega_{+})=m_{0}^{N}(\Omega_{-}) \text{ and }
m_{1}^{D}(\Omega_{+})=m_{1}^{N}(\Omega_{-})+m_{1}^{D}(\Omega_{+}\setminus\overline{\Omega_{-}})$$
are consistent with the relations on the numbers of generalized
critical points:
$$  \tilde{m}_{0}^{D}(\Omega_{+})=\tilde{m}_{0}^{N}(\Omega_{-}) \text{
  and }
\tilde{m}_{1}^{D}(\Omega_{+})=\tilde{m}_{1}^{N}(\Omega_{-})+\tilde{m}_{1}^{D}(\Omega_{+}\setminus\overline{\Omega_{-}}).$$
As explained in the proof below, the Hypothesis~\ref{hyp.5} is particularly useful to verify the
condition \eqref{eq.decayhyp} in Hypothesis~\ref{hyp.1}.

The following proposition thus
yields a simple set of assumptions on $f$ such that
Theorem~\ref{th.main} holds.
\begin{proposition}
\label{pr.assMorse}
Assume Hypotheses~\ref{hyp.0},~\ref{hyp.4}, \ref{hyp.5}
and let $\mathcal{U}^{(0)}$ (resp. $\mathcal{U}^{(1)}$) denote the set
of critical points with index $0$ (resp. $1$) of $f\big|_{\Omega_{-}}$\,.
Let us consider the Agmon distance $d_{Ag}$ introduced in Lemma~\ref{le.dist}.
Then the inequality
\begin{equation}
  \label{eq.conddag}
d_{Ag}(\partial \Omega_{-},\mathcal{U}^{(0)})>\max_{U^{(1)}\in
  \mathcal{U}^{(1)}\,,\, U^{(0)}\in \mathcal{U}^{(0)}} f(U^{(1)})-f(U^{(0)})
\end{equation}
implies~\eqref{eq.decayhyp}. As a consequence, the
 inequality~\eqref{eq.conddag} together with the
 Hypotheses~\ref{hyp.0},~\ref{hyp.3},~\ref{hyp.4} and \ref{hyp.5}  are sufficient
 conditions for the results of
Theorem~\ref{th.main} and its corollaries to hold.
\end{proposition}
Figures~\ref{fig:2D} and~\ref{fig:1D_OK} give examples of functions
$f$ for which the inequality~\eqref{eq.conddag} together with the
 Hypotheses~\ref{hyp.0},~\ref{hyp.3},~\ref{hyp.4} and~\ref{hyp.5} are
 fulfilled.  Figure~\ref{fig:1D_pas_OK} is an example of a function
 $f$ which satisfies Hypotheses~\ref{hyp.0},~\ref{hyp.3},~\ref{hyp.4}
 and \ref{hyp.5}, but not the inequality~\eqref{eq.conddag}.
\begin{remarque}
\label{re.assMorse}
Since $d_{Ag}(x,y)\geq |f(x)-f(y)|$ (see~\eqref{eq:fx-fy}), the condition
\eqref{eq.condcvmax} given in the introduction is a sufficient
condition for \eqref{eq.conddag}. The condition~\eqref{eq.condcvmax}
also implies Hypothesis~\ref{hyp.3}.
Thus, a set of sufficient conditions for Theorem~\ref{th.main} to
hold is Hypotheses~\ref{hyp.0},~\ref{hyp.4} and \ref{hyp.5} together
with~\eqref{eq.condcvmax}. This is indeed the simple setting presented
in the introduction (see the four assumptions stated in Section~\ref{sec.res}).
\end{remarque}
\begin{remarque}
It may happen that ${\mathcal U}^{(1)}=\emptyset$. In this case,
the inequality~\eqref{eq.conddag} is automatically satisfied, and
there are no exponentially small nonzero eigenvalue for
$\Delta_{f,h}^{N,(0)}(\Omega_-)$. Consistently,~\eqref{eq.decayhyp} is
a void condition in this case.
\end{remarque}
\begin{proof}
According to the previous discussion, it only remains to prove that
Hypotheses~\ref{hyp.0},~\ref{hyp.4}, \ref{hyp.5} together
with~\eqref{eq.conddag} imply~\eqref{eq.decayhyp} for the proposition
to hold. According to
\cite{Lep3}, the smallest nonzero eigenvalue of
$\Delta_{f,h}^{N,(0)}(\Omega_-)$ (namely $\mu_2^{(0)}(\Omega_-)$), satisfies under Hypotheses~\ref{hyp.4} and
\ref{hyp.5} the inequality
$$
\lim_{h\to
  0}h\log(\mu_2^{(0)}(\Omega_-))=-2
\left(f(U_{j_{1}}^{(1)})-f(U_{j_{0}}^{(0)}) \right)
\ge  -2 \max_{U^{(1)}\in\mathcal{U}^{(1)}\,,\, U^{(0)}\in
        \mathcal{U}^{(0)}}f(U^{(1)})-f(U^{(0)}),
$$
where $U^{(0)}_{j_0}$ and $U^{(1)}_{j_1}$ are two critical point of
index $0$ and $1$ respectively.

Let us now consider the exponential decay near $\partial \Omega_{-}$ of an eigenfunction of
$\Delta_{f,h}^{N,(0)}(\Omega_{-})$ associated with a non-zero
exponentially small eigenvalue. 
A stronger version of Proposition~\ref{pr.decayO-} can be given
because
 under Hypotheses~\ref{hyp.0},~\ref{hyp.4} and~\ref{hyp.5} the critical
 points of $f\big|_{\Omega_{-}}$ which are not local minima, are not
 associated with small eigenvalues of $\Delta_{f,h}^{N,(0)}(\Omega_-)$  (they
 are so-called {\em non
 resonant} wells,  see~\cite{HeSj2}). Indeed, when~$U$ is a critical
 point of $f\big|_{\Omega_{-}}$ with $U\not\in \mathcal{U}^{(0)}$, the
 local model of $\Delta_{f,h}^{D,(0)}(B(U,r))$
has his spectrum included
 in~$[h/C(U,r),+\infty)$, for $r>0$ small enough, (see for example \cite{CFKS}).
 Then, the Corollary~2.2.7 of \cite{HeSj2} implies that any normalized
 eigenfunction  $\psi(h)$ of $\Delta_{f,h}^{N,(0)}(\Omega_-)$  associated with an eigenvalue
 $\mu(h)\in [0,e^{-\frac{c_{0}}{h}}]$ satisfies: $\forall \varepsilon
 > 0$, $\exists C_\varepsilon >0$, $\forall x \in \Omega_-$,
$$
\quad |\psi_{h}(x)| \leq C_\varepsilon
\left(e^{-\frac{d_{Ag}(x,\mathcal{U}_{0}) +\varepsilon}{h}}\right)
$$
(compare with the result of Proposition~\ref{pr.decayO-}).
Hence the condition~\eqref{eq.conddag} implies that in a small 
 neighborhood $\mathcal{V}_{-}$ of $\partial\Omega_{-}$, the eigenfunction
$\psi(h)$ is estimated by
\begin{align*}
\|\psi(h)\|_{L^{2}(\mathcal{V}_{-})}
&=
\tilde{\mathcal{O}}\left(e^{-\frac{d_{Ag}( \mathcal{V}_{-},\mathcal{U}^{(0)})}{h}}\right)
\leq 
C\exp \left(-\frac{\displaystyle{\max_{U^{(1)}\in\mathcal{U}^{(1)}\,,\, U^{(0)}\in
        \mathcal{U}^{(0)}}f(U^{(1)})-f(U^{(0)}) + c}}{h} \right)\\
&\leq \tilde{\mathcal{O}} \left(\sqrt{\mu_{2}^{(0)}(\Omega_-) } \right)\leq \tilde{\mathcal{O}}(\sqrt{\mu(h)})\,,
\end{align*}
provided that $\mu(h)\neq \mu_{1}^{(0)}(\Omega_-) =0$. This is exactly~\eqref{eq.decayhyp}.
\end{proof}

\begin{figure}[htbp]
\centerline{\includegraphics[width=0.5\textwidth]{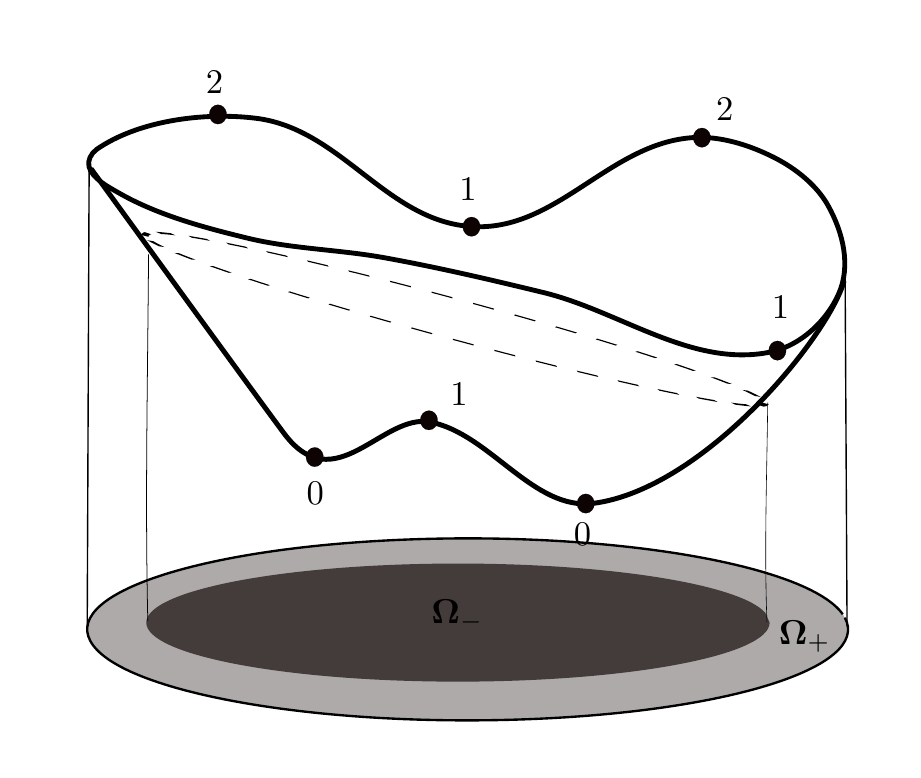}}
\caption{A two-dimensional example where
the
 inequality~\eqref{eq.conddag} together with the
 Hypotheses~\ref{hyp.0},~\ref{hyp.3},~\ref{hyp.4} and~\ref{hyp.5} are fulfilled. The index of the generalized critical points are labelled.}\label{fig:2D}
\end{figure}



\begin{figure}[htbp]
\centerline{\includegraphics[width=0.5\textwidth]{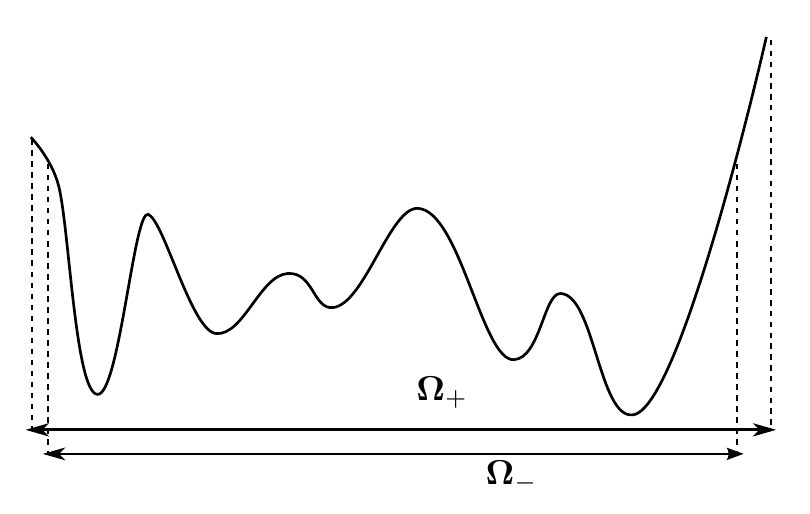}}
\caption{A one-dimensional example where the
 inequality~\eqref{eq.conddag} together with the
 Hypotheses~\ref{hyp.0},~\ref{hyp.3},~\ref{hyp.4} and~\ref{hyp.5} are
fulfilled.}\label{fig:1D_OK}
\end{figure}



\begin{figure}[htbp]
\centerline{\includegraphics[width=0.5\textwidth]{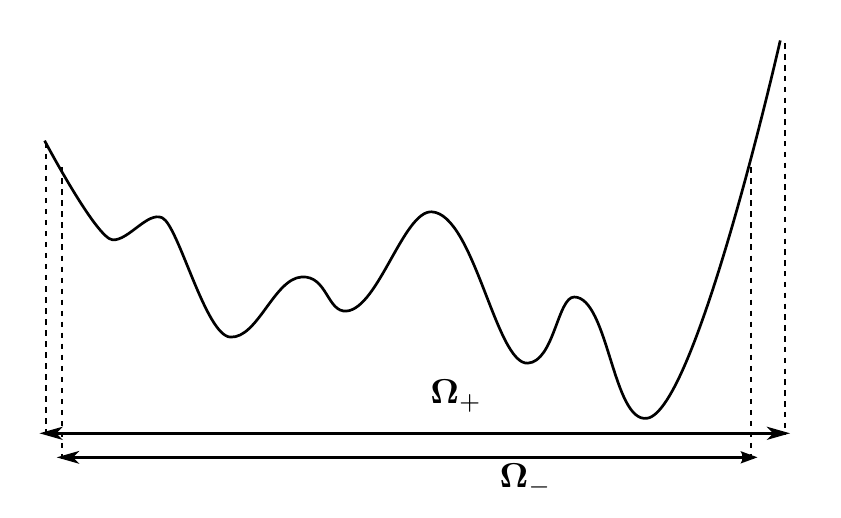}}
\caption{A one-dimensional example where the
 Hypotheses~\ref{hyp.0},~\ref{hyp.3},~\ref{hyp.4} and \ref{hyp.5} are
 fulfilled, but the 
 inequality~\eqref{eq.conddag} is not satisfied. The
condition~\eqref{eq.conddag} would be fulfilled with a lower local
minimum on the left-hand side for example (see Figure~\ref{fig:1D_OK}).}\label{fig:1D_pas_OK}
\end{figure}

\begin{remarque}[Assumptions in terms of  $\Omega_{+}$ only]
\label{se.assmor+}
Let us assume that Hypotheses~\ref{hyp.3},~\ref{hyp.4} and~\ref{hyp.5} hold. Then, it is easy to check that if
\begin{equation}\label{eq.dnf+}
\partial_{n}f\big|_{\partial\Omega_{+}}>0,
\end{equation}
and
\begin{equation}\label{eq.conddag+}
d_{Ag}(\partial\Omega_{+}, \mathcal{U}^{(0)})
> 
\max_{U^{(1)}\in
  \mathcal{U}^{(1)}\,,\, U^{(0)}\in \mathcal{U}^{(0)}} f(U^{(1)})-f(U^{(0)})\,,
\end{equation}
 then there exists a regular open domain $\Omega_{-}$ such that $\overline{\Omega_-}\subset
\Omega_{+}$ and the Hypothesis~\ref{hyp.0}
and the condition~\eqref{eq.conddag} hold.  Indeed, the 
  conditions~\eqref{eq.dnf+} and~\eqref{eq.conddag+} are open and allow small deformation
  from $\Omega_{+}$ to some subset $\Omega_{-}$\,. Note that 
the condition~\eqref{eq.dnf+} implies
that this small deformation can be chosen such that all the
critical points of $f$ are indeed in $\Omega_{-}$: this is exactly Hypothesis~\ref{hyp.0}.
 As a consequence, under the Hypotheses~\ref{hyp.3},~\ref{hyp.4} and~\ref{hyp.5} and the two
 assumptions~\eqref{eq.dnf+} and~\eqref{eq.conddag+},
 the results of Theorem~\ref{th.main} hold for a well chosen domain
 $\Omega_-$ such that $\overline{\Omega_-} \subset \Omega_+$.

In addition, following
Remark~\ref{re.assMorse} above, it is easy to check that
the inequality
\begin{equation}\label{eq:1.9bis}
\min_{\partial \Omega_{+}}f-\mathrm{cvmax}>
\mathrm{cvmax}-\min_{\Omega_{+}}f
\end{equation}
is a sufficient condition for \eqref{eq.conddag+}. It also implies
Hypothesis~\ref{hyp.3}. Thus, under the Hypotheses~\ref{hyp.4},~\ref{hyp.5} and the two
 assumptions~\eqref{eq.dnf+} and~\eqref{eq:1.9bis},
 the results of Theorem~\ref{th.main} hold for a well chosen domain
 $\Omega_-$ such that $\overline{\Omega_-} \subset \Omega_+$.
\end{remarque}

\subsubsection{Proof of Theorem~\ref{th.main2}}
\label{se.resultmorse}

In this section, more explicit formulas for
$\lambda_1^{(0)}(\Omega_+)$ and
$\partial_{n}\left(e^{-\frac{f}{h}}u_{1}^{(0)}\right)$ are given,
 under
the Morse
assumption on $f$ and $f|_{\partial \Omega_+}$.  We shall prove
\begin{proposition}
\label{pr.morseresult}
Assume Hypotheses~\ref{hyp.0},~\ref{hyp.3},~\ref{hyp.4}, \ref{hyp.5}, the
condition~\eqref{eq.conddag} 
and, moreover, 
\begin{equation}\label{eq:dnf_omega+}
\partial_n f >0 \text{ on } \partial \Omega_+.
\end{equation}
Then the first eigenvalue $\lambda_{1}^{(0)}(\Omega_{+})$ of
$\Delta_{f,h}^{D,(0)}(\Omega_{+})$ satisfies
\begin{align}
\label{eq.l10Hess}
\lambda_{1}^{(0)}(\Omega_{+})
&=
\sum_{k=1}^{m_{1}^{D}(\Omega_{+}\setminus\overline{\Omega_{-}})}
\sqrt{\frac{h\det(\Hess
  f)(U_{0})}{\pi\det(\Hess f\big|_{\partial
    \Omega_{+}})(U_{k}^{(1)})}}
\, 2\partial_{n}f(U_{k}^{(1)}) \, 
e^{-2\frac{f(U_{k}^{(1)})-f(U_{0})}{h}}
(1+\mathcal{O}(h))
\\
\label{eq.l10int}
&=
\frac{\int_{\partial \Omega_{+}} 2\partial_{n}f(\sigma)
  e^{-2\frac{f(\sigma)}{h}}\,d\sigma}{
\int_{\Omega_{+}}e^{-2\frac{f(x)}{h}}\,dx} (1+\mathcal{O}(h))\,,
\end{align}
where $U_{0}$ is the (unique) global minimum of $f$ in $\Omega_+$ and the
$U_{k}^{(1)}$'s are the local minima of $f\big|_{\partial
  \Omega_{+}}$\,. Moreover,
 the normalized non negative eigenfunction  $u_{1}^{(0)}$ of $\Delta_{f,h}^{D,(0)}(\Omega_{+})$
associated with $\lambda_{1}^{(0)}(\Omega_{+})$ satisfies
\begin{equation}
  \label{eq.denssortpfn}
- \frac{\partial_{n}\left[e^{-\frac{f}{h}}u_{1}^{(0)}\right]\big|_{\partial \Omega_{+}}}{ 
\left\|\partial_{n}\left[e^{-\frac{f}{h}}u_{1}^{(0)}\right]\right\|_{L^{1}(\partial
  \Omega_{+})}}
=
\frac{(2\partial_{n}f) e^{-\frac{2f}{h}}\big|_{\partial \Omega_{+}}
}{
\left\|
(2\partial_{n}f)
  e^{-\frac{2f}{h}}\right\|_{L^{1}(\partial \Omega_{+})}
}
 + \mathcal{O}(h)\quad \text{in}~L^{1}(\partial \Omega_{+})\,.
\end{equation}
\end{proposition}
\begin{remarque}
The hypothesis $\partial_n f >0$ on $\partial \Omega_+$ 
ensures that the set of all the local minima $U_k^{(1)}$ of $f|_{\partial \Omega_+}$
coincides with the set of generalized critical points with index $1$
for $\Delta_{f,h}^{D}(\Omega_+ \setminus \overline{\Omega_-})$. The
results of Proposition~\ref{pr.morseresult} also hold under the more
general assumption $\partial_{n}f(\sigma)>0$ when $\sigma\in \partial
\Omega_{+}$ is such that $f(\sigma)\leq \min_{\partial \Omega_{+}}f
+\varepsilon_{0}$ for some $\varepsilon_{0}>0$, by adapting the
arguments below.
\end{remarque}
\begin{remarque}
\label{re.densorthpfn}
It is possible to write explicitly a first order approximation for the
probability density $- \frac{\partial_{n}\left(e^{-\frac{f}{h}}u_{1}^{(0)}\right)\big|_{\partial \Omega_{+}}}{ 
\left\|\partial_{n}\left(e^{-\frac{f}{h}}u_{1}^{(0)}\right)\right\|_{L^{1}(\partial
  \Omega_{+})}}$, in the spirit of the
approximation~\eqref{eq.l10Hess} for $\lambda_1^{(0)}(\Omega_+)$. This
approximation uses second order Taylor expansions of $f$ around the local minima
  $U_{k}^{(1)}$, see Equation~\eqref{eq.denssortHess} below. More
  precisely, this approximation writes:
  \begin{equation}
    \label{eq.combconvgaus}
- \frac{\partial_{n}\left[e^{-\frac{f}{h}}u_{1}^{(0)}\right]\big|_{\partial \Omega_{+}}}{ 
\left\|\partial_{n}\left[e^{-\frac{f}{h}}u_{1}^{(0)}\right]\right\|_{L^{1}(\partial
  \Omega_{+})}}= \frac{\sum_{k=1}^{m_{1}^{D}(\Omega_{+}\setminus\overline{\Omega_{-}})}t_{k}(h)G_{k}(h)}{\sum_{k=1}^{m_{1}^{D}(\Omega_{+}\setminus\overline{\Omega_{-}})} t_{k}(h)
  }    + {\mathcal O}(h)
  \end{equation}
where the $G_{k}(h)$are Gaussian densities centered at the
$U_{k}^{(1)}$'s and the weights
  $t_{k}(h)$ are such that  $\lim_{h\to
    0}h\log t_{k}(h)=-f(U_{k}^{(1)})$. When
  $f\big|_{\partial_{\Omega_{+}}}$ has a unique global minimum, the
  sums in
  \eqref{eq.l10Hess} and \eqref{eq.combconvgaus} reduce to a
  single term. 
\end{remarque}
\begin{remarque}
As explained in Remark~\ref{se.assmor+} above, it is again possible to write a
set of assumptions in terms of $\Omega_+$ only. 
In particular, the
results of Proposition~\ref{pr.morseresult} hold under the Hypotheses~\ref{hyp.3},~\ref{hyp.4},~\ref{hyp.5} and the three
 assumptions~\eqref{eq.dnf+},~\eqref{eq.conddag+}
 and~\eqref{eq:dnf_omega+}. 
\end{remarque}
\begin{remarque}
It is possible to extend our analysis to the case of an
$h$-dependent function $f=f_{h}$ such that our assumptions are
verified with uniform constants. For example, the results hold if the values
$f(U_{k}^{(1)})$ of $f$ at the local minima $U_{k}^{(1)}$ are moved in a $\mathcal{O}(h)$ range
without changing $f-f(U_{k}^{(1)})$ locally. This would change the
coefficients $t_{k}(h)$ in~\eqref{eq.combconvgaus} accordingly by 
$\mathcal{O}(1)$ factors.
\end{remarque}
All this section will be devoted to the proof of
Proposition~\ref{pr.morseresult}. Let us first conclude the proof of
Theorem~\ref{th.main2} using the result of Proposition~\ref{pr.morseresult}.
\begin{proof}[Proof of Theorem~\ref{th.main2}]
Let $f$ be a function such that  Hypotheses~\ref{hyp.0}, \ref{hyp.3},
\ref{hyp.1} and~\ref{hyp.2} are satisfied. Let us assume moreover that 
$f\big|_{\partial\Omega_{+}}$ is a Morse function  and $\partial_n f >
0$ on $\partial \Omega_{+}$. It is possible to build a $\mathcal{C}^{\infty}$ function
$\tilde{f}$ such that $\tilde{f}=f$ on $\Omega_{+}\setminus
\Omega_{-}$, and the Hypotheses~\ref{hyp.0},~\ref{hyp.3},~\ref{hyp.4},
\ref{hyp.5} and the condition~\eqref{eq.conddag} are
satisfied by $\tilde{f}$. This relies in particular on the fact that
 Morse functions are dense in $\mathcal{C}^{\infty}$ functions. The
 condition~\eqref{eq.conddag} may require to slightly change the
 local minimal values of the Morse function $\tilde{f}$. 

The function $\tilde{f}$ now fulfills all the requirements of
Proposition~\ref{pr.morseresult} and thus, with obvious notation,
$$\tilde\lambda_{1}^{(0)}(\Omega_{+})
=
\frac{\int_{\partial \Omega_{+}} 2\partial_{n} f(\sigma)
  e^{-2\frac{ f(\sigma)}{h}}\,d\sigma}{
\int_{\Omega_{+}}e^{-2\frac{\tilde f(x)}{h}}\,dx} (1+\mathcal{O}(h))$$
and
$$
- \frac{\partial_{n}\left[e^{-\frac{\tilde f}{h}} \tilde u_{1}^{(0)}\right]\big|_{\partial \Omega_{+}}}{ 
\left\|\partial_{n}\left[e^{-\frac{\tilde f}{h}} \tilde u_{1}^{(0)}\right]\right\|_{L^{1}(\partial
  \Omega_{+})}}
=
\frac{(2\partial_{n} f) e^{-\frac{2 f}{h}}\big|_{\partial \Omega_{+}}
}{
\left\|
(2\partial_{n} f)
  e^{-\frac{2f}{h}}\right\|_{L^{1}(\partial \Omega_{+})}
}
 + \mathcal{O}(h)\quad \text{in}~L^{1}(\partial \Omega_{+})\,.
$$
Here, we have used the fact that $\tilde{f}=f$ on $\Omega_{+}\setminus
\Omega_{-}$. Notice that the function $\tilde{f}$ satisfies the Hypotheses~\ref{hyp.0}, \ref{hyp.3},
\ref{hyp.1} and~\ref{hyp.2}, by the results of the previous section. We thus
conclude the proof by referring to Corollary~\ref{co.variaf}.
\end{proof}

The proof of Proposition~\ref{pr.morseresult} is done in two steps: we
first apply Theorem~\ref{th.main} using a very specific basis of $\Ran
1_{[0,\nu(h)]}(\Delta_{f,h}^{D,(1)}(\Omega_{+}\setminus\overline{\Omega_{-}}))$
to get estimates of $\lambda_1^{(0)}(\Omega_+)$ and $- \frac{\partial_{n}\left(e^{-\frac{f}{h}}u_{1}^{(0)}\right)\big|_{\partial \Omega_{+}}}{ 
\left\|\partial_{n}\left(e^{-\frac{f}{h}}u_{1}^{(0)}\right)\right\|_{L^{1}(\partial
  \Omega_{+})}}$ in terms of second order Taylor expansions of $f$
around the local minima $U_k^{(1)}$ (see Equations~\eqref{eq.l10Hess}
and~\eqref{eq.denssortHess}). We then show that these expansions coincide with the
formula~\eqref{eq.l10int} and~\eqref{eq.denssortpfn}.

In a preliminary step, let us explain how to build the almost orthonormal basis of $\Ran
1_{[0,\nu(h)]}(\Delta_{f,h}^{D,(1)}(\Omega_{+}\setminus\overline{\Omega_{-}}))$
that is needed to prove our results. This construction heavily relies
on the Morse assumption on $f$ and $f|_{\partial \Omega_+}$ (see Hypothesis~\ref{hyp.4}).
We need the results of \cite[Chapter 4]{HeNi} on approximate formulas for a basis of the
eigenspace of
$\Delta_{f,h}^{D,(1)}(\Omega_{+}\setminus\overline{\Omega_{-}})$
associated with $\mathcal{O}(e^{-\frac{c_{0}}{h}})$ eigenvalues (see also \cite{Lep2}
for a more general analysis). In all what follows, it is assumed that
Hypotheses~\ref{hyp.0},~\ref{hyp.4},~\ref{hyp.5} and the condition~\eqref{eq:dnf_omega+} hold. 
The one-forms of that basis are constructed via a WKB expansion around each local
minimum $U_{k}^{(1)}$ of $f\big|_{\partial \Omega_{+}}$ ($1\leq k\leq
m_{1}^{D}(\Omega_{+}\setminus\overline{\Omega_{-}})$).
In a neighborhood ${\mathcal V}_k$ of $U_{k}^{(1)}$, consider the
Agmon distance $\varphi_k$ to
$U_{k}^{(1)}$ (see Lemma~\ref{le.dist}). We assume that all the
${\mathcal V}_k$'s are disjoint subsets of $\overline{\Omega_+}
\setminus \overline{\Omega_-}$.  The function $\varphi_k$ satisfies the eikonal equation
$$
|\nabla \varphi_{k}|^{2}=|\nabla f|^{2}\,, \quad
\varphi_{k}\big|_{\partial \Omega_{+}}=(f-f(U_{k}^{(1)}))\big|_{\partial
  \Omega_{+}}\,,\quad
\partial_{n} \varphi_{k}\big|_{\partial \Omega_{+}}=-\partial_{n}
f\big|_{\partial \Omega_{+}}\,.
$$
In the neighborhood $\mathcal{V}_{k}$, one can build coordinates
$(x',x^{d})=(x^{1},\ldots, x^{d-1},x^{d})$ such that
\begin{itemize}
\item The open set $\Omega_{+}$ looks like a half-space
\begin{align*}
\Omega_{+}\cap \mathcal{V}_{k}&=\left\{(x',x^{d})\,,\, |x'|\leq
    r\,,\, x^{d}<0\right\}\,,\\
  \partial\Omega_{+}\cap \mathcal{V}_{k}&=\left\{(x',x^{d})\,,\, |x'|\leq
    r\,,\, x^{d}=0\right\}\,;
\end{align*}
\item The metric has the form 
$g_{d,d}(x)(dx^{d})^{2}+\sum_{i,j=1}^{d-1}g_{i,j}(x)dx^{i}dx^{j}$ with 
 $g_{i,j}(0)=\delta_{i,j}$  (notice that a different normalization of $g_{d,d}(0)$ was used in
  \cite{HeNi})\,;
\item The coordinates $(x',x^{d})$ are  Morse coordinates both for $f$
  and $\varphi_{k}$:
\begin{equation}
    \label{eq.fphi}
f(x)-f(U_{k}^{(1)})=\partial_{n}f(U_{k}^{(1)})x^{d}+\frac{1}{2}\sum_{j=1}^{d-1}\lambda_{j}(x^{j})^{2}\,,
\quad \varphi_{k}(x)=-\partial_{n}f(U_{k}^{(1)})x^{d}+\frac{1}{2}\sum_{j=1}^{d-1}\lambda_{j}(x^{j})^{2}
\end{equation}
where the $\lambda_{j}$ are the eigenvalues of $\Hess
(f\big|_{\partial \Omega_{+}})(U_{k}^{(1)})$\,.
\end{itemize}
In~\cite{HeNi} a local self-adjoint realization of
$\Delta_{f,h}^{(1)}$ around $U_{k}^{(1)}$ is introduced with the same
boundary conditions along $\partial\Omega_{+}$ as for $\Delta^{D,(1)}_{f,h}(\Omega_+)$, with
a unique exponentially small eigenvalue
$\zeta_{k}(h)=\mathcal{O}(e^{-\frac{c_{k}}{h}})$\,.
A corresponding approximate eigenvector is
given by the WKB expansion (in the limit of small $h$)
\begin{equation}
  \label{eq.wkb}
z_{k}^{wkb,(1)}(x,h)=a_{k}(x,h)e^{-\frac{\varphi_{k}(x)}{h}}\text{
  where } a_{k}(x,h)\sim
a_{k,0}(x)dx^{d}+
\sum_{\ell=1}^{\infty}b_{k,\ell}h^{\ell}
\end{equation}
with $b_{k,\ell}=\sum_{j=1}^{d}a_{k,\ell,j}(x)dx^{j}$ and
$a_{k,0}(0)=1$\,. The sign $\sim$ stands for the equality of
asymptotic expansions.  Let $z_{k}^{(1)}$ be the eigenvector
of the self-adjoint realization of $\Delta_{f,h}^{(1)}$ around
$U_{k}^{(1)}$ introduced above, associated with $\zeta_{k}(h)$ and normalized by
 $\mathbf{i}_{\partial_{x^{d}}}z_{k}^{(1)}(0)=
\mathbf{i}_{\partial_{x^{d}}}z_{k}^{wkb,(1)}(0)$. It is
shown in~\cite[Proposition~4.3.2-b)d)]{HeNi} that the estimates
\begin{align}
  \forall \alpha\in\nz^{d}\,,\,\exists C_{\alpha}>0, \exists
  N_{\alpha}\in\nz\,,\quad
&
|\partial_{x}^{\alpha}z^{(1)}(x)|\leq C_{\alpha}h^{-N_{\alpha}}e^{-\frac{\varphi_{k}(x)}{h}}\,,\label{eq.estimzk}
\\
\forall N\in\nz,
\forall \alpha\in\nz^{d}\,,
\exists C_{\alpha,N}>0\,,\quad
&
|\partial_{x}^{\alpha}(z_{k}^{wkb,(1)}-z_{k}^{(1)})(x)|\leq 
C_{N,\alpha}h^{N}e^{-\frac{\varphi_{k}(x)}{h}}\label{eq.estimzkwkb}
\end{align}
hold for all $x$ in a neighborhood
$\mathcal{V}_{k}'\subset \mathcal{V}_{k}$ of $U_{k}^{(1)}$.  Notice
that the one-forms $z_{k}^{wkb,(1)}$ 
and $z_{k}^{(1)}$ are real-valued.
By taking a cut-off function $\chi_{k}\in
\mathcal{C}^{\infty}_{0}(\mathcal{V}_{k}')$ with $\chi_{k}\equiv 1$ in
a neighborhood of $U_{k}^{(1)}$ a normalized quasimode for
$\Delta^{D,(1)}_{f,h}(\Omega_+\setminus \Omega_-)$ is given
by
$$
w_{k}^{(1)}=\frac{\chi_{k}z_{k}^{(1)}}{\|\chi_{k}z_{k}^{(1)}\|_{L^{2}(\mathcal{V}_{k}')}}.
$$
The set of functions $(w_k^{(1)})_{k\in \left\{1,\ldots,
    m_{1}^{D}(\Omega_{+}\setminus\overline{\Omega_{-}})\right\}}$ is
orthonormal, owing to the disjoint supports of the functions $(\chi_{k})_{k\in \left\{1,\ldots,
    m_{1}^{D}(\Omega_{+}\setminus\overline{\Omega_{-}})\right\}}$.
According to \cite[Proposition~6.6]{HeNi}, those quasimodes belong to
the form domain of $\Delta_{f,h}^{D,(1)}(\Omega_{+} \setminus \overline{\Omega_-})$ and there exist
two constants $C,c>0$ such that
\begin{equation}
  \label{eq.estimwkform}
\|d_{f,h}w_{k}^{(1)}\|_{L^{2}(\Omega_{+})}^{2}
+\|d_{f,h}^{*}w_{k}^{(1)}\|_{L^{2}(\Omega_{+})}^{2}\leq Ce^{-\frac{c}{h}}
\end{equation}
holds for all $k\in \left\{1,\ldots,
  m_{1}^{D}(\Omega_{+}\setminus\overline{\Omega_{-}})\right\}$\,. In addition, the estimates \eqref{eq.estimzk} and \eqref{eq.estimzkwkb} with
$\zeta_{k}(h)=\mathcal{O}(e^{-\frac{c_{k}}{h}})$ imply that
the $w_{k}$'s solve
\begin{equation}
  \label{eq.eqwk}
  \left\{
    \begin{array}[c]{l}
\Delta_{f,h}^{(1)}w_{k}^{(1)}=r_{k} \text{ on } \Omega_+\setminus
\overline{\Omega_-},\\
  \mathbf{t} w_{k}^{(1)}\big|_{\partial \Omega_{+}\cup \partial \Omega_{-}}=0\,,\quad
  \mathbf{t}d_{f,h}^{*}w_{k}^{(1)}\big|_{\partial \Omega_{-}}=0\,,\quad
  \mathbf{t}d_{f,h}^{*}w_{k}^{(1)}\big|_{\partial \Omega_{+}}=\rho_{k}\,,
\end{array}
\right.
\end{equation}
where $r_k$ and $\rho_k$ satisfy:
\begin{equation}\label{eq.estimrestewk}
\forall n\in\nz\,,\,\exists C_{n}>0\,,\,\forall k\in \left\{1,\ldots,
  m_{1}^{D}(\Omega_{+}\setminus \overline{\Omega_{-}})\right\}\,,
\|r_{k}\|_{W^{n,2}(\overline{\Omega_{+}}\setminus \Omega_{-})}+\|\rho_{k}\|_{W^{n+1/2,2}(\partial\Omega_{+})}\leq C_{n}e^{-\frac{c'}{h}}\,,
\end{equation} 
for some $c'>0$. The construction of the almost orthonormal basis of $\Ran
1_{[0,\nu(h)]}(\Delta_{f,h}^{D,(1)}(\Omega_{+}\setminus\overline{\Omega_{-}}))$
is completed with the next lemma.
\begin{lemme}
\label{pr.wk} Assume Hypotheses~\ref{hyp.0}, \ref{hyp.4},
\ref{hyp.5} and the condition~\eqref{eq:dnf_omega+}
and set for any $k\in \left\{1,\ldots,
  m_{1}^{D}(\Omega_{+}\setminus\overline{\Omega_{-}})\right\}$ 
$$
u_{k}^{(1)}=1_{[0,\nu(h)]}(\Delta_{f,h}^{D,(1)}(\Omega_{+}\setminus\overline{\Omega_{-}}))w_{k}^{(1)}
\,.
$$
Then $(u_{k}^{(1)})_{k\in \left\{1,\ldots,  m_{1}^{D}(\Omega_{+}\setminus\overline{\Omega_{-}})\right\}}$ is an
almost orthonormal basis of $\Ran
1_{[0,\nu(h)]}(\Delta_{f,h}^{D,(1)}(\Omega_{+}\setminus\overline{\Omega_{-}}))$\,.\\
Moreover, it holds: $\exists c>0, \, \forall n\in \nz, \, \exists
C_{n}>0, \, \forall 
k\in \left\{1,\ldots,
  m_{1}^{D}(\Omega_{+}\setminus\overline{\Omega_{-}})\right\},
$
\begin{equation}\label{eq:wn2}
\|u_{k}^{(1)}-w_{k}^{(1)}\|_{W^{n,2}(\Omega_{+}\setminus\overline{\Omega_{-}})}\leq C_{n}e^{-\frac{c}{h}},
\end{equation}
for all sufficiently small $h$.
\end{lemme}
\begin{proof}
Let us introduce
$v_{k}^{(1)}=u_{k}^{(1)}-w_{k}^{(1)}$ for $k\in \left\{1,\ldots,
  m_{1}^{D}(\Omega_{+}\setminus\overline{\Omega_{-}})\right\}$. The one-form $v_{k}^{(1)}$
belongs to the form domain of
$\Delta_{f,h}^{D,(1)}(\Omega_{+}\setminus\overline{\Omega_{-}})$ and
the spectral theorem leads to
\begin{align*}
\nu(h)\|v_{k}^{(1)}\|_{L^{2}(\Omega_{+}\setminus\overline{\Omega_{-}})}^{2}
&\leq \|d_{f,h} v_{k}^{(1)}\|_{L^{2}(\Omega_{+}\setminus\overline{\Omega_{-}})}^{2}
+\|d_{f,h}^{*}v_{k}^{(1)}\|_{L^{2}(\Omega_{+}\setminus\overline{\Omega_{-}})}^{2}\leq Ce^{-\frac{c}{h}}
\leq Ce^{-\frac{c_{1}}{h}}
\end{align*}
owing to \eqref{eq.estimwkform} and
$\sigma\left(\Delta_{f,h}^{D,(1)}(\Omega_{+}\setminus\overline{\Omega_{-}})\right)\cap
[0,\nu(h)]\subset [0,e^{-\frac{c_{0}}{h}}]$. 
%
%
With \eqref{eq.hypnu} this implies $\|v_{k}^{(1)}\|_{
  L^{2}(\Omega_{+}\setminus\overline{\Omega_{-}})}^{2}=\mathcal{O}(e^{-\frac{c_{2}}{h}})$\,.
By using
$$
h^{2}\left(\|dv_{k}^{(1)}\|_{L^{2}(\Omega_{+}\setminus\overline{\Omega_{-}})}^{2}+
  \|d^{*}v_{k}^{(1)}\|_{L^{2}(\Omega_{+}\setminus\overline{\Omega_{-}})}^{2}\right)
\leq 2 \|d_{f,h}v_{k}^{(1)}\|_{L^{2}(\Omega_{+}\setminus\overline{\Omega_{-}})}^{2}+
2\|d_{f,h}^{*}v_{k}^{(1)}\|_{L^{2}(\Omega_{+}\setminus\overline{\Omega_{-}})}^2
+C\|v_{k}^{(1)}\|_{L^{2}(\Omega_{+}\setminus\overline{\Omega_{-}})}^{(2)}\,,
$$
we obtain
$$
\|v_{k}^{(1)}\|_{W^{1,2}(\Omega_{+}\setminus\overline{\Omega_{-}})}^{2}=
\mathcal{O}(h^{-2}e^{-\frac{c_{2}}{h}})= \mathcal{O}(e^{-\frac{c_{2}}{2h}})\,.
$$
Thus, the almost orthonormality property  of $(u_{k}^{(1)})_{k\in \left\{1,\ldots,  m_{1}^{D}(\Omega_{+}\setminus\overline{\Omega_{-}})\right\}}$ is  due to the
orthonormality of $(w_{k}^{(1)})_{k\in \left\{1,\ldots,
  m_{1}^{D}(\Omega_{+}\setminus\overline{\Omega_{-}})\right\}}$\,.

The $W^{n,2}$-estimates~\eqref{eq:wn2} are then obtained by a
bootstrap argument (induction on $n$), using the
elliptic regularity of the Hodge Laplacian.
With $\Delta_{f,h}^{D,(1)}(\Omega_{+}\setminus
\Omega_{-})u_{k}^{(1)}= \tilde{\mathcal{O}}(e^{-\frac{c_{0}}{h}})$ in
any $W^{n,2}$ (see Lemma~\ref{le.L2W}), the equation~\eqref{eq.eqwk} leads to 
$$
\left\{
  \begin{array}[c]{l}
   \Delta_H
    v_{k}^{(1)}=r'_{k}(h)-h^{-2} \left(\Delta_{f,h}-h^2 \Delta_H\right)v_{k}^{(1)}\,,\\
  \mathbf{t} v_{k}^{(1)}\big|_{\partial \Omega_{+}\cup \partial \Omega_{-}}=0\,,\quad
 \mathbf{t}d^{*}v_{k}^{(1)}\big|_{\partial \Omega_{-}}=0\,,\quad
  \mathbf{t}d^{*}v_{k}^{(1)}\big|_{\partial
    \Omega_{+}}=-h^{-1}\rho_{k}-h^{-1}\mathbf{i}_{\nabla f}v_{k}^{(1)}\,,
  \end{array}
\right.
$$
where
$\|r'_{k}(h)\|_{W^{n,2}(\Omega_{+}\setminus\overline{\Omega_{-}})}$
satisfies the same estimate~\eqref{eq.estimrestewk} as
$\|r_{k}(h)\|_{W^{n,2}(\Omega_{+}\setminus\overline{\Omega_{-}})}$\,. Using
the fact that the zeroth order differential  operator 
$(\Delta_{f,h}-h^{2} \Delta_H)=|\nabla
f|^{2}+h(\mathcal{L}_{\nabla f}+\mathcal{L}_{\nabla f}^{*})$ is
bounded in $L^\infty$-norm, we thus obtain the
$W^{n,2}$-estimates~\eqref{eq:wn2} by induction on $n$.
\end{proof}
We are now in position to prove Proposition~\ref{pr.morseresult}. 
\begin{proof}[Proof of Proposition~\ref{pr.morseresult}.]
Let us apply Theorem~\ref{th.main} and its Corollary~\ref{co.trace} to the almost orthonormal basis $(u_{k}^{(1)})_{1\leq k\leq
   m_{1}^{D}(\Omega_{+}\setminus\overline{\Omega_{-}})}$ introduced in
 Lemma~\ref{pr.wk} (see Remark~\ref{rem:almost_orthonormal}). From the
 estimate~\eqref{eq:wn2} and the fact that
 $\lim_{h\to 0}h\log\lambda_{1}^{(0)}(\Omega_{+})=-2\kappa_{f}$
, we deduce
\begin{align*}
  &\lambda_{1}^{(0)}(\Omega_{+})=
\frac{h^{2}\sum_{k=1}^{m_{1}^{D}(\Omega_{+}\setminus\overline{\Omega_{-}})}
\left(\int_{\partial\Omega_{+}} e^{-\frac{f}{h}}\mathbf{i}_{n}w_{k}^{(1)}(\sigma)~d\sigma\right)^{2}}{
\int_{\Omega_{+}}
e^{-\frac{2f(x)}{h}}~dx} (1+\mathcal{O}(e^{-\frac{c}{h}}))\,,\\
&
\partial_{n}u_{1}^{(0)}\big|_{\partial \Omega_{+}}=
-\sum_{k=1}^{m_{1}^{D}(\Omega_{+}\setminus\overline{\Omega_{-}})}\frac{\int_{\partial
    \Omega_{+}}e^{-\frac{f(\sigma)}{h}}\mathbf{i}_{n}
  w_{k}^{(1)}(\sigma)~d\sigma}{\left(\int_{\Omega_{+}}e^{-\frac{2f(x)}{h}}~dx\right)^{1/2}}\mathbf{i}_{n}w_{k}^{(1)}+ \mathcal{O}(e^{-\frac{\kappa_{f}+c}{h}})\,,
\end{align*}
where the last remainder term is measured in $W^{n,2}(\partial
\Omega_{+})$-norm for any $n \in \nz$.
In particular, we deduce
$$\frac{e^{-\frac{f}{h}}}{\left(\int_{\Omega_{+}}e^{-\frac{2f(x)}{h}}~dx\right)^{1/2}}\partial_{n}u_{1}^{(0)}\big|_{\partial
\Omega_{+}}=
-\sum_{k=1}^{m_{1}^{D}(\Omega_{+}\setminus\overline{\Omega_{-}})}
\left(\int_{\partial
    \Omega_{+}} 
  \theta_{k}(\sigma)\,d\sigma\right)
\theta_{k}+\mathcal{O}(e^{-\frac{2\kappa_{f}+c}{h}})\quad
\text{in}~L^{1}(\partial \Omega_{+})$$
and
\begin{equation}\label{eq:lambda_theta}
\lambda_{1}^{(0)}(\Omega_{+})=h^{2}\sum_{k=1}^{m_{1}^{D}(\Omega_{+}\setminus\overline{\Omega_{-}})}
\left(\int_{\partial
    \Omega_{+}}\theta_{k}~d\sigma\right)^{2}(1+\mathcal{O}(e^{-\frac{c}{h}}))\,,
\end{equation}
where $\displaystyle{\theta_{k}=\frac{e^{-\frac{f}{h}}}{\left(\int_{\Omega_{+}}e^{-\frac{2f(x)}{h}}~dx\right)^{1/2}}
\mathbf{i}_{n}w_{k}^{(1)}\big|_{\partial\Omega_{+}}}$\,.

Using $\partial_{n}u_{1}^{(0)}\big|_{\partial\Omega_{+}}\leq 0$ and
the fact that the $\theta_{k}$'s have disjoint supports, the following
estimates hold:
\begin{align*}
\left(\int_{\Omega_{+}}e^{-\frac{2f(x)}{h}}~dx\right)^{-1/2}
\left\|e^{-\frac{f}{h}}\partial_{n}u_{1}^{(0)}\big|_{\partial
\Omega_{+}}\right\|_{L^{1}(\partial \Omega_{+})}
&=
\sum_{k=1}^{m_{1}^{D}(\Omega_{+}\setminus\overline{\Omega_{-}})}
\left(\int_{\partial
    \Omega_{+}}\theta_{k}(\sigma)\,d\sigma\right)^{2}+{\mathcal{O}}(e^{-\frac{2\kappa_{f}+c}{h}})
\\
&=
h^{-2}\lambda_{1}^{(0)}(\Omega_{+})(1+\tilde{\mathcal{O}}(e^{-\frac{c}{h}}))\,.
\end{align*}
In the last equality, we used~\eqref{eq.applambda10} to
get a lower bound on $\lambda_{1}^{(0)}(\Omega_{+})$.
By recalling that the Dirichlet boundary condition
$u_{1}^{(0)}\big|_{\partial \Omega_{+}}=0$ implies
$$
\partial_{n}\left[e^{-\frac{f}{h}}u_{1}^{(0)}\right]\big|_{\partial
    \Omega_{+}}
= 
e^{-\frac{f}{h}}
\partial_{n}u_{1}^{(0)}\big|_{\partial \Omega_{+}}\,,
$$
we thus get
\begin{equation}\label{eq:dnu_theta}
-\frac{\partial_{n}\left[e^{-\frac{f}{h}}u_{1}^{(0)}\right]\big|_{\partial \Omega_{+}}}{ 
\left\|\partial_{n}\left[e^{-\frac{f}{h}}u_{1}^{(0)}\right]\right\|_{L^{1}(\partial
  \Omega_{+})}}
=
\frac{\sum_{k=1}^{m_{1}^{D}(\Omega_{+}\setminus\overline{\Omega_{-}})}
\left(\int_{\partial \Omega_{+}}\theta_{k}~d\sigma\right)\theta_{k}
}{\sum_{k=1}^{m_{1}^{D}(\Omega_{+}\setminus\overline{\Omega_{-}})}
 \left(\int_{\partial \Omega_{+}}\theta_{k}~d\sigma\right)^{2}
 }
+
\tilde{\mathcal{O}}(e^{-\frac{c}{h}})\quad\text{in}~L^{1}(\partial\Omega_{+})\,.
\end{equation}

In order to get estimates from~\eqref{eq:lambda_theta}
and~\eqref{eq:dnu_theta} in terms of $f$, it remains to approximate the quantities
 $\theta_{k}$  and $\int_{\partial\Omega_{+}}\theta_{k}~d\sigma$ in
 the limit $h \to 0$. Recall that
$$
\theta_{k}=\frac{e^{-\frac{f}{h}}}{\left(\int_{\Omega_{+}}e^{-\frac{2f(x)}{h}}~dx\right)^{1/2}}
\mathbf{i}_{n}w_{k}^{(1)}\big|_{\partial\Omega_{+}} \text{ and }
w_{k}^{(1)}=\frac{\chi_{k}z_{k}^{(1)}}{\|\chi_{k}z_{k}^{(1)}\|_{L^2({\mathcal
    V}_k')}}.
$$
The estimates are obtained using the Laplace method and the WKB
expansion \eqref{eq.wkb} together with~\eqref{eq.estimzkwkb} to approximate $z_{k}^{(1)}$\,.
\begin{description}
\item[$\bullet~\int_{\Omega_{+}}e^{-\frac{2f(x)}{h}}\,dx:$] A direct application
  of the Laplace method gives
$$
\int_{\Omega_{+}}e^{-\frac{2f(x)}{h}}~dx=
e^{-2\frac{f(U_{0})}{h}}(\pi
h)^{\frac{d}{2}}\left( \det(\Hess f)(U_{0})\right)^{-1/2}(1+\mathcal{O}(h))\,,
$$ 
where $U_{0}$ is the unique global minimum of $f$.
\item[$\bullet~\|\chi_{k}z_{k}^{(1)}\|_{L^{2}(\mathcal{V}_{k}')}:$]
  Recall the coordinates around $U_{k}^{(1)}$ used in~\eqref{eq.fphi}
  and~\eqref{eq.wkb}. Using these coordinates
  and~\eqref{eq.estimzkwkb}, there is a
  $\mathcal{C}^{\infty}_{0}(\left\{x^{d}\leq 0\right\})$ function
  $\alpha(x,h)\sim\sum_{k=0}^{\infty}\alpha_{k}(x)h^{k}$ with
  $\alpha_{0}(0)=1$, such that
  \begin{align*}
   \|\chi_{k}z_{k}^{(1)}\|_{L^{2}(\mathcal{V}_{k}')}^{2}&=
 \int_{\left\{x^{d}\leq
     0\right\}}e^{-2\frac{\varphi_{k}(x)}{h}}\alpha(x,h)~dx^{1}\ldots
 dx^{d} 
\\
&=
 \int_{\left\{x^{d}\leq
     0\right\}}e^{\frac{2\partial_{n}f(U_{k}^{(1)}) x^d}{h}}
e^{-\frac{\sum_{j=1}^{d-1}\lambda_{j}(x^{j})^{2}}{h}}\alpha(x,h)~dx^{1}\ldots
 dx^{d} \\
&= \frac{h}{2\partial_{n}f(U_{k}^{(1)})}
\frac{(\pi
  h)^{\frac{d-1}{2}}}{\sqrt{\lambda_{1}\ldots\lambda_{d-1}}}
(1+\mathcal{O}(h))
\\
&=
\frac{(\pi h)^{\frac{d+1}{2}}}{2\pi\partial_{n}f(U_{k}^{(1)})
  \left(\det(\Hess f\big|_{\partial
    \Omega_{+}})(U_{k}^{(1)})\right)^{1/2}}(1+\mathcal{O}(h))\,.
  \end{align*}
We applied the
  Laplace method to get the estimate of the integral (using the fact
  that $\partial_{n}f(U_{k}^{(1)})>0$ by~\eqref{eq:dnf_omega+}).
\item[$\bullet~\theta_{k}:$] On the one hand, using
  $f(x)=f(U_{k}^{(1)})+\partial_{n}f(U_{k}^{(1)})x^{d}+\frac{1}{2}\sum_{j=1}^{d-1}\lambda_{j}(x^{j})^{2}$
  in a neighborhood of $U_k^{(1)}$
 (see~\eqref{eq.fphi}), we have: on $\partial \Omega_+$ (so that $x^d=0$),
$$
\chi_k \frac{e^{-\frac{f}{h}}}{
\left(\int_{\Omega_{+}}e^{-\frac{f(x)}{h}}~dx\right)^{1/2}}
= \chi_k \, e^{-\frac{f(U_{k}^{(1)})-f(U_{0})}{h}}
(\pi h)^{-\frac{d}{4}}\left(\det(\Hess f)(U_{0})\right)^{1/4}
e^{-\frac{\sum_{j=1}^{d-1}\lambda_{j}(x^{j})^{2}}{2h}} (1+\mathcal{O}(h))\,.
$$
On the other hand, the  function $\mathbf{i}_{n}w_{k}^{(1)}\big|_{\partial\Omega_{+}}=\frac{\chi_{k}\mathbf{i}_{n}z_{k}^{(1)}\big|_{\partial\Omega_{+}}}{\|\chi_{k}z_{k}^{(1)}\|_{L^2({\mathcal
      V}_k')}}
$ satisfies
$$
\mathbf{i}_{n}w_{k}^{(1)}\big|_{\partial\Omega_{+}}=
\chi_{k}\frac{\sqrt{2\pi\partial_{n}f(U_{k}^{(1)})}
\left(\det(\Hess f\big|_{\partial
  \Omega_{+}})(U_{k}^{(1)})\right)^{1/4}}{(\pi h)^{\frac{d+1}{4}}}
e^{-\frac{\sum_{j=1}^{d-1} \lambda_{j}(x^{j})^{2}}{2h}}(1+\mathcal{O}(h))\,.
$$
From these two estimates, $\theta_k$ satisfies
$$\theta_{k}=A_{k} \, \chi_{k} \,
e^{-\frac{\sum_{j=1}^{d-1}
    \lambda_{j}(x^{j})^{2}}{h}}(1+\mathcal{O}(h))$$
where $A_{k}= \frac{\sqrt{2\pi\partial_{n}f(U_{k}^{(1)})}}{(\pi h)^{\frac{2d+1}{4}} }
\left(\det(\Hess f\big|_{\partial
  \Omega_{+}})(U_{k}^{(1)})\right)^{1/4}\left(\det(\Hess
f)(U_{0})\right)^{1/4} e^{-\frac{f(U_{k}^{(1)})-f(U_{0})}{h}}
$.
\item[$\bullet~\displaystyle{\int_{\partial \Omega_+} \theta_{k}}:$]
  The Laplace method implies that
$$
\int e^{-\frac{\sum_{j=1}^{d-1}
    \lambda_{j}(x^{j})^{2}}{h}}~dx^{1}\ldots dx^{d-1} = \frac{(\pi h)^{\frac{d-1}{2}}}{\sqrt{\lambda_{1}\ldots
    \lambda_{d-1}}}(1 +\mathcal{O}(h))=
(\pi h)^{\frac{d-1}{2}}\left(\det(\Hess f\big|_{\partial \Omega_{+}})(U_{k}^{(1)})\right)^{-1/2}(1 +\mathcal{O}(h))\,.
$$ 
We thus obtain
$$
\int_{\partial \Omega_+} \theta_{k} =
\frac{\sqrt{2\pi \partial_{n}f(U_{k}^{(1)})}\left(\det(\Hess
  f)(U_{0})\right)^{1/4}}{(\pi h)^{\frac{3}{4}}\left(\det(\Hess f\big|_{\partial
    \Omega_{+}})(U_{k}^{(1)})\right)^{1/4}}
e^{-\frac{f(U_{k}^{(1)})-f(U_{0})}{h}} (1+\mathcal{O}(h))\,.
$$
\end{description}
Putting together the above information and using~\eqref{eq:lambda_theta}
and~\eqref{eq:dnu_theta} finally imply
$$
\lambda_{1}^{(0)}(\Omega_{+})
=
\sqrt{\frac{h\det(\Hess
  f)(U_{0})}{\pi}}\sum_{k=1}^{m_{1}^{D}(\Omega_{+}\setminus\overline{\Omega_{-}})}
\frac{2\partial_{n}f(U_{k}^{(1)})
}{\sqrt{\det(\Hess f\big|_{\partial
    \Omega_{+}})(U_{k}^{(1)})}}
e^{-2\frac{f(U_{k}^{(1)})-f(U_{0})}{h}}
(1+\mathcal{O}(h))\,,
$$
which is exactly~\eqref{eq.l10Hess} and
\begin{align}
\label{eq.denssortHess}
- \frac{\partial_{n}\left[e^{-\frac{f}{h}}u_{1}^{(0)}\right]\big|_{\partial \Omega_{+}}}{ 
\left\|\partial_{n}\left[e^{-\frac{f}{h}}u_{1}^{(0)}\right]\right\|_{L^{1}(\partial
  \Omega_{+})}}
&=
\frac{\sum_{k=1}^{m_{1}^{D}(\Omega_{+}\setminus\overline{\Omega_{-}})}
\partial_{n}f(U_{k}^{(1)})e^{-2\frac{f(U_{k}^{(1)})-f(U_{0})}{h}}
\, \chi_{k} \, e^{-\frac{\sum_{j=1}^{d-1}
    \lambda_{j}(x^{j})^{2}}{h}}
}{(\pi h)^{\frac{d-1}{2}}
\sum_{k'=1}^{m_{1}^{D}(\Omega_{+}\setminus\overline{\Omega_{-}}) }
\frac{\partial_{n}f(U_{k'}^{(1)})}{\sqrt{\det(\Hess
f\big|_{\partial\Omega_{+}})(U_{k'}^{(1)})}}
e^{-2\frac{f(U_{k'}^{(1)})-f(U_{0})}{h}}}(1+\mathcal{O}(h)).
\end{align}
We thus obtain estimates of $\lambda_1^{(0)}(\Omega_+)$ and $- \frac{\partial_{n}\left(e^{-\frac{f}{h}}u_{1}^{(0)}\right)\big|_{\partial \Omega_{+}}}{ 
\left\|\partial_{n}\left(e^{-\frac{f}{h}}u_{1}^{(0)}\right)\right\|_{L^{1}(\partial
  \Omega_{+})}}$ in terms of second order Taylor expansions of $f$
around the local minima $U_k^{(1)}$. This ends the first step of the proof.

Actually, the two estimates~\eqref{eq.l10Hess} and~\eqref{eq.denssortHess} can be rewritten in a simpler form using
again the Laplace method. By recalling the equality  $f(x)=f(U_{k}^{(1)})+\partial_{n}f(U_{k}^{(1)})x^{d}+\frac{1}{2}\sum_{j=1}^{d-1}\lambda_{j}(x^{j})^{2}$
  in a neighborhood of $U_k^{(1)}$, the
Laplace method gives by similar computations to those performed above
\begin{align*}
&\frac{\int_{\partial \Omega_{+}} 2\partial_{n}f(\sigma)
  e^{-2\frac{f(\sigma)}{h}}~d\sigma}{
\int_{\Omega_{+}}e^{-2\frac{f(x)}{h}}~dx}=\sqrt{\frac{h\det(\Hess
  f)(U_{0})}{\pi}}\sum_{k=1}^{m_{1}^{D}(\Omega_{+}\setminus\overline{\Omega_{-}})}
\frac{2\partial_{n}f(U_{k}^{(1)})
}{\sqrt{\det(\Hess f\big|_{\partial
    \Omega_{+}})(U_{k}^{(1)})}}
e^{-2\frac{f(U_{k}^{(1)})-f(U_{0})}{h}}
(1+\mathcal{O}(h))\,,\\
&
\frac{(2\partial_{n}f) e^{-\frac{2f}{h}}\big|_{\partial \Omega_{+}}
}{
\left\|
(2\partial_{n}f)
  e^{-\frac{2f}{h}}\right\|_{L^{1}(\partial \Omega_{+})}
}= 
\frac{\sum_{k=1}^{m_{1}^{D}(\Omega_{+}\setminus\overline{\Omega_{-}})}
\partial_{n}f(U_{k}^{(1)})e^{-2\frac{f(U_{k}^{(1)})-f(U_{0})}{h}}
\, \chi_{k} \, e^{-\frac{\sum_{j=1}^{d-1}
    \lambda_{j}(x^{j})^{2}}{h}}
}{(\pi h)^{\frac{d-1}{2}}
\sum_{k'=1}^{m_{1}^{D}(\Omega_{+}\setminus\overline{\Omega_{-}}) }
\frac{\partial_{n}f(U_{k'}^{(1)})}{\sqrt{\det(\Hess
f\big|_{\partial\Omega_{+}})(U_{k'}^{(1)})}}
e^{-2\frac{f(U_{k'}^{(1)})-f(U_{0})}{h}}}
 + \mathcal{O}(h)
\end{align*}
where the last remainder term is measured in
$L^{1}(\partial\Omega_{+})$-norm. Comparing with the two
estimates~\eqref{eq.l10Hess} and~\eqref{eq.denssortHess} above, we thus obtain~\eqref{eq.l10int}
and~\eqref{eq.denssortpfn}. This concludes the proof.
\end{proof}
\subsection{Beyond Morse assumptions}
\label{se.nonMorse}

In this section, we discuss Hypotheses~\ref{hyp.1} and~\ref{hyp.2} for
functions $f$ which do not fulfill the Morse assumptions of
Hypothesis~\ref{hyp.4} above. In Sections~\ref{se.onedim}
and~\ref{se.twoD}, we present two examples (respectively in dimension $1$ and $2$) of
functions $f$ which do not fulfill Hypothesis~\ref{hyp.4}, but for which
Hypotheses~\ref{hyp.1} and~\ref{hyp.2} still hold
true. Section~\ref{se.genrem} is first devoted to a few remarks that will be useful in  the
examples we will discuss below.

\subsubsection{General remarks}
\label{se.genrem}

First, we will use the duality between the
chain complexes associated with $d_{f,h}$ and $d_{f,h}^{*}$. More precisely,
conjugating with the Hodge $\star$-operator exchanges $p$ and $\dim M-p$
forms, $d$ and $d^{*}$, $f$ and $-f$,
Neumann and Dirichlet boundary conditions.
This was used extensively in \cite{Lep4,LNV}.

Second, the following Lemma will also be useful. It is a variant of
Proposition~\ref{pr.gapO-}.
\begin{lemme}
\label{le.apriori}
Let $\Omega$ be a regular bounded  domain of the Riemannian
manifold $(M,g)$ and let $f\in \mathcal{C}^{\infty}(\overline{\Omega})$ be
such that $(\nabla f)^{-1}(\left\{0\right\})$ has a unique non empty connected component in
$\Omega$\,.
\begin{itemize}
\item If $\partial_{n}f\big|_{\partial \Omega}>0$ then the two first
  eigenvalues of $\Delta_{f,h}^{N,(0)}(\Omega)$ satisfy
$$
\mu_{1}^{(0)}(\Omega)=0\quad\text{and}\quad \lim_{h\to 0}h\log
\mu_{2}^{(0)}(\Omega)= 0\,.
$$
\item If $\partial_{n}f\big|_{\partial \Omega}<0$ and $|\nabla
  f|^{2}-h \, \Delta f\geq 0$ in $\Omega$ for all $h\in (0,h_{0})$\,, then 
the first eigenvalue of $\Delta_{f,h}^{D,(0)}(\Omega)$ satisfies
$$
\lim_{h\to 0}h\log \lambda_{1}^{(0)}(\Omega)= 0\,.
$$
\end{itemize}
\end{lemme}
\begin{proof}
Up to the addition of a constant to the function $f$ (which only  affects the
normalization of $e^{-\frac{f}{h}}$), one may assume without loss of
generality that $f\equiv 0$ on $(\nabla f)^{-1}(\left\{0\right\})$
(using  the connectedness assumption on  $(\nabla f)^{-1}(\left\{0\right\})$). Then,
$f\geq 0$ in $\Omega$ when
$\partial_{n}f\big|_{\partial\Omega}>0$ and $f\leq 0$ when
$\partial_{n}f\big|_{\partial\Omega}< 0$\,.

The fact that $\mu_1^{(0)}(\Omega)=0$ is obvious, by considering the
associated eigenvector $e^{-\frac{f}{h}}$.  The Witten Laplacian acting on functions is the Schr{\"o}dinger type
  operator
$$
\Delta_{f,h}^{(0)}=-h^{2}\Delta +|\nabla f|^{2}-h (\Delta f)\,.
$$
Since the function $|\nabla f|^{2}-h \, \Delta f$ is uniformly bounded in
$\overline{\Omega}$, the two inequalities
$$
\limsup_{h\to 0}h\log \mu_{2}^{(0)}(\Omega)\leq 0
\quad\text{and}\quad
\limsup_{h\to 0}h\log \lambda_{1}^{(0)}(\Omega)\leq 0
$$
are consequences of the min-max principle. For the Dirichlet case,
 any fixed non zero function in
$\mathcal{C}^{\infty}_{0}(\Omega))$ will provide an $\mathcal{O}(1)$
Rayleigh quotient.
 For the Neumann case, consider
two regular functions $\chi_{1},\chi_{2}\in
\mathcal{C}^{\infty}_{0}(\Omega)$ such that $\supp \chi_{1}\cap
\supp\chi_{2}=\emptyset$ and $\|\chi_{1}\|_{L^{2}(\Omega)}=\|\chi_{2}\|_{L^{2}(\Omega)}=1$\,,
and take   $\psi_{h}=\alpha_{1}(h)\chi_{1}+\alpha_{2}(h)\chi_{2}$ such
that
$\|\psi_{h}\|_{L^{2}}^{2}=|\alpha_{1}(h)|^{2}+|\alpha_{2}(h)|^{2}=1$ and
$\langle \psi_{h}\,,\, e^{-\frac{f}{h}}\rangle_{L^2(\Omega)}=0$\,. We get $\langle
\psi_h\,,\, \Delta_{f,h}^{N,(0)} \psi_h\rangle_{L^2(\Omega)}=\mathcal{O}(1)$ and  the min-max
principle applied to
$\Delta_{f,h}^{N,(0)}(\Omega)$ on the orthogonal of $e^{-\frac{f}{h}}$
yields $\mu_{2}^{(0)}(\Omega)=\mathcal{O}(1)$ as $h\to 0$.
%
%


Let us first consider the case $\partial_{n}f\big|_{\partial \Omega}<0$ and $|\nabla
  f|^{2}-h \, \Delta f\geq 0$. It remains to prove that $
\liminf_{h\to 0}h\log \lambda_{1}^{(0)}(\Omega)\geq 0$.
Let $\omega$ be a normalized
  eigenfunction associated with $\lambda^{(0)}_1(\Omega)$:
  $\Delta_{f,h}^{D,(0)}(\Omega)\omega=\lambda_{1}^{(0)}(\Omega) \, \omega$ and $\|\omega\|_{L^{2}(\Omega)}=1$. 
Using Lemma~\ref{le.Agmon} with
  $\varphi=0$ and the Poincar\'e inequality, we get
$$
\lambda_{1}^{(0)}(\Omega)\geq h^{2} \| \nabla \omega
\|_{L^2(\Omega)}^2 \geq C_{\Omega} \, h^{2}\,.
$$
This concludes the proof in the case $\partial_{n}f\big|_{\partial \Omega}<0$ and $|\nabla
  f|^{2}-h \, \Delta f\geq 0$.

Let us now consider the case $\partial_{n}f\big|_{\partial
  \Omega}>0$. It remains to prove that $\liminf_{h\to 0}h\log
\mu_{2}^{(0)}(\Omega)\geq 0$. Let us reason by contradiction by
assuming that there exists $c>0$ and a sequence
$(h_{n})_{n\in\nz}$ such that 
$$
\lim_{n\to \infty}h_{n}=0\text{ and }
\mu_{2}^{(0)}(\Omega)\leq e^{-\frac{c}{h_{n}}}\quad\text{with}~c>0\,.
$$
Notice that $\mu_{2}^{(0)}(\Omega)$ depends on $n$. Let us introduce
$\omega_{n}$ a normalized eigenfunction associated with
$\mu_2^{(0)}(\Omega)$:
$\Delta_{f,h_{n}}^{N,(0)}\omega_{n}=\mu_{2}^{(0)}(\Omega)\omega_{n}$
and $\|\omega_n\|_{L^{2}(\Omega)}=1$. Notice that
$\int_{\Omega} \omega_{n}  e^{-\frac{f}{h_{n}}} =0$.
For $\varepsilon>0$, consider the open set 
$$
K_{\varepsilon}=\left\{x\in \overline{\Omega}\,, d(x,(\nabla
  f)^{-1}(\left\{0\right\})) < \varepsilon\right\}\,,
$$
such that $\overline{K_{\varepsilon}}$ is contained in $\Omega$ for $\varepsilon\in
(0,\varepsilon_{0})$ and $\varepsilon_0$ sufficiently small.
Take a partition of unity $\chi_{1}^{2}+\chi_{2}^{2}\equiv 1$ in
$\overline{\Omega}$ such that $\chi_{i}\in
\mathcal{C}^{\infty}(\overline{\Omega})$\,,
$\chi_{1}\equiv 1$ in a neighborhood of $K_{\varepsilon/2}$ and
$\supp \chi_{1}\subset K_{\varepsilon}$\,.
The IMS localization formula (see for example \cite{CFKS}) gives
\begin{align}
e^{-\frac{c}{h_n}}&\geq \langle 
\omega_n, \Delta_{f,h_n}^{N,(0)}(\Omega)\omega_n \rangle_{L^{2}(\Omega)} \nonumber\\
&= \langle \chi_{1}\omega_n\,,\,
\Delta_{f,h_n}^{N,(0)}(\Omega)\chi_{1}\omega_n\rangle_{L^{2}(\Omega)}
+
\langle \chi_{2}\omega_n\,,\,
\Delta_{f,h_n}^{N,(0)}(\Omega)\chi_{2}\omega_n\rangle_{L^{2}(\Omega)}
-h_n^{2}\sum_{j=1}^{2}\|\omega_n\nabla
\chi_{j}\|_{L^{2}(\Omega)}^{2}\,. \label{eq:ineq}
\end{align}
The lower bound (which is a consequence of $|\nabla f|^2 >0$ on $\supp
\chi_2$, and $\partial_n \chi_2 =0$ on $\partial \Omega$)
$$
\langle \chi_{2}\omega_{n}\,,\,
\Delta_{f,h_{n}}^{N,(0)}(\Omega)\chi_{2}\omega_{n}\rangle_{L^{2}(\Omega)}
\geq 
\langle \chi_{2}\omega_{n}\,,\,
|\nabla
f|^{2}\chi_{2}\omega_{n}\rangle_{L^{2}(\Omega)}-Ch_{n}\|\chi_{2}\omega_{n}\|_{L^{2}(\Omega)}^{2}
\geq \frac{1}{C_{\varepsilon}}\|\chi_{2}\omega_{n}\|_{L^{2}(\Omega)}^{2}
$$
for $n$ sufficiently large, together with~\eqref{eq:ineq} implies 
$$
\forall \delta>0,\,\forall\varepsilon>0,\, \exists
N\in\nz,\,
\forall n\geq N,\quad
\|\omega_{n}\|_{L^{2}(K_{\varepsilon})}^{2}\geq 1-\delta\,.
$$
Since $(\nabla  f)^{-1}(\left\{0\right\})$ is assumed connected and
for every point of the open set ${K}_{\varepsilon}$
the gradient flow associated with $f$ defines a path
to $(\nabla  f)^{-1}(\left\{0\right\})$,
${K}_{\varepsilon}$ is a connected open set.
The function $v_{n}=\omega_{n}\big|_{{K}_{\varepsilon}}$
belongs to $W^{1,2}({K}_{\varepsilon})$ with
$$
h_n^{2}e^{-2\frac{C\varepsilon^{2}}{h_{n}}}
\|de^{\frac{f}{h_{n}}}v_{n}\|_{L^{2}({K}_{\varepsilon})}^{2}
\leq
\|d_{f,h}v_{n}\|_{L^{2}({K}_{\varepsilon})}^{2}
\leq e^{-\frac{c}{h_{n}}}\,,
$$
thanks to the fact that $\exists C>0$,
$\forall x\in K_{\varepsilon}$, $0\leq f(x)\leq C\varepsilon^{2}$.
By choosing $\varepsilon>0$ so that $c-2C\varepsilon^{2}>0$, 
 the spectral gap estimate for the Neumann Laplacian in $\Omega$ (or
 equivalently the Poincar\'e-Wirtinger inequality in $\Omega$) provides a
constant $C_n$ such that
$$
\lim_{n\to\infty}
\|e^{\frac{f}{h_{n}}}v_{n}-C_n\|_{L^{2}({K}_{\varepsilon})}=0.$$
We thus deduce
$$
\lim_{n\to\infty}
\|\omega_{n}-C_ne^{-\frac{f}{h_{n}}}\|_{L^{2}({K}_{\varepsilon})}=0
\quad \text{with}\quad
\|\omega_{n}\|^{2}_{L^{2}({K}_{\varepsilon})}\geq
1-\delta\;,\; 
\|\omega_{n}\|_{L^{2}(\Omega)}=1\,.
$$
For $\delta<1$, this is in contradiction with $\int_{\Omega} \omega_{n}
e^{-\frac{f}{h_{n}}} = 0$.
\end{proof}
\subsubsection{A one-dimensional example}
\label{se.onedim}
In this section, we exhibit a simple one-dimensional example of a
function $f$ satisfying Hypotheses~\ref{hyp.1} and~\ref{hyp.2} though not
being a Morse function. An extension is then briefly discussed.
\begin{proposition}
  Consider a function $f\in
  \mathcal{C}^{\infty}(\overline{\Omega_{+}})$,
  $\Omega_{+}=(a_{+},b_{+})$, with $a_+ < b_+$ two real numbers,
  such that
  \begin{align*}
    &f^{-1}(0)=(f')^{-1}(0)=[a_{1},b_{1}]\quad -\infty<a_{+}<a_{1}\leq
    b_{1}< b_{+}<+\infty\\
   & f'(a_+)<0\quad\text{and}\quad f'(b_+)>0\,.
  \end{align*}
Then, for any $\Omega_{-}=(a_{-},b_{-})$ such that
$a_{+}<a_{-}<a_{1}\leq b_{1}<b_{-}<b_{+}$\,,
Hypotheses~\ref{hyp.1} and \ref{hyp.2} are valid with
$m_{0}^{N}(\Omega_{-})=1$\,, $m_{1}^{N}(\Omega_{-})=0$ and
$m_{1}^D(\Omega_{+}\setminus \overline{\Omega_{-}})=2$\,.
\end{proposition}
Notice that for this example, Hypotheses~\ref{hyp.0} and~\ref{hyp.3}
are also satisfied, which mean that the results of
Theorem~\ref{th.main} are valid.
\begin{proof}
On an interval $I$ with the Euclidean metric, the one-forms can be written 
$u^{(1)}=u_{1}(x)\,dx$\,. The Witten Laplacians $\Delta_{f,h}^{(p)}(I)$
with $p=0,1$ are then given by
\begin{align*}
  \Delta_{f,h}^{(0)}(I)u^{(0)}&=\left(-h^{2}\partial_{x,x}+|\partial_x f|^{2}-h
(\partial_{x,x} f)\right)u^{(0)}\\
  \Delta_{f,h}^{(1)}(I)(u_{1}dx)&=\left[\left(-h^{2}\partial_{x,x}+|\partial_x
    f|^{2}+h(\partial_{x,x} f)\right)u_{1}\right]\,dx\,.
\end{align*}
The Dirichlet boundary conditions are given by
$$
u^{(0)}=0 \text{ on } \partial I \quad\text{and}\quad
-h \partial_x u_{1}+ (\partial_x f) \, u_{1}=0  \text{ on } \partial I
$$
while the Neumann boundary conditions are given by
$$
h \partial_{x}u^{(0)}+(\partial_x f)u^{(0)}=0  \text{ on } \partial I\quad\text{and}\quad 
u_{1}=0  \text{ on } \partial I\,.
$$
This is a particular case of the general duality recalled at the
beginning of Section~\ref{se.genrem}. Let us now check Hypotheses~\ref{hyp.1} and \ref{hyp.2}.

First, $e^{-\frac{f}{h}}$
belongs to the kernel of $\Delta_{f,h}^{N,(0)}(\Omega_{-})$\,. A
direct application of Lemma~\ref{le.apriori} shows
that~\eqref{eq.nombN} holds for $p=0$ with
$m_{0}^{N}(\Omega_{-})=1$. Second, by the duality argument, proving
that~\eqref{eq.nombN} holds for $p=1$ with $m_{1}^{N}(\Omega_{-})=0$
is equivalent to proving that there is no exponentially small
eigenvalues for $\Delta^{D,(0)}_{-f,h}(\Omega_-)$ (notice that $f$ has
been changed to $-f$). But this is a consequence of the second part of
Lemma~\ref{le.apriori}, since $f$ is convex.
Finally note that the condition \eqref{eq.decayhyp} is empty since the
only exponentially small eigenvalue of
$\Delta_{f,h}^{N,(0)}(\Omega_{-})$ is $0$\,. This shows that
Hypothesis~\ref{hyp.1} holds.

The
open set $\Omega_{+}\setminus \overline{\Omega_{-}}$ is the disjoint union of
the two open intervals $(a_{+},a_{-})$ and $(b_{-},b_{+})$\,. On each
of them $\partial_x f$ does not vanish and the Morse assumptions of
Hypothesis~\ref{hyp.4} are
satisfied. On $(a_{+},a_{-})$ (resp. $(b_{-},b_{+})$), $f$ has one
generalized critical point of index $1$ at $a_{+}$ (resp. at
$b_{+}$). Therefore, using the results of~\cite{HeNi} (see Section~\ref{se.assmorse}),~\eqref{eq.nombD} holds with
$m_{1}^{D}(\Omega_{+}\setminus\overline{\Omega_{-}})=2$\,.  This shows that
Hypothesis~\ref{hyp.2} holds.
\end{proof}

It is not
 difficult to treat the case when $f\in
\mathcal{C}^{\infty}([a_+,b_+])$ has a finite number of critical intervals
$$
(f')^{-1}(\left\{0\right\})=\cup_{n=1}^{2N+1}[a_{n},b_{n}]\,,\quad
a_{+}<a_{1}\leq b_{1}\ldots <a_{2N+1}\leq b_{2N+1}<b_{+}\,,
$$
with $f'(a_{+})<0$ and $f'(b_{+})>0$\,. Again,
$\Omega_{-}=(a_{-},b_{-})$,  with $a_+ < a_- < a_1 < b_{2N+1} < b_- < b_+$.
The local problems around every $[a_{n},b_{n}]$ can
be studied with the help of the duality argument and
Lemma~\ref{le.apriori}. Using an argument based on a partition of
unity, one can check that~\eqref{eq.nombN} and~\eqref{eq.nombD} hold with
$m_{0}^{N}(\Omega_{-})=2N+1$\,,\, $m_{1}^{N}(\Omega_{-})=2N$ and
$m_{1}^{D}(\Omega_{+}\setminus\overline{\Omega_{-}})=2N+2$\,. Hypothesis~\ref{hyp.0}
is of course satisfied.
Ensuring that
Hypothesis~\ref{hyp.3} and  the condition~\eqref{eq.decayhyp} hold
then requires to correctly choose the heights of the critical values. They
hold 
for example when
$\max_{1\leq n\leq 2N+1}f(a_{i})< \min
\left\{f(a_{+}),\,f(b_{+})\right\}$ and when
 $f(a_{1})$ and $f(b_{2N+1})$ are the two smallest critical values.

\subsubsection{A two-dimensional example}
\label{se.twoD}

This example is inspired from the works of Bismut, Helffer and
Sj{\"o}strand \cite{Bis,HeSj6,HeSj7} about Bott
inequalities.
We consider the following $\mathcal{C}^{\infty}$ radial functions in $\rz^{2}$
\begin{align*}
  &\varphi_{in}(x)=e^{-\frac{1}{(|x|^{2}-1)^{2}}}1_{[0,1]}(|x|)\,,\\
 &
 \varphi_{ext}\equiv 0~\text{for}~|x|\leq 1\,,\quad
 \varphi_{ext}~\text{strictly~convex~in}\left\{|x|>1\right\}\,.
\end{align*}
The domain $\Omega_{+}$ is the disc $D((-R,0),2R)$ and $\Omega_{-}$
the disc $D((-R,0),2R-1)$ with $R>3$. The function
$f$ is defined by $f(x)=\varphi_{in}(x)+\varphi_{ext}(\frac{x}{2})$.
The level sets of the function $f$ are represented on Figure~\ref{fig:2D_degenere}.




\begin{figure}[htbp]
\centerline{\includegraphics[width=0.5\textwidth]{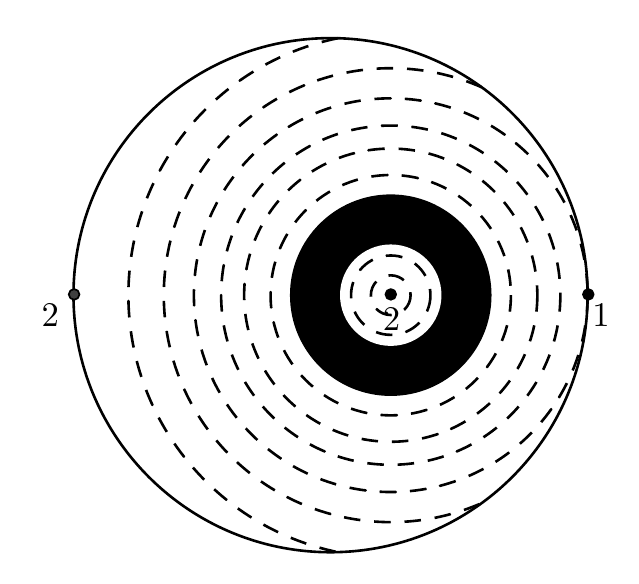}}
\caption{Only $\Omega_{+}$ is represented. The level sets of $f$
are represented by dashed lines. The black area is the $0$ level
set. The dots indicate the generalized critical points, together with
there indices (for Dirichlet boundary conditions). }\label{fig:2D_degenere}
\end{figure}

\begin{proposition}
\label{pr.exampdeg2D}
When $R>3$ is chosen large enough, the above triple
$(\Omega_{+},\Omega_{-},f)$ fulfills
 Hypotheses~\ref{hyp.0},~\ref{hyp.3},~\ref{hyp.1} and~\ref{hyp.2} with
 $m_{0}^{N}(\Omega_{-})=1$\,, $m_{1}^{N}(\Omega_{-})=1$ and
 $m^D_{1}(\Omega_{+}\setminus \overline{\Omega_{-}})=1$\,.
\end{proposition}
\begin{proof}
Thanks to the convexity assumption on $x \mapsto \varphi_{ext}(\frac{x}{2})$ and its local
 behavior around $\left\{|x|=2\right\}$, Hypotheses~\ref{hyp.0} and~\ref{hyp.3}
 holds for $R>3$ large enough.

 The choice of non $0$-centered disks 
for $\Omega_{+}$ and $\Omega_{-}$ while $f$ is a radial function
implies that $f|_{\partial \Omega_{+}}$ has a unique local minimum and
therefore, using the results recalled in Section~\ref{se.assmorse},~\eqref{eq.nombD} is
satisfied with
$m_{1}^D(\Omega_{+}\setminus\overline{\Omega_{-}})=1$\,. This shows that
Hypothesis~\ref{hyp.2} holds.

The fact that~\eqref{eq.nombN} holds for $p=0$ with $m_{0}^{N}(\Omega_{-})=1$ is a direct application of
Lemma~\ref{le.apriori}. This also implies that the condition \eqref{eq.decayhyp} is
void. It only remains to prove that~\eqref{eq.nombN} holds for $p=1$
with $m_{1}^{N}(\Omega_{-})=1$. We will actually prove that~\eqref{eq.nombN} holds for $p=2$ with $m_{2}^{N}(\Omega_{-})=1$. Then the quasi-isomorphism  with the
absolute cohomology of the disc (see
Section~\ref{se.restric}) 
%
%
gives $m_{2}^{N}(\Omega_{-})-m_{1}^{N}(\Omega_{-})+m_{0}^{N}(\Omega_{-})=1$\,,
which indeed implies $m_{1}^{N}(\Omega_{-})=1$\,. Moreover, by the
duality argument,~\eqref{eq.nombN} holds for $p=2$ with
$m_{2}^{N}(\Omega_{-})=1$ if~\eqref{eq.nombD} holds for $p=0$ with
$m_{0}^{D}(\Omega_{-})=1$, $f$ being changed into $-f$\,. The proof of
this claim will conclude the demonstration.

In the rest of this proof $m_{0}^{D}(\Omega_{-})$ denotes the number
of small eigenvalues for $\Delta^{D,(0)}_{-f,h}(\Omega_-)$. The function
$-f$ has a local minimum at $x=(0,0)$. Applying the min-max principle with 
a quasimode $\chi(x)e^{\frac{f(x)}{h}}$ where $\chi$ is a smooth non negative
function such that $\chi\equiv 1$ on
$\{|x|\leq \frac{1}{4}\}$
and $\chi\equiv 0$ on $\{|x|\geq \frac{1}{2}\}$ implies that
$m_{0}^{D}(\Omega_{-})\geq 1$\,. 

Let us now consider  $\omega\in
D(\Delta_{-f,h}^{D,(0)}(\Omega_{-}))$ a normalized eigenvector
associated with an exponentially small eigenvalue: $\langle \omega\,,\,
\Delta_{-f,h}^{D,(0)}(\Omega_{-})\omega\rangle_{L^{2}(\Omega_-)}\leq e^{-\frac{c}{h}}$ for
some $c>0$.
%
%
Let $\chi_{1}^{2}+\chi_{2}^{2}=1$ be a partition of unity on $\Omega_-$,
with $\chi_1^2\equiv 1$ on
$\{|x|\leq \varepsilon \}$
and $\chi_1^2\equiv 0$ on $\{|x|\geq 2 \varepsilon\}$ (for
$\varepsilon < 1/4$). The IMS localization formula gives:
\begin{align}
\langle \omega\,,\,
\Delta_{-f,h}^{D,(0)}(\Omega_{-})\omega\rangle _{L^{2}(\Omega_{-})}
&=
\langle \chi_{1}\omega\,,\,
\Delta_{-f,h}^{D,(0)}(\Omega_{-})\chi_{1}\omega\rangle _{L^{2}(\Omega_{-})}
+
\langle \chi_{2}\omega\,,\,
\Delta_{-f,h}^{D,(0)}(\Omega_{-})\chi_{2}\omega\rangle
_{L^{2}(\Omega_{-})} \nonumber \\
& \quad
-h^{2}\sum_{j=1}^{2}\|\omega\nabla \chi_{j}\|_{L^{2}(\Omega_{-})}^{2}\,.\label{eq:IMS}
\end{align}
The second term of the right-hand side is equal to
$
\langle \chi_{2}\omega\,,\,
\Delta_{-f,h}^{D,(0)}(\Omega)\chi_{2}\omega \rangle_{L^2(\Omega_-
  \setminus \Omega_\varepsilon)}
$
with $\Omega_\varepsilon=\left\{x\in \Omega_{-}\,,\, |x| \le \varepsilon\right\}$\,.
Our choice of the function
$f(x)=\varphi_{in}(x)+\varphi_{ext}(\frac{x}{2})$ ensures that for
$h\in (0,h_{0})$ with $h_{0}$ small enough $|\nabla f|^{2}+h \Delta f$
is non negative on $\Omega_-\setminus \Omega_\varepsilon$. The second part of Lemma~\ref{le.apriori} thus
implies that there exists a function $\nu$ of $h$ such that
$$
\langle \chi_{2}\omega\,,\,
\Delta_{-f,h}^{D,(0)}(\Omega)\chi_{2}\omega\rangle_{L^{2}(\Omega_-
  \setminus \Omega_\varepsilon)}
\geq \nu(h)\|\chi_{2}\omega\|_{L^{2}(\Omega_- \setminus \Omega_\varepsilon)}^{2}\,,
$$
with $\liminf_{h\to 0}h\log \nu(h)=0$\,. In addition, exponential
decay estimates based on the Agmon identity implies that
$\sum_{j=1}^{2}\|\omega\nabla
\chi_{j}\|_{L^{2}(\Omega_{-})}^{2}={\mathcal O}(e^{-\frac{c}{h}})$
since $|\nabla f| > 0$ on $\supp(\chi_1) \cup \supp(\chi_2)$ (this is
obtained by adapting the arguments of Proposition~\ref{pr.decayO-},
for example). By using the IMS localization formula~\eqref{eq:IMS}, we
thus obtain that $\|\chi_{2}\omega\|_{L^{2}(\Omega \setminus \Omega_\varepsilon)}$ goes to zero
when $h$ goes to zero, and thus that $\lim_{h \to
  0}\|\chi_{1}\omega\|_{L^{2}(\Omega_-)}= \lim_{h \to
  0}\|\omega\|_{L^{2}(\Omega_\varepsilon)}=1$. Using then the same
argument as in the end of the proof of the
first part of Lemma~\ref{le.apriori}, we obtain that, for sufficiently
small $\varepsilon$, $\lim_{h \to
  0}\|\omega -C_h e^{\frac{f}{h}} \|_{L^{2}(\Omega_\varepsilon)}=0$, for some
constant $C_h \in \R$. The two limits $\lim_{h \to
  0}\|\omega\|_{L^{2}(\Omega_\varepsilon)}=1$ and $\lim_{h \to
  0}\|\omega -C_h e^{\frac{f}{h}} \|_{L^{2}(\Omega_\varepsilon)}=0$ imply that, in the asymptotic $h \to 0$, $\omega$
cannot be orthogonal to $\chi e^{\frac{f}{h}}$ (recall that $\chi
\equiv 1$ on $\Omega_\varepsilon$), which is in the
spectral subspace associated with exponentially small eigenvalues. This
concludes the proof.
\end{proof}

It is not difficult to adapt the previous argument to the case when
the function $f$ has several local maxima. Set $(x_{0},r_0)=(0,1)$  and consider a finite number of
points and radii $(x_{k},r_{k})_{1\leq k\leq N}$
 such that the open discs $D(x_{k},r_{k})$\,, $k=0,\ldots N$\,, are all
 disjoints
and included in $D(0,2)$\,. Let us consider the function
$
f(x)=\varphi_{ext}\left(\frac{x}{2}\right)
+\sum_{k=0}^{N}\varphi_{in}\left(\frac{x-x_{k}}{r_{k}}\right)
$ (see Figure~\ref{fig:2D_degenere_bis}).
Then the Hypotheses~\ref{hyp.0},~\ref{hyp.3},~\ref{hyp.1} and~\ref{hyp.2} hold with
$
m_{0}^{N}(\Omega_{-})=1$, $m_{1}^{N}(\Omega_{-})=N+1$ and
$m_{1}^{D}(\Omega_{+}\setminus\overline{\Omega_{-}})=1$.

\begin{remarque}
Interestingly, one can extend the last example to build a function $f$
for which the Hypothesis~\ref{hyp.1} is {\em not} satisfied. Consider
an infinite sequence $(x_{k},r_{k})_{k\in \nz}$\,, with 
$x_{0}=0$ and $r_{0}=1$\,, and 
 such that the open discs $D(x_{k},r_{k})$\,, $k \ge 0$\,, are all
 disjoints
and included in $D(0,2)$. Take the function
$
f(x)=\varphi_{ext}\left(\frac{x}{2}\right)+\sum_{k=0}^{\infty}\frac{r_{k}^{k}}{(1+k^{2})}\varphi_{in}\left(\frac{x-x_{k}}{r_{k}}\right)
$
in the domain $\Omega_-=D((-R,0), 2R-1)$ with $R>3$ large enough.
By Lemma~\ref{le.apriori} we know $m_{0}^{N}(\Omega_-)=1$\,, while
quasimodes associated with every $x_{k}$ show that the number of
eigenvalues of $\Delta_{f,h}^{N,(2)}(\Omega_-)$ (or equivalently
$\Delta_{-f,h}^{D,(0)}(\Omega_-)$) lying in $[0,e^{-\frac{\delta}{h}}]$
is larger than any fixed $n \in \nz$ for $h$ sufficiently small. Using
as in the proof of Proposition~\ref{pr.exampdeg2D} the identity $m_{2}^{N}(\Omega_{-})-m_{1}^{N}(\Omega_{-})+m_{0}^{N}(\Omega_{-})=1$, the number
of eigenvalues of $\Delta_{f,h}^{N,(1)}(\Omega_-)$ lying in
$[0,e^{-\frac{\delta}{h}}]$ 
is thus also larger than any $n \in \nz$ for $h$ sufficiently small. 
The Hypothesis~\ref{hyp.1} is not satisfied. 

Actually, there are
up to now no satisfactory necessary and sufficient
  conditions which guarantees that Witten Laplacians with general
  $\mathcal{C}^{\infty}$ potentials have a finite number of
  exponentially small eigenvalues.
\end{remarque}



\begin{figure}[htbp]
\centerline{\includegraphics[width=0.5\textwidth]{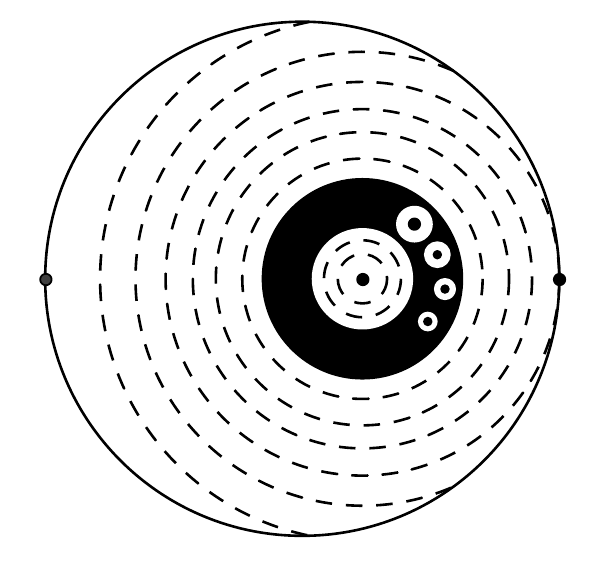}}
\caption{A variant of Figure~\ref{fig:2D_degenere} with $N=4$. The
  supports of the additional terms in $f$ (compared with Figure~\ref{fig:2D_degenere}) are represented by the white disks. }\label{fig:2D_degenere_bis}
\end{figure}

\appendix
\section{Riemannian geometry formulas}
\label{se.formulas}
For the sake of completeness and in order to help the reader not
so familiar with those tools, here is a list of formulas of
Riemannian geometry which were used in this text. 
We refer the reader for example to
\cite{AbMa,CFKS,GHL,Ste,Gol98} for
introductory texts to differential and Riemannian geometry.
 We also consider here only real-valued
differential forms (the extension to complex-valued
differential forms is easy).

Let $(M,g)$ be a $d$-dimensional Riemannian manifold.
The tangent (resp. cotangent) bundle is denoted
by $TM$ (resp. $T^{*}M$) and its fiber over $x\in M$, $T_{x}M$
(resp. $T_{x}^{*}M$)\,.
The exterior algebra over $T_{x}^{*}M$ is $\bigwedge
T_{x}^{*}M=\bigoplus_{p=0}^{d}\bigwedge^{p}T_{x}^{*}M$
endowed with the exterior product $\wedge$ and the associated
 fiber bundle is denoted by $\bigwedge T^{*}M=\bigoplus
 \bigwedge^{p}T^{*}M$\,.  The exterior product of $p$ elements
 $(\varphi_i)_{1 \le i \le p}$ of
 $T_{x}^{*}M$ is defined by:
$$\varphi_1 \wedge \ldots \wedge \varphi_p=\sum_{\sigma \in
  \mathfrak{S}_{\{1,\ldots,p\}}} \epsilon _{\{1,\ldots,p\}} (\sigma) \,
\varphi_{\sigma(1)}\otimes\cdots\otimes \varphi_{\sigma(p)} \,
,$$
where $\epsilon_{E}(\pi)$
is the signature of the permutation $\pi\in \mathfrak{S}_{E}$\,.
Differential forms are sections of this fiber bundle and their
regularity is encoded by the notation: $\bigwedge
\mathcal{C}^{\infty}(M)$ is the set of
$\mathcal{C}^{\infty}$-differential forms, $\bigwedge L^{2}(M)$ is the
set of $L^{2}$-differential forms, and so on. This notation was
used in the present text for the sake of  conciseness. A more standard and general notation would be
$\mathcal{C}^{\infty}(M;\bigwedge T^{*}M)$ as for
$\mathcal{C}^{\infty}(M;E)$, the set of $\mathcal{C}^{\infty}$
sections of
the differential fibre bundle $(E,\Pi)$ on $M$ with $\Pi: E\to M$ (a
section $x\to s(x)$ satisfies $\Pi(s(x))=x$).

In a local coordinate system $(x^{1},\ldots,x^{d})$ a basis of
$\bigwedge^{p} T_{x}^{*}M$ is formed by the elements
$$
dx^{I}=dx^{i_{1}}\wedge\cdots\wedge dx^{i_{p}} \quad,\quad
  I= \{i_{1},\ldots,i_{p}\}\quad,\quad i_{1}< \ldots <i_{p}\, .
$$
Here and in the following, $I=\{i_1, \ldots ,i_p\}$ denotes a subset of $\{1, \ldots, d\}$
with $\#I=p$ elements, which can be identified equivalently to an ordered
$p$-uplet $(i_1,\ldots,i_p)$ with $i_1 < \ldots < i_p$.

A differential form $\omega\in \bigwedge^{p}T^{*}M$ is written
$$
\omega=\sum_{\sharp I=p}\omega_{I}(x)~dx^{I}\,,
$$
and its differential is given by 
$$
d\omega=\sum_{\sharp I =p}\partial_{x^{i}}\omega_{I}(x)dx^{i}\wedge
dx^{I}\,.
$$
Remember that the exterior product is bilinear  associative and
antisymmetric:
$$
\omega_{1}\wedge \omega_{2}=(-1)^{p_{1}p_{2}}\omega_{2}\wedge
\omega_{1}\,,\quad \omega_{i}\in \bigwedge^{p_{i}}T^{*}_{x}M\,.
$$
The differential and the $\wedge$ product satisfy $d\circ d=0$ and
$$ d(\omega_{1}\wedge \omega_{2})=d\omega_{1}\wedge
\omega_{2}+(-1)^{p_{1}}\omega_{1}\wedge d\omega_{2}\,,\quad
\omega_{i}\in \bigwedge^{p_{i}}\mathcal{C}^{\infty}(M)\,.
$$A $\mathcal{C}^{\infty}$-vector field $X$ on $M$ is a
$\mathcal{C}^{\infty}$ section of $TM$, $X\in
\mathcal{C}^{\infty}(M;TM)$\,. The interior product $\mathbf{i}_{X}$
is  the local operation defined for $X_{x}\in T_{x}M$ and $\omega_{x}\in
\bigwedge^{p}T_{x}^{*}M$ by
\begin{equation}\label{eq:interior}
\mathbf{i}_{X_{x}}\omega_{x} (T_{2},\ldots, T_{p})=
\omega_{x}(X_{x},T_{2},\ldots, T_{p})\,,\quad \forall
T_{2},\ldots,T_{p}\in T_{x}M\,.
\end{equation}
For $X\in \mathcal{C}^{\infty}(M;TM)$ and $\omega_{i}\in
\bigwedge^{p_{i}}\mathcal{C}^{\infty}(M)$, one has
$$
\mathbf{i}_{X}(\omega_{1}\wedge \omega_{2})=
(\mathbf{i}_{X}\omega_{1})\wedge \omega_{2}
+(-1)^{p_{1}}\omega_{1}\wedge (\mathbf{i}_{X}\omega_{2})\,.
$$
When $\Phi:M\to N$ is a $\mathcal{C}^{\infty}$ map $\Phi_{*}$
denotes the functorial push-forward and $\Phi^{*}$ the functorial
pull-back. 
For a $\mathcal{C}^{\infty}$-map $\Phi$, two forms $\omega_{1}$,
$\omega_{2}$\,, one has
$$
\Phi^{*}(d\omega_{1})= d\left(\Phi^{*}\omega_{1}\right)\,,\quad
\Phi^{*}(\omega_{1}\wedge \omega_{2})=(\Phi^{*}\omega_{1})\wedge(\Phi^{*}\omega_{2})\,.
$$
When $\Phi$ is a diffeomorphism, $\omega$ a $p$-form and $X$ a vector
field
$$
\Phi^{*}\mathbf{i}_{X}\omega=\mathbf{i}_{\Phi^{*}X}\Phi^{*}\omega\,.
$$
When $\Phi$ is a diffeomorphism given by the exponential map of a
vector field $X$, this allows to define the Lie derivative
\begin{equation}\label{eq:Lie}
\mathcal{L}_{X}\omega=\frac{d}{dt}(e^{tX})^{*}\omega\Big|_{t=0}\quad\text{for}~\omega\in\bigwedge \mathcal{C}^{\infty}(M)\,.
\end{equation}
The Lie derivative satisfy
$$
\mathcal{L}_{X}(\omega_{1}\wedge
\omega_{2})=(\mathcal{L}_{X}\omega_{1})\wedge
\omega_{2}+\omega_{1}\wedge(\mathcal{L}_{X}\omega_{2})\,,
$$
and the magic Cartan's formula says
$$
\mathcal{L}_{X}=\mathbf{i}_{X}\circ d + d\circ \mathbf{i}_{X}\,.
$$
Differential forms $\omega$ with degree $p$ can be integrated along a $p+1$-chain,
or briefly a $p+1$-dimensional submanifold with boundary; let us write it $C$
with the boundary $\partial C$\,. Stoke's formula is written
$$
\int_{C}d\omega=\int_{\partial C}\omega\,,
$$
and it is the ground for de~Rham's cohomology.\\
The Riemannian structure adds the pointwise dependent scalar product
$g(x)$ given by
$$
\langle S,T\rangle_{T_{x}M}=\sum_{1\leq i,j\leq d}g_{i,j}(x)S^{i}T^{j}
$$
and with a dual metric
$(g^{i,j}(x))_{1\leq i,j\leq d}:=g({x})^{-1}$ defined on
$T^{*}_{x}M$\,. This is also written with Einstein's conventions
$$
g=g_{i,j}dx^{i}dx^{j}\quad, \quad g_{i,j}g^{j,k}=\delta_{i}^{k}\,.
$$
Both $g(x)$ and $g(x)^{-1}$ are extended by tensor product to $\bigwedge T_{x}M$
and $\bigwedge T_{x}^{*}M$: for $\omega,\omega'\in
\bigwedge^{p} \mathcal{C}^{\infty}(M)$,
$$\langle \omega'\,,\,
\omega\rangle_{\bigwedge^{p}T_{x}^{*}M}=\sum_{\#I=p}\sum_{\#J=p} \left(\prod_{k=1}^p
g^{i_k,j_k}\right) \omega'_I \omega_J,
$$
where $I=\{i_1,\ldots,i_p\}$ ($i_1< \ldots < i_p$) and $J=\{j_1,\ldots,j_p\}$ ($j_1< \ldots < j_p$).

%
%

 The Riemannian infinitesimal volume
(denoted simply by $dx$ in the text) is
in an oriented local coordinate system:
$$
d {\rm Vol}_{g}(x)=(\det g)^{\frac{1}{2}}~dx^{1}\wedge\ldots\wedge
dx^{d}=
(\det g)^{\frac{1}{2}}~dx^{1}\ldots dx^{d}\,.
$$
Those scalar products as non degenerate bilinear forms,
 allow identifications between forms
and vectors. Here are examples: when $\omega=\omega_{i}(x)dx^{i}$ is a one
form, the vector $\omega^{\sharp}$ is given by
$(\omega^{\sharp})^{i}=g^{i,j}\omega_{j}$; when
$X=X^{i}\partial_{x^{i}}$ is a vector field $X^{\flat}$ is the one
form defined by $(X^{\flat})_{i}=g_{i,j}X^{j}$\,. As an application the
gradient for a function is nothing but $\nabla
f=(df)^{\sharp}$\,. Similarly the Hessian of a function $f$ at a point
$x$, initially defined as a
bilinear form,  can be viewed a linear map of $T_{x}M$\,.

Another duality between forms of complementary degrees
$p+p'=d=\dim M$\,, is provided by Hodge~$\star$ operator.
When the Riemannian manifold $(M,g)$ is orientable (locally it is
always the case), the operator
$\star:\bigwedge^{p}\mathcal{C}^{\infty}(M)\to
\bigwedge^{d-p}\mathcal{C}^{\infty}(M)$\,, is defined by
$$
\int \langle \omega'\,,\,
\omega\rangle_{\bigwedge^{p}T_{x}^{*}M}~d{\rm Vol}_{g}(x)
=\int \omega'\wedge (\star \omega)\,,\quad \omega,\omega'\in
\bigwedge^{p} \mathcal{C}^{\infty}(M)\,.
$$
In a coordinate system it is given by
$$
(\star\omega)_{J}=\sum_{I}\delta_{I\cup J}^{\left\{1,\ldots,d\right\}}
\epsilon_{\left\{1,\ldots,d\right\}}(I,J)(\det g)^{1/2}(\omega^{\#})^{I}\quad,\quad
\left\{
\begin{array}[c]{l}
  I=\{i_{1},\ldots,i_{p}\}, \, i_1< \ldots < i_p\,,\\
  J=\{j_{1},\ldots,j_{d-p}\}, \, j_1< \ldots < j_{d-p}\,,\\
  (I,J)=(i_{1},\ldots,i_{p},j_{1},\ldots,j_{d-p})\,,
\end{array}
\right.
$$
where $\delta_{A}^{B}=1$ when $A=B$ and $0$ otherwise.
We have the additional properties for $\omega,\omega'\in
\bigwedge^{p}\mathcal{C}^{\infty}(M)$:
\begin{align*}
& \star(\lambda\omega+\omega')=\lambda\star\omega+\star\omega'\,,\quad
  \lambda\in \mathcal{C}^{\infty}(M)\,,\\
&\star\star\omega=(-1)^{p(d+1)}\omega\,,\\
&
\omega\wedge(\star \omega')= \omega'\wedge(\star \omega)\,,\\
& \star 1=d{\rm Vol}_{g}(x)\quad\text{(assuming $M$ oriented)}.
\end{align*}
The codifferential $d^{*}$ is defined as the formal adjoint  of the
differential $d:\bigwedge \mathcal{C}^{\infty}(M)\to
\bigwedge \mathcal{C}^{\infty}(M)$:
$$
\langle d\omega\,,\,\omega'\rangle=\langle \omega\,,\,
d^{*}\omega'\rangle \, .
$$
 With Hodge~$\star$ operator
(do the identification on a compact oriented manifold without boundary
with $\int_{M}d\eta =0$), it
is written
$$
\forall \omega\in
\bigwedge^{p}\mathcal{C}^{\infty}(M)\,,\quad
\left\{
  \begin{array}[c]{l}
 \star d^{*}\omega =(-1)^{p}d\star\omega\,,\\
\star d\omega=(-1)^{p+1}d^{*}\star\omega\,,\\
    d^{*}\omega=(-1)^{pd+d+1}\star d \star \omega\,.
  \end{array}
\right.
$$
The Hodge Laplacian is then given by
\begin{equation}\label{eq:Hodge}
\Delta_{\textrm{H}}=(d+d^{*})^{2}=dd^{*}+d^{*}d\,.
\end{equation}
It is possible to write $d^{*}$ and $\Delta_{\textrm{H}}$ in a coordinate
system. For example 
\begin{align*}
  &(d^{*}\omega)_{I}=-g^{i,j}\delta^{J}_{i\cup
    I}\epsilon_{J}(i,I)\nabla_{j}\omega_{J}\,,
  \quad(i,I)=(i,i_{1},\ldots,i_{p-1})\\
&
\nabla_{j}\omega_{J}=\partial_{x^{j}}\omega_{J}-\sum_{\ell=1}^{p}\omega_{I\cup\left\{k\right\}\setminus
i_{\ell}}
\epsilon_{I\cup\left\{k\right\}\setminus
i_{\ell}}(i_{1}\dots i_{\ell-1},k,i_{\ell+1},\dots i_{p})\Gamma^{~k}_{i_{\ell}~j}\\
& \Gamma^{k}_{i_{\ell},j} = \frac{1}{2}g^{k,
  m}\left(\partial_{x^{i_{\ell}}}g_{j,m}+\partial_{x^{j}}g_{m,
    i_{\ell}}-\partial_{x^{m}}g_{i_{\ell}, j}\right)\,,
\end{align*}
where one recognizes the covariant derivative $\nabla_{j}$ associated
with the metric $g$ (Levi-Civita connection) and the Christoffel
symbols $\Gamma^{j}_{k,\ell}$\,. The writing of $\Delta_{\textrm{H}}$ involves
the Riemann curvature tensor and is known as Weitzenbock's formula.
We wrote the above example to convince the non familiar reader that
the explicit writing in a coordinate system is not always more informative
than the intrinsic formula.

Here is the example of the Witten Laplacian, $\Delta_{f,h}=(d_{f,h}+d_{f,h}^{*})^{2}=d_{f,h}^{*}d_{f,h}+d_{f,h}d_{f,h}^{*}$: 
\begin{align}
   d_{f,h}&=e^{-\frac{f}{h}}(hd)e^{\frac{f}{h}}=hd +df\wedge \,, \label{eq:dfh_d} \\
 d_{f,h}^{*}&=e^{\frac{f}{h}}(hd^{*})e^{-\frac{f}{h}}=hd^{*}
+\mathbf{i}_{\nabla f} \, , \label{eq:dfh*_d*} \\
 \Delta_{f,h} &= d_{f,h}d_{f,h}^{*}+d_{f,h}^{*}d_{f,h}=
(hd+df\wedge)(hd^{*}+\mathbf{i}_{\nabla f})
+(hd^{*}+\mathbf{i}_{\nabla f}) (hd+df\wedge) \nonumber \\
& = h^{2}(dd^{*}+d^{*}d)+ [(df\wedge)\circ \mathbf{i}_{\nabla f}
+\mathbf{i}_{\nabla f}\circ(df\wedge)]
+ h[d\mathbf{i}_{\nabla f}+\mathbf{i}_{\nabla f}d]
+ h[(df\wedge)\circ d^{*}+d^{*}\circ(df\wedge)]
\nonumber \\
& = h^{2}\Delta_{\textrm{H}}+|\nabla f|^{2}+h(\mathcal{L}_{\nabla f}+
\mathcal{L}_{\nabla f}^{*})\,, \label{eq:hodge_witten}
\end{align}
where we used $\mathbf{i}_{X}(df\wedge \omega)=df(X)\omega -
df\wedge(\mathbf{i}_{X}\omega)$ with $X=\nabla f$\,, Cartan's magic
formula and an easy identification of $\mathcal{L}_{\nabla
  f}^{*}$\,. No explicit computation of $d^{*}$ or the Hodge Laplacian
is necessary to understand the structure of the Witten Laplacian. In
particular $\mathcal{L}_{X}+\mathcal{L}_{X}^{*}$ is clearly
 a zeroth order
differential operator because in a coordinate system the formal adjoint of
$a^{j}(x)\partial_{x^{j}}$ in $L^{2}(\rz^{d},\varrho(x)dx)$ equals
$-a^{j}(x)\partial_{x^{j}}+b[a,\varrho](x)$ where 
$b[a,\varrho]$ is the multiplication by a function of $x$. The
operator $\mathcal{L}_{\nabla f}+\mathcal{L}_{\nabla f}^{*}$ is not
the local action of a tensor field on $M$\,, because it does not
follow the change of coordinates rule for tensors.
Actually, one can give a meaning to the general
writing 
$$
\Delta_{f,h}^{(p)}=h^{2}\Delta_{\textrm{H}}^{(p)}+|\nabla
f|^{2}-h(\Delta f) + 2 h
\left(\Hess f\right)_{p}\,,
$$
where $\left(\Hess f\right)_{p}$ is an element of the curvature tensor
algebra (see~\cite{Jam} and references therein).

Let us conclude this appendix with integration by parts formulas in
the case of a manifold with a boundary. All these formulas rely first on Stokes
formula $\int_{\Omega}d\omega=\int_{\partial \Omega}\omega$ when
$\omega\in
\bigwedge^{d-1}\mathcal{C}^{\infty}(\overline{\Omega})$\,. Note that
the right-hand side of Stokes formula would equivalently (and more
explicitly written) $\int_{\partial \Omega}\omega=\int_{\partial
  \Omega}j^{*}\omega$ where $j:\partial \Omega\to \overline{\Omega}$
is the natural embedding map (a trace along $\partial \Omega$ is taken
and pointwise $j^{*}\omega_{x}$ is evaluated only on $(d-1)$ vector
{\em tangent} to $\partial \Omega$).
Another writing taken initially from \cite{Schz} is also convenient. 
For $\sigma\in \partial \Omega$ let $n(\sigma)$ be the outward normal
vector and write for any element $X\in T_{\sigma}M$
$X=X_{T}+X_{n} n$\,.\\
For $\omega\in \bigwedge^{p}\mathcal{C}^{\infty}(\overline{\Omega})$
defined $\mathbf{t}\omega$, and $\mathbf{n}{\omega}=\omega-\mathbf{t}\omega$ by
$$
\forall \sigma\in \partial \Omega\,,\, \forall X_{1},\ldots, X_{p}\in
T_{\sigma}\Omega\,,\quad
\mathbf{t}\omega(X_{1},\ldots,
X_{p})=\omega(X_{1,T},\ldots, X_{p,T})\,.
$$
If $(x^{1},\ldots,x^{d})=(x',x^{d})$ is a coordinate system in a
neighborhood $\mathcal{V}$ of
$\sigma_{0}\in \partial \Omega$ such that $\Omega\cap \mathcal{V}$ is
given locally by $\left\{x^{d}<0\right\}$\,, $\partial \Omega \cap
\mathcal{V}$ by $\left\{x^{d}=0\right\}$ and
$n=\partial_{x^{d}}$\,, then a $p$-form can be written
$$
\omega=\sum_{\scriptsize
  \begin{array}[c]{c}
 \sharp I=p\\ 
 d\not\in I
  \end{array}}
\omega_{I}dx^{I}+ 
\sum_{\scriptsize \begin{array}[c]{c}
    \sharp I'=p-1\\ d\not\in I'
  \end{array}}
\omega_{I'}dx^{I'}\wedge dx^{d}\,,
$$
and the operators $\mathbf{t}$ and $\mathbf{n}$ acts as
$$
\mathbf{t}\omega= \sum_{\scriptsize
  \begin{array}[c]{c}
 \sharp I=p\\ 
 d\not\in I
  \end{array}}\omega_{I}dx^{I}
\quad,\quad
\mathbf{n}\omega=  \sum_{\scriptsize \begin{array}[c]{c}
    \sharp I'=p-1\\ d\not\in I'
  \end{array}}\omega_{I'}dx^{I'}\wedge dx^{d}\,.
$$
Stokes formula can be written now $\int_{\Omega}d\omega=\int_{\partial
\Omega}\mathbf{t}\omega$ for $\omega\in
\bigwedge^{d-1}\mathcal{C}^{\infty}(\overline{\Omega})$\,, but
contrary to the operator $j^{*}$ the operator $\mathbf{t}$ makes sens
in a collar neighborhood of $\partial \Omega$\,. In particular the
following formula 
$$
\mathbf{t}d\omega=d\mathbf{t}\omega
$$
makes sense for any $\omega \in \bigwedge
\mathcal{C}^{\infty}(\overline{\Omega})$ and it is rather easy to
check with the above coordinates description.
One also gets in the same way
\begin{align}
\label{eq.tn1}
  & \mathbf{t}\omega=
  \mathbf{i}_{n}(n^{\flat}\wedge
  \omega)\quad\text{for }\omega\in \bigwedge\mathcal{C}^{\infty}(\overline{\Omega})\,,
\\
\label{eq.tn2}
&\star\mathbf{n}=\mathbf{t}\star\,,\quad
\star\mathbf{t}=\mathbf{n}\star\,,\\
\label{eq.tn3}
&
\mathbf{t}d=d\mathbf{t}\,,\quad \mathbf{n}d^{*}=d^{*}\mathbf{n}\,,\\
\label{eq.tn4}&
\mathbf{t}\omega_{1}\wedge \star \mathbf{n}\omega_{2}= \langle
\omega_{1}\,,\,
\mathbf{i}_{n}\omega_{2}\rangle_{\bigwedge^{p}T^{*}_{\sigma}\Omega}\times 
d{\rm Vol}_{g,\partial\Omega}\quad\text{for } \omega_{i}\in \bigwedge^{p}\mathcal{C}^{\infty}(\overline{\Omega})\,,
\end{align}
where we recall that $d{\rm Vol}_{g,\partial \Omega}(X_{1},\ldots,
X_{d-1})=d{\rm Vol}_{g}(n, X_{1},\ldots, X_{d-1})$\,.

The above formulas for example lead to the following integration by parts for
$\omega_{1},\omega_{2}\in
\bigwedge^{p}\mathcal{C}^{\infty}(\overline{\Omega})$~:
\begin{align*}
\langle d_{f,h}\omega_{1}\,,\, d_{f,h}\omega_{2}\rangle_{L^2(\Omega)}+ \langle
d_{f,h}^{*}\omega_{1}\,,\, d_{f,h}^{*}\omega_{2}\rangle_{L^2(\Omega)}
&=
\langle \omega_{1}\,,\,  \Delta_{f,h}\omega_{2}\rangle_{L^2(\Omega)}\\
&\quad
+ h\int_{\partial \Omega}(\mathbf{t}\omega_{2})\wedge{\star
  \mathbf{n}d_{f,h}\omega_{1}}
- h\int_{\partial \Omega}(\mathbf{t}d_{f,h}^{*}\omega_{1})\wedge(\star
\mathbf{n}\omega_{2})\,.
\end{align*}
This shows for example that $\Delta_{f,h}^{D}$
(resp. $\Delta_{f,h}^{N}$) with its form domain $W^{1,2}_{D}=\left\{\omega\in
  \bigwedge W^{1,2}\,,
 \mathbf{t}\omega=0\right\}$ (resp. $W^{1,2}_{N}=\left\{\omega\in
  \bigwedge W^{1,2}\,,
 \mathbf{n}\omega=0\right\}$) is associated with the Dirichlet
form $\|d_{f,h}\omega\|^{2}+\|d_{f,h}^{*}\omega\|^{2}$\,. Interpreting
the weak formulation of $\Delta_{f,h}\omega= f$ leads to the operator
domains $D(\Delta_{f,h}^{D})$ and $D(\Delta_{f,h}^{N})$ (we refer the
reader to \cite{HeNi} for details).  The boundary terms
of Lemma~\ref{le.Agmon}  are obtained in a very similar way.

\medskip
\noindent\textbf{Acknowledgements:} The first author would like to
thank C. Le Bris, D. Perez and A. Voter for useful discussions. This work was essentially
completed while the second author had a ``D{\'e}l{\'e}gation INRIA'' in CERMICS
at
Ecole des Ponts. He acknowledges the support of INRIA and
thanks the people of CERMICS for their hospitality.

\end{document}